\documentclass[11pt,twoside]{article}
\usepackage[margin=1.3in]{geometry}

\usepackage{latexsym,amsfonts,amsmath,amsthm,amssymb,makeidx}
\usepackage[title]{appendix}
\usepackage{CJK,CJKnumb,CJKulem,times,dsfont,ifthen,mathrsfs,latexsym,amsfonts, color}
\usepackage{amsmath,amsthm,makeidx,fontenc,amssymb,bm,graphicx,psfrag,listings, curves,extarrows}
\usepackage{cite}
\usepackage[colorlinks,citecolor=black,linkcolor=black]{hyperref}
\usepackage{cleveref}

\let\oldbibliography\thebibliography
\renewcommand{\thebibliography}[1]{%
\oldbibliography{#1}%
\setlength{\itemsep}{0pt}%
}

\makeindex
\newtheorem{theorem}{Theorem}[section]
\newtheorem{lemma}[theorem]{Lemma} 
\newtheorem{proposition}[theorem]{Proposition}
\newtheorem{definition}[theorem]{Definition}
\newtheorem{example}[theorem]{Example}
\newtheorem{corollary}[theorem]{Corollary}
\newtheorem{remark}[theorem]{Remark}

\newcommand{\R}{\mathbb R}

\newcommand{\bt}{\begin{theorem}}
\newcommand{\et}{\end{theorem}}
\newcommand{\bl}{\begin{lemma}}
\newcommand{\el}{\end{lemma}}
\newcommand{\bd}{\begin{definition}}
\newcommand{\ed}{\end{definition}}
\newcommand{\bc}{\begin{corollary}}
\newcommand{\ec}{\end{corollary}}
\newcommand{\bp}{\begin{proof}}
\newcommand{\ep}{\end{proof}}
\newcommand{\bx}{\begin{example}}
\newcommand{\ex}{\end{example}}
\newcommand{\bi}{\begin{exercise}}
\newcommand{\ei}{\end{exercise}}
\newcommand{\bo}{\begin{prop}}
\newcommand{\eo}{\end{prop}}
\newcommand{\br}{\begin{remark}}
\newcommand{\er}{\end{remark}}
\newcommand{\be}{\begin{equation}}
\newcommand{\ee}{\end{equation}}
\newcommand{\ba}{\begin{align}}
\newcommand{\ea}{\end{align}}
\newcommand{\bn}{\begin{enumerate}}
\newcommand{\en}{\end{enumerate}}
\newcommand{\bg}{\begin{align*}}
\newcommand{\bcs}{\begin{cases}}
\newcommand{\ecs}{\end{cases}}

\newcommand{\bean}{\begin{eqnarray*}}
\newcommand{\eean}{\end{eqnarray*}}

\numberwithin{equation}{section}

\begin{document}
\title{ {\bf Local estimates for conformal $Q$-curvature equations} }
\date{July 9,   2021} 
\author{\\ {Tianling Jin \footnote{T. Jin was partially supported by Hong Kong RGC grant GRF 16302217.}
}
 \;\;  and \;\;   Hui Yang 
 }
\maketitle

\begin{center}
\begin{minipage}{130mm}
\begin{center}{\bf Abstract}\end{center}
We derive local estimates of positive solutions to the conformal $Q$-curvature equation 
$$
(-\Delta)^m u = K(x) u^{\frac{n+2m}{n-2m}} ~~~~~~  \textmd{in} ~ \Omega \backslash \Lambda 
$$
near their singular set $\Lambda$, where $\Omega \subset \mathbb{R}^n$ is an open set, $K(x)$  is a positive continuous function on $\Omega$, $\Lambda$ is a closed subset of $\mathbb{R}^n$, $2 \leq m  < n/2$ and $m$ is an integer.  Under certain flatness conditions at critical points of $K$ on $\Lambda$,  we prove that $u(x) \leq C  [\textmd{dist}(x,  \Lambda)]^{-(n-2m)/2}$ when the upper Minkowski dimension of $\Lambda$ is less than $(n-2m)/2$.    

\vskip0.10in


\vskip0.10in


\end{minipage}

\end{center}

\vskip0.20in

\section{Introduction and main results} 

In this paper, we study the higher order conformal $Q$-curvature equation 
\begin{equation}\label{H-01} 
(-\Delta)^m u = K(x) u^{\frac{n+2m}{n-2m}}, ~~~ u>0  ~~~~~~  \textmd{in} ~ \Omega \backslash \Lambda
\end{equation} 
with a singular set $\Lambda$, where $\Omega \subset \mathbb{R}^n$ is an open set, $K(x)$  is a positive continuous function on $\Omega$,  $\Lambda$ is a closed subset of $\mathbb{R}^n$,  $1 \leq m  < n/2$ and $m$ is an integer.  Throughout the paper, {\it  $K(x)$ is assumed to be bounded between two positive constants in $\Omega$}.    

When $m=1$, Eq. \eqref{H-01} is the conformal scalar curvature equation which reads as 
\begin{equation}\label{H-01=00hh-01} 
-\Delta u = K(x) u^{\frac{n+2}{n-2}}, ~~~ u>0  ~~~~~~  \textmd{in} ~ \Omega \backslash \Lambda. 
\end{equation} 
This equation appears in the problem of finding a metric conformal to the flat metric $\delta_{ij}$ on $\mathbb{R}^n$ such that $K(x)$ is the scalar curvature of the new metric $u^{4/(n-2)}\delta_{ij}$.   The classical works of Schoen and Yau \cite{Sch84,Sch88,SY88} on the Yamabe problem and conformally flat manifolds have indicated the importance of studying the equation \eqref{H-01=00hh-01} with a singular set. In particular, an interesting question is to understand how $u(x)$ tends to infinity when $x$ approaches the singular set.  When $K(x) \equiv 1$ and $\Lambda=\{ 0 \}$ is an isolated singularity in $\Omega$, Caffarelli, Gidas and Spruck in the pioneering paper \cite{CGS} proved that every singular solution $u$ of \eqref{H-01=00hh-01} is asymptotically radially symmetric near $0$,  and further proved that $u$ is asymptotic to a radial singular solution of the same equation on $\mathbb{R}^n \backslash \{0\}$.  Later, Korevaar, Mazzeo, Pacard and Schoen in \cite{KMPS}  studied refined asymptotics and expanded such a singular solution $u$ to the first order.   Han, Li and Li \cite{HLL} recently  established the expansions up to arbitrary orders for such a singular solution $u$. Subsequent to \cite{CGS}, other second-order Yamabe type equations with isolated singularities related to \eqref{H-01=00hh-01} have also been studied; see, for example, \cite{CHanY, HLL,HLT,LC96,Li06,Mar,XZ} and the references therein.  When $K(x) \equiv 1$ and $\Lambda$ is a general singular set with Newtonian capacity zero, Chen and Lin \cite{Chen-Lin95} proved that any solution $u$ of \eqref{H-01=00hh-01} satisfies the following a priori estimate   
\begin{equation}\label{H-01=00hh-022} 
u(x) \leq C [\textmd{dist}(x, \Lambda)]^{-\frac{n-2}{2}}, 
\end{equation}
and showed  that $u$ is asymptotically symmetric near  $\Lambda$ based on this estimate.  In a series of masterful papers \cite{Chen-Lin97,Chen-Lin98,Chen-Lin99,Lin00},  Chen and Lin studied the equation \eqref{H-01=00hh-01}  in the case when $K(x)$ is a non-constant positive function. They first proved under some flatness conditions on $K(x)$ that every $C^2$ solution $u$ of \eqref{H-01=00hh-01} satisfies the a priori estimate \eqref{H-01=00hh-022} via the method of moving planes, and then applied the Pohozaev identity to describe the precise asymptotic behavior of solutions when the singularity set  $\Lambda$ is isolated.   By using the method of moving spheres,  Zhang in \cite{Zhang02} simplified and improved the argument of Chen-Lin \cite{Chen-Lin97} to derive the local estimate \eqref{H-01=00hh-022} for the equation \eqref{H-01=00hh-01} under some flatness assumptions on  $K(x)$.     In \cite{TZ},  Taliaferro and Zhang further gave conditions on $K(x)$ to characterize the precise behavior of solutions of \eqref{H-01=00hh-01} near an isolated singularity.

When $m=2$,   Eq. \eqref{H-01} is a fourth-order equation with critical Sobolev exponent. One of the motivations to study this  equation arises from the problem of finding a metric which is conformal to the flat metric on $\mathbb{R}^n$ such that $K(x)$ is the fourth-order $Q$ curvature  of the new metric $u^{4/(n-4)}\delta_{ij}$ (see \cite{Bran,Pan}).   For $m\geq 3$, the higher-order equation \eqref{H-01} also arises in the study of similar problem in conformal geometry (see, e.g., \cite{GJMS}).  We refer to  Gursky-Malchiodi \cite{GM} and  Hang-Yang \cite{HY16} for the recent progress of the fourth-order $Q$ curvature problem on compact Riemannian manifolds. 
 In \cite{CHY},   Chang, Hang and Yang have shown that the Hausdorff dimension of singular set $\textmd{dim}_H(\Lambda) < (n-4)/2$ is a necessary condition for the existence of a complete conformal metric whose scalar curvature and fourth-order $Q$ curvature both have a positive lower bound. For higher values $m \geq 3$ and a smooth $k$-dimensional submanifold $\Lambda$, Gonz\'{a}lez-Mazzeo-Sire \cite{GMS} showed  that $\Gamma(\frac{n}{4} - \frac{k}{2} +\frac{m}{2})/\Gamma(\frac{n}{4} - \frac{k}{2} -\frac{m}{2}) >0$ is necessary to have a complete conformal metric under some conditions, which holds in particular when $k < (n-2m)/2$.

The problem of characterizing asymptotic behavior of solutions of  \eqref{H-01} for  $m\geq 2$ is significantly more challenging due to the lack of maximum principle.   Recently, this problem has also  attracted much attention. When $K(x) \equiv 1$ and $\Lambda=\{ 0 \}$ is an isolated singularity in $\Omega$,  Jin and Xiong in \cite{JX19} proved sharp blow up rates and the asymptotic radial symmetry of singular solutions of \eqref{H-01} near the singularity $0$ under the sign assumptions 
\begin{equation}\label{H-01=00hh-033} 
(-\Delta)^s u \geq 0  ~~~~~~ \textmd{in} ~  \Omega\backslash \{0\} ~~~ \textmd{for} ~  \textmd{all} ~  s=1, \dots,  m-1.  
\end{equation}
Based on the a priori estimates of Jin and Xiong,  in the case of $K(x) \equiv 1$,  Andrade-do \'{O} \cite{Ad} and Ratzkin \cite{R20} further proved for $m = 2$ that singular solutions of \eqref{H-01} satisfying  \eqref{H-01=00hh-033}  are  asymptotic to positive singular solutions of $(-\Delta)^2 u= u^{(n+4)/(n-4)}$ on  $\mathbb{R}^n \backslash \{0\}$.  All these singular solutions  on  $\mathbb{R}^n \backslash \{0\}$ have been classified by Frank and K\"{o}nig in \cite{Fr-K19} using ODE analysis according to the radial symmetry result of Lin \cite{Lin98}.  When $K(x) \equiv 1$ and $\Lambda$ is a general singular set with the   upper Minkowski dimension being less than $(n-2m)/2$,  Du and Yang in \cite{DY01} established the following a priori estimate     
\begin{equation}\label{H-01=00hh-044} 
u(x) \leq C [\textmd{dist}(x, \Lambda)]^{-\frac{n-2m}{2}} 
\end{equation} 
for any solution $u$ of \eqref{H-01} satisfying \eqref{H-01=00hh-033} in $\Omega\backslash \Lambda$,  and proved further that $u$ is asymptotically symmetric near $\Lambda$. When $m=2$ and $\Omega=\R^n$, the estimate \eqref{H-01=00hh-044} was obtained in Chang, Han and Yang \cite{CHanY} under the assumption that the metric $u^{4/(n-4)}\delta_{ij}$ has nonnegative scalar curvature.

Our interest in this paper is the situation where the $Q$-curvature $K(x)$ is a non-constant positive function.  Under some conditions on the order of flatness at critical points of $K$ on $\Lambda$,  we will establish the local estimate \eqref{H-01=00hh-044} for any solution of \eqref{H-01} near its singular set $\Lambda$ when the upper Minkowski dimension of $\Lambda$ is  less than $(n-2m)/2$.

We first recall the definition of the Minkowski dimension (see, e.g., \cite{KLV13,Mattila}).   Suppose  $E\subset \mathbb{R}^n$ is a compact set, the $\lambda$-dimensional Minkowski $r$-content of $E$ is defined by 
$$
\mathcal{M}_r^\lambda(E) = \inf \left\{l r^\lambda ~ \big| ~  E \subset \bigcup_{k=1}^l B(x_k, r),  ~ x_k \in E  \right\}, 
$$ 
and the upper and lower Minkowski dimensions are defined, respectively, as 
$$
\overline{\textmd{dim}}_M(E) = \inf \Big\{\lambda \geq 0  ~ \big| ~  \limsup_{r\to 0} \mathcal{M}_r^\lambda(E) =0 \Big\}, 
$$ 
$$
\underline{\textmd{dim}}_M(E) = \inf \Big\{\lambda \geq 0  ~ \big| ~  \liminf_{r\to 0} \mathcal{M}_r^\lambda(E) =0 \Big\}. 
$$
If $\overline{\textmd{dim}}_M(E) = \underline{\textmd{dim}}_M(E) $, then the common value,  denoted by $\textmd{dim}_M(E)$,  is the Minkowski dimension of $E$.  Recall also that for a compact set $E\subset \mathbb{R}^n$, we have the relation $\textmd{dim}_H (E) \leq \underline{\textmd{dim}}_M (E)  \leq \overline{\textmd{dim}}_M (E)$, where $\textmd{dim}_H(E)$ is the Hausdorff dimension of $E$.  
We also  introduce a notation $\mathcal{C}^\alpha(\Omega)$ with $\alpha >0$ which will be used later. 
\begin{definition} ~

\begin{itemize}
\item [$1.$]If $\alpha$ is a positive integer, then $\mathcal{C}^\alpha(\Omega)$ is the usual space $C^\alpha(\Omega)$.
\item [$2.$] If $\alpha > [\alpha]$, then $\mathcal{C}^\alpha(\Omega)$ is the set of all functions $f\in C^{[\alpha]}(\Omega)$ satisfying 
$$
|\nabla^{[\alpha]}f(x) - \nabla^{[\alpha]}f(y)| \leq c(|x-y|) |x-y|^{\alpha - [\alpha]},~~~ x, y \in \Omega, 
$$
where $c(\cdot)$ is a nonnegative continuous function with $c(0)=0$. 
\end{itemize} 
\end{definition}

We will use $B_r(x)$ to denote the open ball of radius $r$ in $\mathbb{R}^n$ with center $x$ and write $B_r(0)$ as $B_r$ for short.   From now on, without loss of generality, we take the domain  $\Omega=B_2$.

Our assumption on the $Q$-curvature $K(x)$ is as follows:  

\begin{itemize}

\item [$1.$] For $n=2m+1$, $K \in \mathcal{C}^{\frac{1}{2}}(B_2)$.  

\item [$2.$] For $n=2m+2$, $K\in C^1(B_2)$. 

\item [$3.$] For $n\geq 2m+3$, $K\in C^1(B_2)$ and one of the following is satisfied: 

\begin{itemize}
\item [(K1)] If $x \in \Lambda$ is a critical point of $K$,  then there exists a neighborhood $N$ of $x$ such that
\begin{equation}\label{K1-0}
c_1 |y - x|^{\alpha-1} \leq |\nabla K(y)| \leq c_2 |y - x|^{\alpha-1}
\end{equation}
and for any $\varepsilon>0$, there exists $\delta=\delta(\varepsilon)$ such that for $|z - y|< \delta|y - x|$ we have
\begin{equation}\label{K1-1}
|\nabla K(z) - \nabla K(y)| < \varepsilon |y -x|^{\alpha - 1}, 
\end{equation}
where $\alpha >1$,   $y, z \in N$,  $c_1$ and $c_2$ are two positive constants.

\item [(K2)] $K\in \mathcal{C}^\alpha(B_2)$ with $\alpha=\frac{n-2m}{2}$.  In addition,  for  $n \geq 2m+4$, if  $x \in \Lambda$ is a critical point of $K$,  then there exists a neighborhood $N$ of $x$ such that 
\begin{equation}\label{K2-1}
|\nabla^i  K(y)| \leq c(|y - x|) |\nabla K(y)|^{\frac{\alpha-i}{\alpha-1}}, ~~~ 2\leq i \leq [\alpha], ~~ y \in N, 
\end{equation}
where $c(\cdot)$ is a nonnegative continuous function satisfying $c(0)=0$. 
\end{itemize} 

\end{itemize}   
Then our main result in this paper is 
\begin{theorem}\label{Thm01} 
Suppose that $1 \leq m < n/2$ and $m$ is an integer. Let $\Lambda\subset B_{1/2}$ be a compact set with the upper Minkowski dimension $\overline{\textmd{dim}}_M(\Lambda)$ (not necessarily an integer),  $\overline{\textmd{dim}}_M(\Lambda) < \frac{n-2m}{2}$,  or $\Lambda\subset B_{1/2}$ be a smooth $k$-dimensional closed manifold with $k\leq \frac{n-2m}{2}$.  Let $K$ satisfy the above assumption and  let $u\in C^{2m}(B_2 \backslash \Lambda)$ be a  solution of 
\begin{equation}\label{HOE} 
(-\Delta)^m u = K(x) u^{\frac{n+2m}{n-2m}} , ~~~ u>0  ~~~~~~   \textmd{in} ~B_2 \backslash \Lambda. 
\end{equation} 
Suppose
\begin{equation}\label{HOE02}   
(-\Delta)^s u \geq 0 ~~~~~~  \textmd{in} ~ B_2 \backslash \Lambda,~~ s=1, \dots, m-1.  
\end{equation} 
Then there exists a constant  $C > 0$ such that
\begin{equation}\label{Est01} 
u(x) \leq C  [\textmd{dist}(x,  \Lambda)]^{-\frac{n-2m}{2}} 
\end{equation} 
for all $x\in B_1\backslash \Lambda$, where $\textmd{dist}(x,  \Lambda)$ is the distance  between $x$ and $\Lambda$. 
\end{theorem}

\begin{remark}\label{Re-mama=01}
For $n \geq 2m+4$, if we only assume $K \in C^1(B_2)$, then Theorem \ref{Thm01} does not hold.  Indeed,  when $n \geq 2m+4$ and $\Lambda=\{0\}$,   Du and Yang in \cite{DY02} have shown the existence of $K\in C^1(B_2)$ such that Eq. \eqref{HOE}  has a $C^{2m}$  solution $u(x)$ which satisfies \eqref{HOE02}  but does not satisfy \eqref{Est01},  and can even be constructed to be arbitrarily large near its singularity $0$.  Such large singular solution of  \eqref{HOE} extends a similar result of Taliaferro \cite{T} for $m=1$.  On the other hand,  Theorem \ref{Thm01} also indicates that Theorem 1.3 of \cite{DY02} is not true in dimension $n=2m+1$ or $n=2m+2$ for  $m\geq 2$. 
\end{remark} 
\begin{remark}\label{Re-mama=02}
When $m = 2$,  a solution $u$ of \eqref{HOE}  defines  a conformal metric $g_{ij}=u^{\frac{4}{n-4}} \delta_{ij}$ which has $Q$-curvature $\frac{2}{n-4}K(x)$.  Assuming  that the scalar curvature of $g_{ij}$ is positive,  then one can obtain the sign condition $-\Delta u >0$. We also mention that under the assumption of positive scalar curvature, Gursky and Malchiodi in \cite{GM} studied the positivity of the Paneitz operator and its Green's function. 
\end{remark}  
\begin{remark}\label{Re-mama=03}
The assumption (K2) also covers the situation of $K(x) \equiv 1$.  In this case,  Mazzeo and Pacard in \cite{MP}  have constructed singular solutions of \eqref{HOE} for $m=1$ when  $\Lambda$ is a  smooth submanifold of dimension $k \leq (n-2)/2$,  and Hyder-Sire \cite{HS} recently proved for $m=2$ that Eq. \eqref{HOE} has singular solutions which satisfy \eqref{HOE02}  when $\Lambda$ is a  smooth submanifold of dimension $k < (n-4)/2$. 
\end{remark} 

As mentioned earlier,  Eq. \eqref{HOE}  is more challenging for $m\geq 2$ to study since the maximum principle is lacking. Notice also that the sign conditions \eqref{HOE02} may change when performing the Kelvin transformation. Thus, even under the assumption \eqref{HOE02},  it is still difficult to apply the method of moving planes  or moving spheres  directly to the local differential equation \eqref{HOE}.  Inspired by the work of Jin-Li-Xiong  \cite{JLX17,JX19},  we will rewrite the differential equation \eqref{HOE} into the local integral equation \eqref{Int} below and derive local estimates of singular solutions to this integral equation. Similar idea has also been used in \cite{DY01} to study the case when $K(x)$ is identically a positive constant  and the singular set $\Lambda$ is not isolated.

Suppose the dimension $n \geq 1$, $0< \sigma < \frac{n}{2}$ is a real number, and $\Sigma$ is a closed set in  $\mathbb{R}^n$.  We consider the local integral equation 
\begin{equation}\label{Int}
u(x) =\int_{B_2} \frac{ K(y) u(y)^{\frac{n+2\sigma}{n-2\sigma}} }{|x-y|^{n-2\sigma}} dy + h(x),~~~ u(x) >0, ~~~~~ x \in B_2 \backslash \Sigma, 
\end{equation}
where $u \in L^{\frac{n+2\sigma}{n-2\sigma}}(B_2) \cap C(B_2 \backslash \Sigma)$ and $ h\in C^1(B_2)$ is a positive function.  Under the assumptions in Theorem \ref{Thm01}, by a similar argument as in \cite{DY01}  we can show $u\in L_{\textmd{loc}}^{\frac{n+2\sigma}{n-2\sigma}}(B_2) $ and  can rewrite the equation \eqref{HOE}  locally  into the integral equation \eqref{Int} after some scaling (see Theorem \ref{T-Lo094} in the next section).

Next we state the corresponding result for singular solutions of the integral equation \eqref{Int}.  Denote $\mathcal{L}^n$  the $n$-dimensional Lebesgue measure on $\mathbb{R}^n$.   We assume that  the positive $Q$-curvature function $K(x)$ satisfies the one of the following: 
\begin{itemize}
\item [(K1)] $K\in C^1(B_2)$. If $x \in \Sigma$ is a critical point of $K$,  then there exists a neighborhood $N$ of $x$ such that
\begin{equation}\label{Khy1-0}
c_1 |y - x|^{\alpha-1} \leq |\nabla K(y)| \leq c_2 |y - x|^{\alpha-1}
\end{equation}
and for any $\varepsilon>0$, there exists $\delta=\delta(\varepsilon)$ such that for $|z - y|< \delta|y - x|$ we have
\begin{equation}\label{Khy1-1}
|\nabla K(z) - \nabla K(y)| < \varepsilon |y -x|^{\alpha - 1}, 
\end{equation}
where $\alpha >1$,   $y, z \in N$,  $c_1$ and $c_2$ are two positive constants.

\item [(K2)] $K\in \mathcal{C}^\alpha(B_2)$ with $\alpha=\frac{n-2\sigma}{2}$.  In addition,  for  $n \geq 2\sigma +4$, if  $x \in \Sigma$ is a critical point of $K$,  then there exists a neighborhood $N$ of $x$ such that 
\begin{equation}\label{Khy2-1}
|\nabla^i  K(y)| \leq c(|y - x|) |\nabla K(y)|^{\frac{\alpha-i}{\alpha-1}}, ~~~ 2\leq i  \leq [\alpha], ~~ y \in N, 
\end{equation}
where $c(\cdot)$ is a nonnegative continuous function satisfying $c(0)=0$. 
\end{itemize}

\begin{theorem}\label{IEthm01}
Suppose $n \geq 2$, $ 1 \leq  \sigma < n/2$, and $\Sigma$ is a closed set in  $\mathbb{R}^n$ with $\mathcal{L}^n (\Sigma)=0$.  Suppose that $K(x)$ satisfies the assumption (K1) or (K2) above.   Let $h\in C^1(B_2)$ be a positive function and  $u \in L^{\frac{n+2\sigma}{n-2\sigma}}(B_2) \cap C(B_2 \backslash \Sigma)$  be a positive solution of \eqref{Int}. Then there exists a constant $C>0$ such that 
\begin{equation}\label{In-Es01}
u(x) \leq C [\textmd{dist}(x, \Sigma)]^{-\frac{n-2\sigma}{2}}
\end{equation}
for all $x\in B_1\backslash \Sigma$, where  $\textmd{dist}(x,  \Sigma)$ is the distance  between $x$ and $\Sigma$.  
\end{theorem} 

\begin{remark}\label{Re-mama=04}
For the integral equation \eqref{Int},  we only assume that the singular set $\Sigma$ has $n$-dimensional Lebesgue measure zero. This is a weaker condition than singular set of Newtonian capacity zero, which is used in \cite{Chen-Lin95,Chen-Lin97,Zhang02} to study the second-order scalar curvature equation \eqref{H-01=00hh-01}.  In particular, if one obtains a solution $u \in C^{2m}(B_2 \backslash \Lambda) \cap L^{\frac{n+ 2m}{n - 2m}}(B_2)$ to \eqref{HOE} satisfying \eqref{HOE02} by whatever method with $\mathcal{L}^n(\Lambda)=0$,  then $u$ is a distributional solution in  $B_2$ and thus Theorem \ref{IEthm01} can give the local estimate \eqref{Est01} of $u$ near $\Lambda$ even though the Newtonian capacity of $\Lambda$ might be greater than $0$.  For example, Pacard \cite{P} have constructed solutions with such high dimensional singular set to \eqref{H-01=00hh-01} when $K(x)\equiv 1$ in dimension $4$ and in dimension $6$. 
\end{remark}  

\begin{remark}\label{Re+Ma=0620}
In Theorem \ref{IEthm01}, we assume that the order of the integral equation \eqref{Int} satisfies $1 \leq \sigma < n/2$ which is sufficient to apply to higher order $Q$-curvature  equation \eqref{H-01}. When $0 < \sigma <1$,  equation \eqref{Int} is closely related to the fractional Nirenberg problem (see \cite{JLX17}). For the fractional case $0< \sigma< 1$, our proof of Theorem \ref{IEthm01} encounters a difficulty due to the more singular properties of the integral kernel $G(0, \lambda; \xi, z)$ defined in \eqref{Kel=0620-2021}.  More specifically, the negativity of $\Phi_\lambda$ in \eqref{SL4-312=02} is very important when using the moving sphere method, but in the case of  $0< \sigma< 1$, the negative part of the integral in  \eqref{SL4-312=02}  cannot control the positive parts according to the estimates of $G(0, \lambda; \xi, z)$ in Lemma \ref{InK-01}. We plan to deal with this difficulty in future work.  We also mention that fractional order critical equations with singularities in the case $K(x) \equiv 1$  have been studied in \cite{ACDFGW,ADGW,CJSX,DG,DPGW,JQSX} and so on. 
\end{remark} 

We will prove Theorem \ref{IEthm01} in the spirit of the works of  Chen, Lin, Taliaferro and Zhang \cite{Chen-Lin97,Lin00,TZ,Zhang02} by using the moving sphere method of an integral form introduced by Li\cite{Li04},  which is inspired by Li-Zhu \cite{Li-Zhu} and Li-Zhang \cite{LZ03}.  One difference is that the authors of \cite{Chen-Lin97,Lin00,TZ,Zhang02} dealt directly with the second-order equation \eqref{H-01=00hh-01}, while we work with the integral equation \eqref{Int} by exploring its various specific features.  Thus we need some analysis techniques to overcome the difficulties caused by the absence of a maximum principle. Another difference is the non-locality of the integral equation \eqref{Int}.     

The rest of this paper is organized as follows. In Section \ref{S2}, we show the integral representation for  singular positive solutions to  the differential equation \eqref{HOE}. In Section \ref{Sw46}, we provide some preliminary estimates for the integral kernel involved in the moving sphere method of integral form.     In Section \ref{S3000},  we give a detailed proof of Theorem \ref{IEthm01} under the assumption (K1).  In Section \ref{S5000},  we present the proof of Theorem \ref{IEthm01} under the assumption (K2). Finally, Theorem \ref{Thm01} follows from Theorem \ref{IEthm01} and Theorem \ref{T-Lo094} via the covering and rescaling arguments.

\section{An integral representation for singular solutions}\label{S2}

In this section, we show that every singular positive solution of the differential equation \eqref{HOE} satisfies the integral equation \eqref{Int} in some local sense under suitable assumptions.   We first prove that under the assumptions of Theorem \ref{Thm01} (\eqref{HOE02} is not needed here), $u \in L_{\textmd{loc}}^{\frac{n + 2m}{n - 2m}}(B_2) $ and  $u$ is a distributional solution in the entire ball $B_2$. 

\begin{proposition}\label{P201}
Suppose that $1 \leq m < n/2$ and $m$ is an integer. Let $\Lambda\subset B_{1/2}$ be a compact set with the upper Minkowski dimension $\overline{\textmd{dim}}_M(\Lambda)$ (not necessarily an integer),  $\overline{\textmd{dim}}_M(\Lambda) < \frac{n-2m}{2}$,  or $\Lambda\subset B_{1/2}$ be a smooth $k$-dimensional closed manifold with $k\leq \frac{n-2m}{2}$. Let $K$ be a positive continuous function on $\overline{B}_2$ and  let $u\in C^{2m}(\overline{B}_2 \backslash \Lambda)$ be a positive solution of \eqref{HOE}. Then $u \in L^{\frac{n+ 2m}{n - 2m}}(B_2) $ and $u$ is a distributional solution in the entire ball $B_2$, i.e., we have 
\begin{equation}\label{Dis2}
\int_{B_2} u (-\Delta)^m \varphi  dx = \int_{B_2} K u^{\frac{n+2m}{n-2m}} \varphi  dx
\end{equation}
for every  $\varphi \in C_c^\infty(B_2)$. 
\end{proposition}

\begin{proof}The proof is very similar to that of \cite[Proposition 2.1]{DY01}, so we omit the details. See also some related arguments in \cite{AGHW,Y}. 
\end{proof}  

Suppose $n >2m$. Let  $G_m(x, y)$ be the Green function of $(-\Delta)^m$ on $B_2$ under the Navier
boundary condition: 
\begin{equation}\label{green}
\begin{cases}
(-\Delta)^m G_m(x, \cdot) =\delta_x ~~ & \textmd{in }~ B_2, \\
G_m(x, \cdot)=-\Delta G_m(x, \cdot) = \cdots =(-\Delta)^{m-1} G_m(x, \cdot) =0 ~~ & \textmd{on}~ \partial B_2,
\end{cases}
\end{equation}
where $\delta_x$ is the Dirac measure to the point  $x \in B_2$. 
Then,  for any $u\in C^{2m}(B_2) \cap C^{2m-2}(\overline{B}_2)$ we have 
\begin{equation}\label{Gr-0578}
u(x)=\int_{B_2} G_m(x, y) (-\Delta)^m u(y) dy + \sum_{i=1}^m  \int_{\partial B_2} H_i(x, y) (-\Delta)^{i-1}u(y) dS_y, 
\end{equation}
where 
$$
H_i(x,y)= - \frac{\partial}{\partial\nu_y} (-\Delta_y)^{m-i} G_m(x, y)~~~~ \textmd{for} ~ x\in B_2, y\in \partial B_2. 
$$ 
In particular, taking $u\equiv1$ in \eqref{Gr-0578} yields  
\begin{equation}\label{Gn-0798}
\int_{\partial B_2} H_1(x,y) d S_y=1~~~~~~~ \textmd{for} ~ \textmd{all} ~ x\in B_2.
\end{equation}
Moreover, a direct computation gives 
\begin{equation}\label{GrFu}
G_m(x,y)=c_{n,m} |x-y|^{2m-n} + A_m(x, y),
\end{equation}
 $c_{n,m}=\frac{\Gamma(\frac{n}{2} - m)}{2^{2m} \pi^{n/2} \Gamma(m)}$, $A_m(x, y)$ is smooth in $B_2 \times B_2$, and 
\begin{equation}\label{GrFu09}
H_i(x, y) \geq 0,~~~~~~ i=1, \dots, m. 
\end{equation}

\begin{proposition}\label{P202}
Under the  assumptions of  Proposition \ref{P201}, we have  
\begin{equation}\label{InRLo}
u(x)=\int_{B_2} G_m(x, y) K(y) u(y)^{\frac{n+2m}{n-2m}} dy + \sum_{i=1}^m  \int_{\partial B_2} H_i(x, y) (-\Delta)^{i-1}u(y) dS_y
\end{equation}
for all $x\in B_2 \backslash \Lambda$. 
\end{proposition}
\begin{proof}
For any $x\in B_2 \backslash \Lambda$,  we define 
$$
v(x) = \int_{B_2} G_m(x, y) K(y) u(y)^{\frac{n+2m}{n-2m}} dy + \sum_{i=1}^m  \int_{\partial B_2} H_i(x, y) (-\Delta)^{i-1}u(y) dS_y. 
$$
Because  $u(y)^{\frac{n+2m}{n-2m}} \in L^1(B_2)$ and the Riesz potential $|x|^{2m-n}$ is weak type $\left( 1, \frac{n}{n-2m} \right)$, we obtain $v\in L^{\frac{n}{n-2m}}_{weak}(B_2) \cap L^1(B_2)$.  Define  $w=u-v$. By Proposition \ref{P201} we have that $w$ satisfies 
$$
(-\Delta)^m w=0 ~~~~~ \textmd{in} ~ B_2
$$
in the distributional sense, i.e., for any $\varphi\in C_c^\infty(B_2)$,
$$
\int_{B_2} w (-\Delta)^m \varphi dx =0. 
$$
It follows from the regularity for polyharmonic functions (see, e.g., Mitrea \cite{Dis-book}) that $w$ is smooth and satisfies  $(-\Delta)^m w=0$   in  $B_2$.  Note that $w=-\Delta w = \cdots =
(-\Delta)^{m-1} w =0$ on $\partial B_2$, hence $w\equiv 0$.  This implies that $u=v$ in $B_2 \backslash \Lambda$.   
\end{proof}

Now we show that $u$ satisfies the integral equation \eqref{Int} in some local sense under the additional assumption \eqref{HOE02}.      
\begin{theorem}\label{T-Lo094}
Suppose that $1 \leq m < n/2$ and $m$ is an integer. Let $\Lambda\subset B_{1/2}$ be a compact set with the upper Minkowski dimension $\overline{\textmd{dim}}_M(\Lambda)$ (not necessarily an integer),  $\overline{\textmd{dim}}_M(\Lambda) < \frac{n-2m}{2}$,  or $\Lambda\subset B_{1/2}$ be a smooth $k$-dimensional closed manifold with $k\leq \frac{n-2m}{2}$.  Let $K$ be a positive  continuous function on $\overline{B}_2$ and let $u\in C^{2m}(B_2 \backslash \Lambda)$ be a positive solution of 
\begin{equation}\label{HOE-409} 
(-\Delta)^m u = K(x) u^{\frac{n+2m}{n-2m}} ~~~~~~  \textmd{in} ~B_2 \backslash \Lambda. 
\end{equation} 
Suppose
\begin{equation}\label{HOE02-409}   
(-\Delta)^s u \geq 0 ~~~~~~  \textmd{in} ~ B_2 \backslash \Lambda,~~ s=1, \dots, m-1.  
\end{equation} 
Then there exists $\tau>0$ (independent of $x\in \Lambda$)  such  that for any $x_0 \in \Lambda$  we have
\begin{equation}\label{T_Lo672}
u(x) = c_{n,m} \int_{B_\tau(x_0)} \frac{K(y) u(y)^{\frac{n+2m}{n-2m}}}{|x - y|^{n-2m}} dy + h_1(x) ~~~~~\textmd{for} ~ x\in B_\tau(x_0) \backslash \Lambda, 
\end{equation}
where $h_1(x)$ is a positive smooth function in $B_\tau(x_0)$.  
\end{theorem}
 
\begin{proof}
We may assume, without loss of generality, that $u\in C^{2m}(\overline{B}_2 \backslash \Lambda)$ and $u > 0$ in $\overline{B}_2 \backslash \Lambda$. Otherwise, we just consider the equation in a smaller ball. 

From the assumptions on the singular set $\Lambda$  we obtain  $\textmd{Cap}(\Lambda)=0$, where $\textmd{Cap}(\Lambda)$ is the Newtonian capacity of $\Lambda$ (see, e.g.,  \cite{EGbook}).      Since $u>0$ and $-\Delta u \geq 0$ in $B_2 \backslash \Lambda$, the maximum principle (see, e.g.,  \cite[Lemma 2.1]{Chen-Lin95}) leads to   
$$
u(x) \geq c_0:=\inf_{\partial B_2} u >0~~~~~ \textmd{for} ~ \textmd{all} ~ x \in \overline{B}_2 \backslash \Lambda. 
$$
By Proposition \ref{P201} we have $u^{\frac{n+2m}{n-2m}} \in L^1(B_2)$.  Thus, there exists $0< \tau< \frac{1}{4}$ independent of $z\in B_1$ such that  
$$
\int_{B_\tau(z)} |A_m(x, y)| K(y) u(y)^{\frac{n+2m}{n-2m}} dy  < \frac{c_0}{2} ~~~~~ \textmd{for}~  \textmd{all}  ~ x \in B_\tau(z) \subset B_{3/2},
$$
where $A_m(x, y)$ is defined in \eqref{GrFu}. For every  $x_0 \in \Lambda$, using Proposition \ref{P202} we can write 
$$
u(x)=c_{n,m} \int_{B_\tau(x_0)} \frac{K(y) u(y)^{\frac{n+2m}{n-2m}}}{|x - y|^{n-2m}} dy + h_1(x)~~~~~ \textmd{for}  ~ x\in B_\tau(x_0) \backslash \Lambda, 
$$
where
$$
\aligned
h_1(x)  &  =  \int_{B_\tau(x_0)} A_m(x, y) K(y) u(y)^{\frac{n+2m}{n-2m}}dy + c_{n,m} \int_{B_2 \backslash B_\tau(x_0)} G_m(x, y) K(y) u(y)^{\frac{n+2m}{n-2m}}dy \\
&  ~~~ + \sum_{i=1}^m  \int_{\partial B_2} H_i(x, y) (-\Delta)^{i-1}u(y) dS_y \\
&  \geq -\frac{c_0}{2} +  \int_{\partial B_2} H_1(x, y) u(y) dS_y \\
&  \geq  -\frac{c_0}{2} + \inf_{\partial B_2} u =\frac{c_0}{2} >0 ~~~~~~ \textmd{for} ~ x \in B_\tau(x_0), 
\endaligned
$$
where we have used the sign conditions \eqref{HOE02-409} in the first inequality and \eqref{Gn-0798} in the second inequality.  It is easy to check that $h_1$ is smooth in $B_\tau(x_0)$ and satisfies $(-\Delta)^m h_1=0$ in $B_\tau(x_0)$.  The proof of Theorem \ref{T-Lo094} is completed. 
\end{proof} 

\begin{proof}[Proof of Theorem \ref{Thm01}]
It follows from Theorem \ref{IEthm01} and Theorem \ref{T-Lo094} via the covering and rescaling arguments.  
\end{proof}

\section{Preliminary results}\label{Sw46} 
For $x\in \mathbb{R}^n, \lambda >0$ and a function $u$,  we denote
\begin{equation}\label{N0Ta-0021}
\xi^{x, \lambda}= x+ \frac{\lambda^2 (\xi - x)}{|\xi - x|^2} ~~~  \textmd{for} ~  x \neq \xi \in \mathbb{R}^n   ~~~~ \textmd{and} ~~~~\Omega^{x,\lambda} = \{\xi^{x, \lambda}:  \xi\in \Omega\}. 
\end{equation}  
Let $u$ be a positive function,  its Kelvin transformation is defined  as  
\begin{equation}\label{N0Ta-0022}
u^{x, \lambda}(\xi) = \left( \frac{\lambda}{|\xi - x|} \right)^{n-2\sigma} u(\xi^{x, \lambda}). 
\end{equation}
Note that $(\xi^{x,\lambda})^{x, \lambda} =\xi$ and $(u^{x,\lambda})^{x,\lambda} \equiv  u$.  If $x=0$, we use the notations $\xi^{\lambda}=\xi^{0, \lambda}$ and $u^\lambda=u^{0,\lambda}$, i.e.,  
\begin{equation}\label{N0Ta-0039}
\xi^{\lambda}= \frac{\lambda^2  \xi }{|\xi |^2}  ~~~~~ \textmd{and}  ~~~~~  u^{\lambda}(\xi) = \left( \frac{\lambda}{|\xi |} \right)^{n-2\sigma} u \left( \frac{\lambda^2  \xi }{|\xi |^2} \right).  
\end{equation}

Suppose $0 < \sigma < n/2$,  $u \in L^{\frac{n+2\sigma}{n-2\sigma}}(B_2) \cap C(B_2 \backslash \Sigma)$  is a positive solution of \eqref{Int} and  $h\in C^1(B_2)$ is a positive function.   If we extend both $u$ and $K$ to be identically $0$ outside $B_2$, then 
\begin{equation}\label{A-01}
u(x) = \int_{\mathbb{R}^n} \frac{K(y) u(y)^{\frac{n+2\sigma}{n-2\sigma}}}{|x - y|^{n-2\sigma}} dy  + h(x)~~~~~~ \textmd{for}  ~ x\in B_2 \backslash \Sigma. 
\end{equation}
Making a change of variables,  we also have the following two identities (see, e.g., \cite{Li04}), 
\begin{equation}\label{Id-01}
\left( \frac{\lambda}{|\xi - x|} \right)^{n-2\sigma} \int_{|z-x|\geq \lambda}  \frac{K(z) u(z)^{\frac{n+2\sigma}{n-2\sigma}}}{|\xi^{x,\lambda} - z|^{n-2\sigma}} dz = \int_{|z-x| \leq \lambda} \frac{K(z^{x,\lambda}) u^{x,\lambda}(z)^{\frac{n+2\sigma}{n-2\sigma}}}{|\xi - z|^{n-2\sigma}} dz
\end{equation}
and
\begin{equation}\label{Id-02}
\left( \frac{\lambda}{|\xi - x|} \right)^{n-2\sigma} \int_{|z-x|\leq \lambda}  \frac{K(z) u(z)^{\frac{n+2\sigma}{n-2\sigma}}}{|\xi^{x,\lambda} - z|^{n-2\sigma}} dz = \int_{|z-x| \geq \lambda} \frac{K(z^{x,\lambda}) u^{x,\lambda}(z)^{\frac{n+2\sigma}{n-2\sigma}}}{|\xi - z|^{n-2\sigma}} dz.
\end{equation}
Thus, we obtain 
\begin{equation}\label{ABC-01}
u^{x,\lambda}(\xi) = \int_{\mathbb{R}^n} \frac{K(z^{x,\lambda}) u^{x,\lambda}(z)^{\frac{n+2\sigma}{n-2\sigma}}}{|\xi - z|^{n-2\sigma}} dz  + h^{x,\lambda}(\xi)~~~~~~ \textmd{for}  ~  \xi \in (B_2 \backslash \Sigma)^{x, \lambda}. 
\end{equation}
Hence,  for any $x\in B_1$ and $\lambda <1$, we have  for any  $\xi \in B_2 \backslash \left( \Sigma \cup \Sigma^{x,\lambda} \cup B_\lambda(x) \right)$ that 
\begin{equation}\label{ABC-0e571}
\aligned
u(\xi) -u^{x,\lambda}(\xi)  & = \int_{|z-x|\geq \lambda} G(x, \lambda; \xi, z) \Big[ K(z) u(z)^{\frac{n+2\sigma}{n-2\sigma}} - K(z^{x,\lambda}) u^{x,\lambda}(z)^\frac{n+2\sigma}{n-2\sigma} \Big] dz \\
& ~~~~~  + h^{x,\lambda}(\xi) - h(\xi), 
\endaligned
\end{equation} 
where
\begin{equation}\label{Kel=0620-2021}
G(x, \lambda; \xi, z) := \frac{1}{|\xi -z|^{n-2\sigma}} - \left(\frac{\lambda}{|\xi - x|}  \right)^{n-2\sigma} \frac{1}{|\xi^{x,\lambda} -z|^{n-2\sigma}},  ~~ |\xi -x|, |z-x| >\lambda>0.
\end{equation}
It is elementary to check that 
\begin{equation}\label{GIS893-01}
G(x, \lambda; \xi, z) >0~~~~~ \textmd{for} ~ \textmd{all} ~ |\xi -x|, |z-x| >\lambda>0. 
\end{equation}

We also give the following estimates for the integral kernel $G(0, \lambda; \xi, z)$.  Their proofs are elementary,  and we include them  here for completeness.  
\begin{lemma}\label{InK-01} Assume $0 < \sigma < n/2$ and $\lambda>0$. Then for $\xi, z \in \mathbb{R}^n \backslash \overline{B_\lambda}$, we have 
\begin{itemize}
\item [$(1)$]  There exists a positive constant $C_1=C_1(n,\sigma)$ such that for $|\xi -z|< \frac{1}{3}(|\xi| -\lambda)$, 
\begin{equation}\label{GIK0-01}
G(0, \lambda; \xi, z) \geq C_1\frac{1}{|\xi -z|^{n-2\sigma}}. 
\end{equation} 

\item [$(2)$] There exist two positive constants  $C_2=C_2(n,\sigma)$ and  $C_3=C_3(n,\sigma)$ such that for $|\xi -z|\geq  \frac{1}{3}(|\xi| -\lambda)$ and $|\xi| \leq 10 \lambda$, 
\begin{equation}\label{GIK0-02}
C_2 \frac{(|\xi|-\lambda) (|z|^2 - \lambda^2)}{\lambda|\xi- z|^{n-2\sigma+2}} \leq G(0, \lambda; \xi, z) \leq C_3 \frac{(|\xi|-\lambda) (|z|^2 - \lambda^2)}{\lambda|\xi- z|^{n-2\sigma+2}}. 
\end{equation} 
Moreover, if we only assume $\lambda < |\xi| \leq  10\lambda$ and $|z|>\lambda$, then we have the second inequality of \eqref{GIK0-02}. 
\item [$(3)$] $G(0, \lambda; \xi, z)=G(0, \lambda; z, \xi)$. 
\end{itemize}
\end{lemma} 

\begin{remark}\label{InK-02}
By the symmetry of $G(0, \lambda; \xi, z)$,  we can reverse the roles of  $\xi$  and  $z$ in $(1)$ and $(2)$ of Lemma \ref{InK-01} and the corresponding conclusions still hold.   
\end{remark} 

\begin{proof}
(1)   Note that for $|\xi - z| \leq \frac{1}{3}(|\xi| -\lambda)$ and $|\xi| > \lambda$, we have $ |\xi^\lambda -z| \geq |\xi^\lambda - \xi| - |\xi -z| \geq 2|\xi -z|$. 
Thus 
$$
G(0, \lambda; \xi, z) \geq (1- 2^{2\sigma-n})  |\xi -z|^{2\sigma -n}. 
$$

(2) The kernel $G(0, \lambda; \xi, z)$ can be rewritten as
\begin{equation}\label{YGmgt-HK-409=01}
G(0, \lambda; \xi, z) = |\xi -z|^{2\sigma-n} - \bigg| \frac{\lambda}{|\xi|} \xi - \frac{|\xi|}{\lambda} z \bigg|^{2\sigma-n}. 
\end{equation} 
By a direct calculation,  we have the following formula 
\begin{equation}\label{YGmgt-HK-409=02}
\bigg| \frac{\lambda}{|\xi|} \xi - \frac{|\xi|}{\lambda} z \bigg|^{2} - |\xi -z|^{2}  = \frac{(|\xi|^2 - \lambda^2)(|z|^2 - \lambda^2)}{\lambda^2}, ~~~ \xi \neq 0. 
\end{equation} 
This obviously implies that the kernel $G(0, \lambda; \xi, z)$ is positive.  Let $f(t)=t^{-(n-2\sigma)/2}$, then by the mean value theorem and \eqref{YGmgt-HK-409=02} we get 
$$
\aligned
G(0, \lambda; \xi, z) & = f \left( |\xi -z|^2 \right) - f\bigg( \bigg| \frac{\lambda}{|\xi|} \xi - \frac{|\xi|}{\lambda} z \bigg|^2  \bigg)\\
&= \frac{n-2\sigma}{2}  \frac{(|\xi|^2 - \lambda^2)(|z|^2 - \lambda^2)}{\lambda^2 \left(  \theta|\xi -z|^2 + (1-\theta) \left| \frac{\lambda}{|\xi|} \xi - \frac{|\xi|}{\lambda} z \right|^2 \right)^{\frac{n-2\sigma+2}{2}}} 
\endaligned
$$
for some $\theta \in (0, 1)$.  When $|\xi -z|\geq  \frac{1}{3}(|\xi| -\lambda)$ and $\lambda < |\xi| \leq 10 \lambda$ we have
$$
\left| \frac{\lambda}{|\xi|} \xi - \frac{|\xi|}{\lambda} z \right| \leq \left| \frac{\lambda}{|\xi|} \xi - \frac{|\xi|}{\lambda} \xi \right| +  \left| \frac{|\xi|}{\lambda} \xi -  \frac{|\xi|}{\lambda}z  \right| \leq 43 |\xi -z| 
$$
and thus 
$$
G(0, \lambda; \xi, z) \geq  C_2 \frac{(|\xi|-\lambda) (|z|^2 - \lambda^2)}{\lambda|\xi- z|^{n-2\sigma+2}}
$$
with some constant $C_2=C_2(n, \sigma)$. On the other hand, if we only assume $\lambda < |\xi| \leq  10\lambda$ and $|z|>\lambda$, then by  \eqref{YGmgt-HK-409=02} we obtain $|\xi -z| \leq \bigg| \frac{\lambda}{|\xi|} \xi - \frac{|\xi|}{\lambda} z \bigg|$ and so 
$$
G(0, \lambda; \xi, z) \leq  C_1 \frac{(|\xi|-\lambda) (|z|^2 - \lambda^2)}{\lambda|\xi- z|^{n-2\sigma+2}} 
$$
for some constant $C_1=C_1(n, \sigma)$.    

(3) Note that 
\begin{equation}\label{YGmgt-HK-409=03}
\frac{|z|}{\lambda} \frac{|\xi|}{\lambda}  |\xi^\lambda - z^\lambda| = \left| \frac{|z|}{|\xi|} \xi - \frac{|\xi|}{|z|}z \right| = |\xi -z| ~~~~~~ \textmd{for}  ~  \xi,  z \neq 0, 
\end{equation} 
where we have used the basic identity $\left| \frac{|z|}{|\xi|} \xi - \frac{|\xi|}{|z|}z \right|^2 =|\xi -z|^2$. This implies that  
\begin{equation}\label{YGmgt-HK-409=04}
\frac{|\xi|}{\lambda}  |\xi^\lambda - z| = \frac{|z|}{\lambda} |z^\lambda -  \xi| ~~~~~~ \textmd{for}  ~ \xi, z \neq 0, 
\end{equation} 
from which we obtain $G(0, \lambda; \xi, z)=G(0, \lambda; z, \xi)$.  Lemma \ref{InK-01} is proved.  
\end{proof}

\begin{lemma}\label{InK-03} 
Assume $0 < \sigma < n/2$.  Let 
$$
U(y) = \left( \frac{1}{1 + |y - Re|^2}\right)^{\frac{n-2\sigma}{2}}, 
$$
where $e \neq 0$ and $R|e| >10$. Then there exists a positive constant $C=C(n, \sigma, R|e|)$ such that for $\lambda_0=R|e| -2$,  we have
$$
U(y) - U^{\lambda_0} (y) \geq C (|y|^2 - \lambda_0^2) |y|^{-(n-2\sigma+2)} ~~~~~~~~  \textmd{for} ~\textmd{all} ~ |y|\geq \lambda_0
$$
and
$$
\frac{\partial(U - U^{\lambda_0})}{\partial\nu} > C >0 ~~~~~~~~  \textmd{on} ~ \partial B_{\lambda_0}, 
$$
where $\nu$ denotes the unit outer normal vector of $\partial B_{\lambda_0}$. 

For $\lambda_1 = R|e| +2$, we have
$$
U(y) - U^{\lambda_1} (y) < 0 ~~~~~~~~  \textmd{for} ~\textmd{all} ~  |y| > \lambda_1. 
$$
\end{lemma} 

\begin{proof}
Using the mean value theorem and the formula  \eqref{YGmgt-HK-409=02},  we get that for any $\lambda>0$, 
\begin{equation}\label{YGmgt-HK-409=05}
\aligned
U(y) - U^{\lambda} (y) & = \left( \frac{1}{1 + |y - Re|^2}\right)^{\frac{n-2\sigma}{2}} - \left( \frac{1}{\frac{|y|^2}{\lambda^2} + \frac{|y|^2}{\lambda^2}|y^\lambda - Re|^2}\right)^{\frac{n-2\sigma}{2}} \\
& = \frac{n-2\sigma}{2} \cdot \frac{(|y|^2 -\lambda^2)(R^2|e|^2 - \lambda^2 +1)}{\lambda^2 \eta^{(n-2\sigma+2)/2}}, 
\endaligned
\end{equation} 
where $\eta$ is some number between $1 + |y - Re|^2$ and $\frac{|y|^2}{\lambda^2} + \frac{|y|^2}{\lambda^2}|y^\lambda - Re|^2$. When $\lambda=\lambda_0$, it is easy to see that $R^2|e|^2 - \lambda^2 +1 >0$ and $\eta \leq c|y|^2$  for some constant $c=c(R)>0$. Hence there exists $C=C(n, \sigma, R)>0$ such that 
$$
U(y) - U^{\lambda_0} (y) \geq C  (|y|^2 - \lambda_0^2) |y|^{-(n-2\sigma+2)} ~~~~~~~~  \textmd{for} ~\textmd{any} ~ |y|\geq \lambda_0. 
$$
This also implies that 
$$
\frac{\partial(U - U^{\lambda_0})}{\partial\nu} > C >0 ~~~~~~~~  \textmd{on} ~ \partial B_{\lambda_0}
$$
for another constant $C=C(n, \sigma, R)>0$, where $\nu$ means the unit outer normal vector of $\partial B_{\lambda_0}$.   \

When $\lambda=\lambda_1$, we have  $R^2|e|^2 - \lambda^2 +1 <0$ and by \eqref{YGmgt-HK-409=05}, 
$$
U(y) - U^{\lambda_1} (y) < 0 ~~~~~~~~  \textmd{for} ~\textmd{any} ~  |y| > \lambda_1. 
$$
Lemma \ref{InK-03} is proved.  
\end{proof}

\section{Local estimates under the assumption (K1)}\label{S3000} 
In this section,  by using the method of moving spheres introduced by Li-Zhu \cite{Li-Zhu,Li04},  we shall prove Theorem \ref{IEthm01} under the assumption (K1) in the spirit of the works of Chen-Lin \cite{Chen-Lin97,Lin00} and Zhang \cite{Zhang02}.  Unlike \cite{Chen-Lin97,Lin00,Zhang02} dealing directly with second-order differential equations, we study the problem in a framework of integral equations. In particular, analysis techniques for integral equations are needed to overcome the lack of maximum principle.    

\vskip0.10in 

\noindent{\it Proof of Theorem \ref{IEthm01} under the assumption (K1). }  Suppose by contradiction that there exists a sequence $\{x_j\}_{j=_1}^\infty \subset B_1 \backslash \Sigma$ such that 
$$
d_j:=\textmd{dist} (x_j, \Sigma) \to 0~~~~~ \textmd{as}  ~ j\to \infty, 
$$
but
$$
d_j^{\frac{n-2\sigma}{2}} u(x_j) \to \infty~~~~~ \textmd{as}  ~ j\to \infty. 
$$
Without  loss of generality we may assume that $0\in \Sigma$ and $x_j\to 0$ as $j\to \infty$.  Since the following proof is very long, we first explain the idea and the sketch of the proof. 

In {\it Step 1}, by using the blow up analysis we show that $x_j$ can be chosen as the local maximum points of $u$.  

In {\it Step 2},  we show that $\nabla K(0)=0$ under the assumption of $K\in C^1(B_2)$.   If not,  then we assume without loss of generality that $\nabla K(0) =e =(1, 0, \dots, 0)$.  Let $w_j$ be the scaled and shifted function in \eqref{21SS3-01} for some sufficiently large $R>0$ and let $w_j^\lambda$ be the Kelvin transformation of $w_j$  in \eqref{21SS3-022}. Define $\varphi_\lambda(y) = w_j(y) - w_j^\lambda(y)$ and $\Phi_\lambda(y)$ as in \eqref{SL4-312=02} with $\lambda \in [R-2, R+2]$,  then $\varphi_\lambda + \Phi_\lambda$ satisfies the integral inequality \eqref{SL4-312=04}.  More importantly,  $\Phi_\lambda$ is negative and satisfies the estimates in Lemma \ref{LS4-312=01} under the current assumption that $\nabla K(0) =(1, 0, \dots, 0)$.    By Lemmas \ref{LAn-01} and \ref{LS4-312=01},  for $\lambda=\lambda_0=R-2$ we have 
\begin{equation}\label{2021==0623_01}
\varphi_{\lambda} + \Phi_{\lambda} \geq 0~~~~~~ \textmd{in} ~\Pi_j \backslash \overline{B}_{\lambda}, 
\end{equation} 
and for $\lambda_1=R+2$  we have 
\begin{equation}\label{2021==0623_02}
\varphi_{\lambda_1} (y^*)+ \Phi_{\lambda_1} (y^*) < 0  ~~~~~~ \textmd{for} ~ \textmd{some} ~ y^* \in B_{2\lambda_1} \backslash  \overline{B}_{\lambda_1}.    
\end{equation}  
Thus we can start moving the sphere continuously for  $\varphi_{\lambda} + \Phi_{\lambda}$ from $\lambda=\lambda_0$ as long as \eqref{2021==0623_01} holds,   and the sphere must stop at some $\bar{\lambda} \in [\lambda_0, \lambda_1)$.  Furthermore, with the help of the negativity of $\Phi_{\lambda}$ in Lemma \ref{LS4-312=01} and of the lower bounds of the remainder term $J_\lambda$ in Lemma \ref{hhff=01},  the moving sphere procedure may continue beyond $\bar{\lambda}$ where we reach a contradiction. 

In {\it Step 3}, we complete the proof of Theorem \ref{IEthm01} under the assumption of (K1). By {\it Step 2} we know that the origin $0$ is a critical point of $K$.  Without loss of generality we may assume that  $\lim_{j \to \infty} |\nabla K(x_j)|^{-1} \nabla K(x_j) =(1, 0, \dots, 0)$.  We will follow the notation in {\it Step 2}. As in {\it Step 2}, the main idea is still to use the moving sphere method for $\varphi_{\lambda} + \Phi_{\lambda}$, but under the current assumption of $K$, the function $\Phi_{\lambda}$ may not be negative.  However,  we can rewrite $\Phi_{\lambda}$ as $\Phi_{1,\lambda} + \Phi_{2,\lambda}$ where $\Phi_{1,\lambda}$ is negative and $\Phi_{2,\lambda}$ has a good estimate (see Lemmas \ref{L-swjhdg-010} and \ref{Lg_lammfb=01}).  Based on the estimates of $\Phi_{2,\lambda}$, we consider two cases for $\alpha>1$ separately. 

\begin{itemize}
\item [$(1)$]  {\it $\alpha < 2\sigma$ or $\alpha \geq  (n-2\sigma)/2$}. In this case, by choosing a sufficiently small $\varepsilon>0$ we construct the following function
$$
H_\lambda(y)  = - \varepsilon M_j^{-1} (\lambda^{2\sigma-n} - |y|^{2\sigma -n} ) + \Phi_\lambda(y), ~~~~ y\in \Pi_j \backslash \overline{B}_{\lambda} 
$$
which can be negative.  Using the lower estimates of the remainder term $J_\lambda$ in Lemma \ref{hhff=01}, the method of moving spheres can be applied to $\varphi_{\lambda} + H_{\lambda}$ to reach a contradiction.   

\item [$(2)$] {\it $2\sigma \leq \alpha < (n-2\sigma)/2$}.  In this case, $\Phi_{2,\lambda}$ on $\Pi_j \backslash B_{2l_j}$ is too large to construct $H_\lambda$ as in the first case, where $l_j:=u(x_j)^{\frac{2}{n-2\sigma}}|x_j| \to \infty$.  To deal with this difficulty, a different idea is needed. Note that $\Phi_\lambda$ is negative on $B_{2l_j} \backslash \overline{B}_{\lambda}$ and could be non-negative on $\Pi_j \backslash B_{2l_j}$,  thus  when moving the sphere, the trouble is 
\begin{equation}\label{Touble=0012}
\int_{\mathcal{O}_\lambda} G(0, \lambda; y, z) b_\lambda(z) \varphi_\lambda(z) dz, 
\end{equation} 
where $\mathcal{O}_\lambda= \{y \in \Pi_j \backslash B_{2l_j}:  w_j(y) <   w_j^\lambda(y) \}$. To control this integral, we add a function $T_\lambda$ defined in \eqref{Tiyyt=20210624} to both sides of the integral inequality \eqref{SL4-312=04}.  Let $H_\lambda :=\Phi_{1,\lambda} + T_\lambda$, then $\varphi_\lambda + H_\lambda$ satisfies the integral inequality \eqref{318=09sdbfjke}.  Now when the method of moving spheres is applied to $\varphi_\lambda + H_\lambda$ (starting from $\lambda_0=R-2$),  the troublesome integral \eqref{Touble=0012} can be controlled by
$$
M_j^{-\frac{2\alpha}{n-2\sigma}}\int_{\Omega_j \backslash B_{2l_j}} G(0, \lambda; y, z)  |z|^{\alpha-2\sigma-n} dz 
$$
with the help of Lemmas \ref{L-swjhdg-010} and \ref{Lg_lammfb=01}. This will lead to a contradiction.   
\end{itemize}

Now we return to the detailed proof of Theorem \ref{IEthm01} under the assumption (K1).   

\vskip0.10in 

{\it Step 1.  We show that $x_j$ can be chosen as the local maximum points of $u$. Moreover, the functions $u(x_j)^{-1} u( x_j +  u(x_j)^{-\frac{2}{n-2\sigma}} y)$ converge in $C_{loc}^2(\mathbb{R}^n)$, after passing a subsequence,  to a positive function $U_0 \in C^2(\mathbb{R}^n)$ where $U_0$ satisfies   
\begin{equation}\label{S3-W1-0376}
\begin{cases}
U_0(y) = \int_{\mathbb{R}^n} \frac{K(0) U_0(z)^{\frac{n+2\sigma}{n-2\sigma}}}{|y -z |^{n-2\sigma}} dz~~~~~ \textmd{for} ~ y \in \mathbb{R}^n, \\
\max_{\mathbb{R}^n} U_0=U_0(0)=1.
\end{cases}
\end{equation}
}
\vskip0.10in 
Define 
$$
s_j(x) := \left( \frac{d_j}{2} - |x-x_j| \right)^{\frac{n-2\sigma}{2}} u(x), ~~~ |x - x_j| \leq \frac{d_j}{2}. 
$$
Since $u$ is positive and continuous in $\overline{B}_{d_j /2}(x_j)$, we can find a point $\bar{x}_j \in \overline{B}_{d_j/2 }(x_j)$ such that  
$$
s_j(\bar{x}_j) = \max_{|x -x_j| \leq \frac{d_j}{2}} s_j(x) >0. 
$$
Let $2\mu_j:= \frac{d_j}{2} - |\bar{x}_j - x_j|$.   Then
$$
0 < 2\mu_j \leq \frac{d_j}{2}~~~~~~  \textmd{and} ~~~~~~ \frac{d_j}{2} - |x-x_j| \geq \mu_j~~~~\forall ~  |x-\bar{x}_j| \leq \mu_j. 
$$
By the definition of $s_j$, we have
\begin{equation}\label{SLo3-000} 
(2 \mu_j)^{\frac{n-2\sigma}{2}} u( \bar{x}_j ) = s_j(\bar{x}_j) \geq s_j(x) \geq \mu_j^{\frac{n-2\sigma}{2}}u(x)~~~~\forall ~ |x-\bar{x}_j| \leq \mu_j. 
\end{equation}
Hence 
\begin{equation}\label{SLo3-01} 
2^{\frac{n-2\sigma}{2}} u( \bar{x}_j ) \geq u(x)~~~~~~ \forall ~ |x-\bar{x}_j| \leq \mu_j. 
\end{equation}
We also have
\begin{equation}\label{SLo3-02}
(2 \mu_j)^{\frac{n-2\sigma}{2}} u( \bar{x}_j ) = s_j(\bar{x}_j) \geq s_j(x_j) = \left( \frac{d_j}{2} \right)^{\frac{n-2\sigma}{2}} u(x_j) \to \infty ~~~~\textmd{as}  ~ j\to \infty. 
\end{equation} 
Now, we consider 
$$
v_j(y)=\frac{1}{u(\bar{x}_j)} u \left( \bar{x}_j + \frac{y}{u(\bar{x}_j)^{\frac{2}{n-2\sigma}}} \right), ~f_j(y)=\frac{1}{u(\bar{x}_j)} h \left( \bar{x}_j + \frac{y}{u(\bar{x}_j)^{\frac{2}{n-2\sigma}}} \right)~  \textmd{in}~ \Theta_j, 
$$
where 
$$
\Theta_j =\left\{y\in \mathbb{R}^n :   \bar{x}_j + \frac{y}{u(\bar{x}_j)^{\frac{2}{n-2\sigma}}} \in  B_2 \backslash \Sigma \right\}. 
$$
We extend $v_j$ to be identically $0$ outside $\Theta_j$ and $K$ to be identically $0$ outside $B_2$.   Then $v_j$ satisfies $v_j(0)=1$ and 
\begin{equation}\label{Swj01}
v_j(y) = \int_{\mathbb{R}^n} \frac{K(\bar{x}_j + u(\bar{x}_j)^{-\frac{2}{n-2\sigma}} z ) v_j(z)^{\frac{n+2\sigma}{n-2\sigma}}}{|y -z |^{n-2\sigma}} dz + f_j(y)~~~~~
\textmd{for} ~ y\in \Theta_j. 
\end{equation}
Moreover, it follows from \eqref{SLo3-01} and  \eqref{SLo3-02} that  
\begin{equation}\label{38-01234}
v_j(y) \leq 2^{\frac{n-2\sigma}{2}}~~~ \textmd{in} ~ B_{\bar{R}_j}, 
\end{equation}
where 
$$
\bar{R}_j := \mu_j u(\bar{x}_j)^{\frac{2}{n-2\sigma}} \to \infty ~~~ \textmd{as} ~ j \to \infty. 
$$
Clearly $u(\bar{x}_j) \to \infty$ as $j \to \infty$,  and $\|f_j\|_{C^1(B_{\bar{R}_j})} \to 0$ as $ j \to \infty$ since 
\begin{equation}\label{38-012}
\aligned
\|f_j\|_{C^1(B_{\bar{R}_j})} & \leq \frac{1}{u(\bar{x}_j)} \|h\|_{L^\infty(B_{\mu_j}(\bar{x}_j))} + \frac{1}{u(\bar{x}_j)^\frac{n - 2 \sigma + 2}{n - 2 \sigma}} \|\nabla h\|_{L^\infty(B_{\mu_j}(\bar{x}_j))} \\
& \leq \frac{1}{u(\bar{x}_j)} \|h\|_{L^\infty(B_{3/2})} + \frac{1}{u(\bar{x}_j)^\frac{n - 2 \sigma + 2}{n - 2 \sigma}} \|\nabla h\|_{L^\infty(B_{3/2})} \to 0 ~~~ \textmd{as} ~ j \to \infty. 
\endaligned
\end{equation}

{\it Claim 1: There exists a subsequence of $ v_j$,  still denoted by $ v_j$,  such that $ v_j$ in $C_{loc}^2(\mathbb{R}^n)$ converges to a positive function $v \in C^2(\mathbb{R}^n)$ where $v$ satisfies    
\begin{equation}\label{S3-W1}
v(y) = \int_{\mathbb{R}^n} \frac{K(0) v(z)^{\frac{n+2\sigma}{n-2\sigma}}}{|y -z |^{n-2\sigma}} dz~~~~~ \textmd{for} ~ y \in \mathbb{R}^n. 
\end{equation}
}

Since for any $R>0$ we have  $v_j(y)\leq 2^\frac{n-2\sigma}{2}$ in $B_R$ for all large $j$, by regularity results in \cite[Section 2.1]{JLX17}  there exists $v \geq 0$ such that,  up to a subsequence,    
$$
v_j \to v~~~~~ \textmd{in}  ~C_{loc}^{2}(\mathbb{R}^n). 
$$
 Clearly   $v(0)=1$.    
To show  that $v$ satisfies the integral equation \eqref{S3-W1},  we will use some arguments of \cite[Proposition 2.9]{JLX17}.    Write \eqref{Swj01} as  
\begin{equation}\label{S33-00133vv}
v_j(y)=\int_{B_r} \frac{K(\bar{x}_j + u(\bar{x}_j)^{-\frac{2}{n-2\sigma}} z ) v_j(z)^{\frac{n+2\sigma}{n-2\sigma}}}{|y -z |^{n-2\sigma}} dz + F_j(r,y)~~~~~
\textmd{for} ~ y\in \Theta_j, 
\end{equation} 
where
$$
F_j(r,y)=\int_{B_r^c} \frac{K(\bar{x}_j + u(\bar{x}_j)^{-\frac{2}{n-2\sigma}} z ) v_j(z)^{\frac{n+2\sigma}{n-2\sigma}}}{|y -z |^{n-2\sigma}} dz + f_j(y).
$$
Then,  for $y\in B_{r/2}$  we have 
$$
\aligned
F_j(r,y)&=\int_{B_r^c} \frac{K(\bar{x}_j + u(\bar{x}_j)^{-\frac{2}{n-2\sigma}} z ) v_j(z)^{\frac{n+2\sigma}{n-2\sigma}}}{|z |^{n-2\sigma}} \frac{|z|^{n-2\sigma}}{|y-z|^{n-2\sigma}} dz + f_j(y)\\
&\leq C \int_{B_r^c} \frac{K(\bar{x}_j + u(\bar{x}_j)^{-\frac{2}{n-2\sigma}} z ) v_j(z)^{\frac{n+2\sigma}{n-2\sigma}}}{|z |^{n-2\sigma}} dz + \|f_j\|_{L^\infty(B_{r/2})}\\
&\leq  C v_j(0) + \|f_j\|_{L^\infty(B_{r/2})}  
\endaligned
$$
for all large $j$.  Similarly,  for $y\in B_{r/2}$,
$$
|\nabla_y F_j(r,y)|\leq  C(r)  v_j(0) + \|\nabla f_j\|_{L^\infty(B_{r/2})}. 
$$
These together with \eqref{38-01234} and \eqref{38-012} imply that  $\|F_j(r,\cdot)\|_{C^1(B_{r/2})}\leq C(r)$ for all $j$ large.  Thus,  after passing to a subsequence, $F_j(r,\cdot)\to F(r,\cdot)$ in $C^{1/2}(B_{r/2})$. Hence, letting $j \to \infty$ in \eqref{S33-00133vv} we obtain 
\begin{equation}\label{hRywyint}
F(r,y)=v(y)-\int_{B_r} \frac{K(0) v(z)^{\frac{n+2\sigma}{n-2\sigma}}}{|y -z |^{n-2\sigma}} dz ~~~~~ \textmd{for} ~  y\in B_{r/2}. 
\end{equation}
Furthermore,  $F(r,y) \geq 0$ and it is non-increasing in $r$. For  $r >>|y|$,
$$
\aligned
\frac{r^{n-2\sigma}}{(r+|y|)^{n-2\sigma}}(F_j(r,0)-f_j(0)) & \leq F_j(r,y)-f_j(y)\\
& = \int_{B_r^c} \frac{K(\bar{x}_j + u(\bar{x}_j)^{-\frac{2}{n-2\sigma}} z ) v_j(z)^{\frac{n+2\sigma}{n-2\sigma}} }{|y|^{n-2\sigma}} \frac{|y|^{n-2\sigma}}{|y -z |^{n-2\sigma}} dz \\
&\leq \frac{r^{n-2\sigma}}{(r -|y|)^{n-2\sigma}}(F_j(r,0)-f_j(0)).
\endaligned
$$
Let $j$ tend to $\infty$, we get
$$
\frac{r^{n-2\sigma}}{(r+|y|)^{n-2\sigma}}F(r,0) \leq F(r ,y)\\
\leq \frac{r^{n-2\sigma}}{(r -|y|)^{n-2\sigma}}F(r ,0), 
$$
which leads to $\lim_{r\to\infty} F(r,y)=\lim_{r\to\infty} F(r,0)=:C_0\geq  0$.  Sending $r \to +\infty$ in \eqref{hRywyint}  and using Lebesgue's  monotone convergence theorem,  we obtain 
$$
v(y)=\int_{\R^n} \frac{K(0) v(z)^{\frac{n+2\sigma}{n-2\sigma}}}{|y -z |^{n-2\sigma}} dz+C_0~~~~~ \textmd{for} ~ y \in \mathbb{R}^n. 
$$
If $C_0 > 0$, then $v(y) \geq C_0 >0$ for any  $y\in\R^n$ and hence 
$$
1= v(0)\geq \int_{\R^n} \frac{K(0) C_0^{\frac{n+2\sigma}{n-2\sigma}}}{|z |^{n-2\sigma}} dz = +\infty.
$$
This is impossible.  Therefore,  $C_0=0$ and Claim 1 is established.

Since $v(0)=1$,  by the classification results in \cite{CLO06} or \cite{Li04}, $v$ must be of the form 
\begin{equation}\label{Class301}
v(y) =  \left( \frac{a_0}{1+ a_0^2|y - y_0|^2 } \right)^{\frac{n-2\sigma}{2}}
\end{equation}
for  some $y_0 \in \mathbb{R}^n$ and some $a_0 \geq 1$. Obviously, $v$ has an absolute maximum at $y_0$.  It  implies that $v_j(y)$ must have a local maximum at a point $y_j$ near $y_0$ when $j$ is large. Replacing  $\bar{x}_j$ by 
$$
\tilde{x}_j := \bar{x}_j + u(\bar{x}_j)^{-\frac{2}{n-2\sigma}} y_j, 
$$
then $\{\tilde{x}_j\}$ are local maximum points of $u$ for large $j$. Moreover, by \eqref{SLo3-02} we have for large $j$ that 
\begin{equation}\label{Ma-3-8787}
u(\tilde{x}_j) = v_j(y_j) u(\bar{x}_j) \geq \frac{a_0^{(n-2\sigma)/2}}{2} u(\bar{x}_j)  \geq \frac{a_0^{(n-2\sigma)/2}}{2} u(x_j)
\end{equation} 
and $\tilde{x}_j \in B_{\mu_j /2 }(\bar{x}_j) \subset B_{d_j /2}(x_j)$ which implies $\frac{1}{2}d_j \leq \textmd{dist}(\tilde{x}_j, \Sigma) \leq \frac{3}{2}d_j$ and $B_{\mu_j /2 }(\tilde{x}_j) \subset B_{\mu_j} (\bar{x}_j)$. Consequently, 
$$
\textmd{dist}(\tilde{x}_j, \Sigma)^{\frac{n-2\sigma}{2}} u(\tilde{x}_j) \to \infty ~~ \textmd{and} ~~ \tilde{x}_j \to 0 ~~~~~~~ \textmd{as} ~  j\to \infty. 
$$ 
Furthermore, by  \eqref{SLo3-01},  \eqref{Ma-3-8787} and \eqref{SLo3-02} we have 
$$
u(\tilde{x}_j) \geq \frac{1}{2} \left( \frac{a_0}{2} \right)^{\frac{n-2\sigma}{2}} u(x) ~~~~~~ \forall ~ |x-\tilde{x}_j| \leq \frac{\mu_j}{2} 
$$
and 
$$
R_j:= \frac{1}{2} \mu_j u(\tilde{x}_j)^{\frac{2}{n-2\sigma}} \to \infty ~~~ \textmd{as} ~ j \to \infty. 
$$
Using the proof of Claim 1 we know that the functions $u(\tilde{x}_j)^{-1} u( \tilde{x}_j +  u(\tilde{x}_j)^{-\frac{2}{n-2\sigma}} y)$ converge in $C_{loc}^2(\mathbb{R}^n)$, after passing a subsequence,  to a positive function $U_0 \in C^2(\mathbb{R}^n)$ which satisfies \eqref{S3-W1-0376}.  It follows from the classification results in \cite{CLO06} or  \cite{Li04} that, modulo a positive constant,  
\begin{equation}\label{Class37-90401}
U_0(y) = \left( \frac{1}{1+ |y|^2 } \right)^{\frac{n-2\sigma}{2}}. 
\end{equation}
Thus the conclusion of Step 1 is proved.  From now on, we consider $x_j$ as $\tilde{x}_j$.

\vskip0.10in 

{\it Step 2. We show  $\nabla K(0)=0$ assuming $K \in C^1(B_2)$ and $n > 2\sigma$. }

\vskip0.10in 
 
 Suppose $\nabla K(0) \neq 0 $. We may assume without  loss of generality that 
$$
 \nabla K(0) =e =(1, 0, \dots, 0).
$$
Let 
\begin{equation}\label{21SS3-01}
w_j(y)=M_j^{-1} u \left( x_j + M_j^{-\frac{2}{n-2\sigma}} (y-Re) \right), ~h_j(y)=M_j^{-1}  h \left( x_j + M_j^{-\frac{2}{n-2\sigma}} (y-Re) \right) ~  \textmd{in}~ \Omega_j, 
\end{equation} 
where $M_j=u(x_j)$,  $R>10$ is a large positive constant to be determined later,  and 
\begin{equation}\label{hhmm-9e7=5}
\Omega_j =\left\{y\in \mathbb{R}^n :   x_j + M_j^{-\frac{2}{n-2\sigma}} (y-Re)  \in  B_2 \backslash \Sigma \right\}. 
\end{equation} 
By Step 1, $w_j$ converges in $C^2$ norm to the bubble $U_1=U_0(\cdot - Re)$ on every compact subset of $\mathbb{R}^n$. We also extend $w_j$ to be identically $0$ outside $\Omega_j$ and $K$ to be identically $0$ outside $B_2$. For $\lambda >0$, let  
\begin{equation}\label{21SS3-022}
w_j^\lambda(y) =\left( \frac{\lambda}{|y|} \right)^{n-2\sigma} w_j \left( \frac{\lambda^2 y}{|y|^2} \right),~~h_j^\lambda(y) =  \left( \frac{\lambda}{|y|} \right)^{n-2\sigma} h_j \left( \frac{\lambda^2 y}{|y|^2} \right)
\end{equation}
and let 
\begin{equation}\label{21SS3-033}
K_j(y) =K \left( x_j + M_j^{-\frac{2}{n-2\sigma}} (y-Re) \right). 
\end{equation}
Then 
\begin{equation}\label{21SS3-044}
w_j(y) = \int_{\mathbb{R}^n} \frac{K_j(z) w_j(z)^{\frac{n+2\sigma}{n-2\sigma}}}{|y -z |^{n-2\sigma}} dz + h_j(y)~~~~~\textmd{for} ~ y\in \Omega_j. 
\end{equation}
By \eqref{ABC-01},  $w_j^\lambda$ satisfies 
\begin{equation}\label{21SS3-055}
w_j^\lambda(y) = \int_{\mathbb{R}^n} \frac{K_j(z^\lambda) w_j^\lambda(z)^{\frac{n+2\sigma}{n-2\sigma}}}{|y -z |^{n-2\sigma}} dz + h_j^\lambda(y)~~~~~\textmd{for} ~ y \in \Omega_j \backslash   \overline{B}_{\lambda},  
\end{equation}
where $z^\lambda = \frac{\lambda^2 z}{|z|^2}$ is the inversion of $z$ with respect to $\partial B_\lambda$.

The following lemma gives the estimates on the difference between $w_j$ and $w_j^\lambda$ with $\lambda=R-2$ or $R+2$.  

\begin{lemma}\label{LAn-01}
Let $\lambda_0= R-2$ and $\lambda_1 =R +2$. Then there exist $\varepsilon_0=\varepsilon_0(n, \sigma, \min_{B_2\backslash \Sigma}{u}, K, R)>0$ and $j_0=j_0(n, \sigma, \min_{B_2\backslash \Sigma}{u},  K,  R)>1$ such that for all $j \geq j_0$, 
\begin{equation}\label{21SS3&8-001}
w_j(y) - w_j^{\lambda_0}(y) \geq \varepsilon_0 (|y| - \lambda_0) |y|^{2\sigma-1-n} + \varepsilon_0 M_j^{-1} (\lambda_0^{2\sigma-n} - |y|^{2\sigma-n}), ~~~ y \in \Omega_j \backslash  \overline{B}_{\lambda_0}.
\end{equation}
Moreover, there exists $y^* \in B_{2\lambda_1} \backslash  \overline{B}_{\lambda_1}$ such that for all $j \geq j_0$, 
\begin{equation}\label{21SS3&8-003}
w_j(y^*) -  w_j^{\lambda_1}(y^*) \leq -\varepsilon_0. 
\end{equation}
\end{lemma} 

\begin{proof}
Since $w_j$ converges in $C^2$ norm to $U_1=U_0(\cdot - Re)$ on any compact subset of $\mathbb{R}^n$,  by Lemma \ref{InK-03} there exists $\varepsilon_1=\varepsilon_1(n, \sigma, R)>0$ such that for any fixed $R_1 \gg R$, 
\begin{equation}\label{S09hmyf-407=01}
w_j(y) - w_j^{\lambda_0}(y) \geq \varepsilon_1 (|y| - \lambda_0) |y|^{2\sigma-1-n},~~~~~~ \lambda_0 < |y| \leq R_1 
\end{equation}
for sufficiently large $j$.  Clearly,  Lemma \ref{InK-03} also implies that \eqref{21SS3&8-003} is true.   

Next we show that \eqref{21SS3&8-001} holds for $y \in \Omega_j \backslash B_{R_1}$.  Firstly,  it is easy to see that there exists a small $\varepsilon_2=\varepsilon_2(n, \sigma, R)>0$ such that for any $R_1 \gg R$, 
\begin{equation}\label{S09hmyf-407=02}
|  U_1(y) - |y|^{2\sigma-n}  |  \leq \frac{\varepsilon_2}{2} |y|^{2\sigma-n},~~~~~~ |y| \geq R_1
\end{equation}
and
\begin{equation}\label{S09hmyf-407=03}
 U_1^{\lambda_0}(y)  \leq  (1-3\varepsilon_2)  |y|^{2\sigma-n},~~~~~~ |y| \geq R_1. 
\end{equation}
Consequently,  we obtain that   for large $j$, 
\begin{equation}\label{S09hmyf-407=04}
w_j^{\lambda_0}(y)  \leq  (1- 2\varepsilon_2)  |y|^{2\sigma-n},~~~~~~ |y| \geq R_1. 
\end{equation}
On the other hand,  because $w_j$ converges  to $U_1$ in $C^2(B_{R_1})$,  for any $y\in \Omega_j \backslash B_{R_1}$ we have that when $j$ is sufficiently large, 
\begin{equation}\label{S09hmyf-408=01}
\aligned
\Big( 1 - \frac{\varepsilon_2}{8} \Big) w_j(y) & \geq \Big( 1 - \frac{\varepsilon_2}{8} \Big)  \int_{B_{R_1}} \frac{K_j(z) w_j(z)^{\frac{n+2\sigma}{n-2\sigma}}}{|y - z|^{n-2\sigma}} dz \\
& \geq \Big( 1 - \frac{\varepsilon_2}{4} \Big)  \int_{B_{R_1}} \frac{ K(0)w_j(z)^{\frac{n+2\sigma}{n-2\sigma}}}{|y - z|^{n-2\sigma}} dz \\
& \geq \Big( 1 - \frac{\varepsilon_2}{2} \Big)  \int_{B_{R_1}} \frac{ K(0)U_1(z)^{\frac{n+2\sigma}{n-2\sigma}}}{|y - z|^{n-2\sigma}} dz  \\
& \geq  \Big( 1 - \frac{\varepsilon_2}{2} \Big)  U_1(y) -   \int_{ \{ |z|\geq R_1 \} } \frac{K(0)U_1(z)^{\frac{n+2\sigma}{n-2\sigma}}}{|y-z|^{n-2\sigma}} dz. 
\endaligned
\end{equation} 
A simple computation yields that for $y\in \Omega_j \backslash B_{R_1}$, 
$$
\aligned
\int_{ \{ |z|\geq R_1 \} } \frac{U_1(z)^{\frac{n+2\sigma}{n-2\sigma}}}{|y-z|^{n-2\sigma}} dz & \leq \bigg( \int_{ \{ |z|\geq R_1, |y -z|\geq \frac{|y|}{2} \} }  + \int_{ \{ |z|\geq R_1, |y -z| \leq \frac{|y|}{2} \} } \bigg) \frac{U_1(z)^{\frac{n+2\sigma}{n-2\sigma}}}{|y-z|^{n-2\sigma}} dz \\
& \leq \frac{C}{|y|^{n-2\sigma}} \int_{ \{ |z|\geq R_1 \} } \frac{dz}{(1 + |z|^2)^{(n+2\sigma)/2}}  + \frac{C}{|y|^n} \\
& \leq \frac{C}{|y|^{n-2\sigma}} \cdot \frac{1}{R_1^{2\sigma}}
\endaligned
$$
for some positive constant $C = C(n, \sigma, R)$.  This, together with \eqref{S09hmyf-407=02} and \eqref{S09hmyf-408=01}, gives 
\begin{equation}\label{S09hmyf-408=02} 
\aligned
w_j(y) & \geq \frac{\varepsilon_2}{8} w_j(y) +  ( 1 -  \varepsilon_2) |y|^{2\sigma-n} + \frac{\varepsilon_2}{4} |y|^{2\sigma-n} - \frac{C}{|y|^{n-2\sigma}}  \cdot  \frac{1}{R_1^{2\sigma}}  \\
& \geq \frac{\varepsilon_2}{8} w_j(y) + ( 1 -  \varepsilon_2) |y|^{2\sigma-n}  ~~~~~~ \textmd{for}  ~ y\in \Omega_j \backslash B_{R_1}
\endaligned
\end{equation}
by choosing $R_1\gg R$  large enough and then  fixing it.  Moreover, from  the equation \eqref{Int} we have $u(x) \geq 4^{2\sigma-n} \int_{B_2} K(y) u(y)^{\frac{n+2\sigma}{n-2\sigma}} dy=:A_0 > 0$ for all  $x\in B_2 \backslash \Sigma$, so by the definition of $w_j$ we obtain    
\begin{equation}\label{S09hmyf-408=03} 
w_j(y)  \geq \varepsilon_3 M_j^{-1} (\lambda_0^{2\sigma-n} - |y|^{2\sigma-n}) ~~~~~~ \textmd{for}  ~  y \in \Omega_j \backslash B_{R_1}.  
\end{equation}
with some constant $\varepsilon_3=\varepsilon_3(n, \sigma, A_0, R)$.   It follows from \eqref{S09hmyf-407=04}, \eqref{S09hmyf-408=02} and \eqref{S09hmyf-408=03} that  there exists a small $\varepsilon_0>0$ such that for large $j$, 
$$
w_j(y) - w_j^{\lambda_0}(y)  \geq \varepsilon_0 |y|^{2\sigma-n} + \varepsilon_0 M_j^{-1} (\lambda_0^{2\sigma-n} - |y|^{2\sigma-n})~~~~~~ \textmd{for}   ~  y \in \Omega_j \backslash B_{R_1}. 
$$
This together with \eqref{S09hmyf-407=01} implies that \eqref{21SS3&8-001} holds by choosing $\varepsilon_0$ sufficiently small.   Lemma \ref{LAn-01}  is established.  
\end{proof}

Next we will use the moving sphere method for $\lambda \in [R-2, R+2]$ to derive a contradiction. Denote
\begin{equation}\label{21PAI01}
\Pi_j:=  \left\{y\in \mathbb{R}^n :   x_j + M_j^{-\frac{2}{n-2\sigma}} (y-Re)  \in  B_{1} \backslash \Sigma \right\} \subset \Omega_j 
\end{equation} 
and
\begin{equation}\label{21PAI02}
\Sigma_j:=\left\{y\in \mathbb{R}^n :  x_j + M_j^{-\frac{2}{n-2\sigma}} (y-Re)   \in  \Sigma \right\}. 
\end{equation}
Note that we have $\mathcal{L}^n (\Sigma_j)=0$ due to $\mathcal{L}^n (\Sigma)=0$.   It follows from the same arguments as in Lemma 3.1 of \cite{JX19}  that  for all large $j$, there holds
\begin{equation}\label{3-Local3}
h_j^{\lambda}(y) \leq h_j(y)~~~~~~ \forall ~  y\in \Pi_j \backslash \overline{B}_\lambda,  ~ \lambda\in [R-2, R+2]. 
\end{equation}
Let $\varphi_\lambda(y) = w_j(y) - w_j^\lambda(y)$, where we omit $j$ in the notation for brevity.   By \eqref{ABC-0e571}, \eqref{3-Local3} and $K_j \equiv 0$ on $\Omega_j^c \backslash \Sigma_j$,    we have for $\lambda\in [R-2, R+2]$ and $y\in \Pi_j \backslash \overline{ B}_\lambda$ that 
\begin{equation}\label{yh38-01}
\aligned
\varphi_\lambda(y) & \geq \int_{B_\lambda^c} G(0,\lambda; y, z) \left( K_j(z) w_j(z)^{\frac{n+2\sigma}{n-2\sigma}} - K_j(z^\lambda) w_j^\lambda(z)^{\frac{n+2\sigma}{n-2\sigma}} \right) dz  \\
& =  \int_{B_\lambda^c} G(0,\lambda; y, z) K_j(z)\left( w_j(z)^{\frac{n+2\sigma}{n-2\sigma}} -  w_j^\lambda(z)^{\frac{n+2\sigma}{n-2\sigma}} \right) dz \\
& ~~~~~  + \int_{B_\lambda^c} G(0,\lambda; y, z) \left(  K_j(z) - K_j(z^\lambda) \right) w_j^\lambda(z)^{\frac{n+2\sigma}{n-2\sigma}} dz \\
&= \int_{\Pi_j \backslash B_\lambda} G(0, \lambda; y, z) b_\lambda(z) \varphi_\lambda(z) dz + J_\lambda(y) -\Phi_\lambda(y),
\endaligned
\end{equation}
where 
\begin{equation}\label{SL4-312=01}
b_\lambda(y)= K_j(y) \frac{ w_j(y)^{\frac{n+2\sigma}{n-2\sigma}} -  w_j^\lambda(y)^{\frac{n+2\sigma}{n-2\sigma}} }{w_j(y) - w_j^\lambda(y)},
\end{equation}
\begin{equation}\label{SL4-312=02}
\Phi_\lambda(y)= \int_{\Omega_j \backslash B_\lambda} G(0, \lambda; y, z) \left( K_j(z^\lambda) - K_j(z) \right) w_j^\lambda(z)^{\frac{n+2\sigma}{n-2\sigma}} dz 
\end{equation}
and
\begin{equation}\label{SL4-312=03}
\aligned
J_\lambda(y) & = \int_{\Omega_j \backslash \Pi_j} G(0,\lambda; y, z) K_j(z)\left( w_j(z)^{\frac{n+2\sigma}{n-2\sigma}} -  w_j^\lambda(z)^{\frac{n+2\sigma}{n-2\sigma}} \right) dz\\
& ~~~~~ - \int_{\Omega_j^c} G(0,\lambda; y, z) K_j(z^\lambda) w_j^\lambda(z)^{\frac{n+2\sigma}{n-2\sigma}} dz.
\endaligned
\end{equation}
Note that $b_\lambda(y)$  is always non-negative.  Thus, $\varphi_\lambda $ satisfies the integral inequality 
\begin{equation}\label{SL4-312=04}
\varphi_\lambda(y) +  \Phi_\lambda(y) \geq \int_{\Pi_j \backslash B_\lambda} G(0, \lambda; y, z) b_\lambda(z) \varphi_\lambda(z) dz + J_\lambda(y),  ~~~~~ y \in \Pi_j \backslash \overline{B}_\lambda, 
\end{equation}
where $\lambda \in [R-2, R+2]$. 

Now we show some estimates for $\Phi_\lambda$ on $\Pi_j \backslash \overline{B}_\lambda$ in order to apply the moving sphere method to the function $\varphi_\lambda + \Phi_\lambda$.  We first need to establish the estimates on $K_j(z^\lambda) - K_j(z)$ and on $w_j^\lambda(z)$ for $z \in \Omega_j \backslash B_\lambda$.  For this purpose we define a special domain  in $\Omega_j \backslash B_\lambda$. Let  
$$
D_\lambda= \{z \in \mathbb{R}^n : \lambda < |z| < 2\lambda, z_1 > 2 |z^{\prime}| ~\textmd{where} ~ z^{\prime}=(z_2, \dots, z_n) \}.
$$ 

{\it Claim 2: There exists a constant $C>0$ independent of $\lambda$ such that for all large $j$,   
\begin{equation}\label{S4^314=01}
K_j(z^\lambda) - K_j(z) \leq  -C M_j^{-\frac{2}{n-2\sigma}} (|z| - \lambda), ~~~~~ z \in D_\lambda, 
\end{equation}
\begin{equation}\label{S4^314=02}
|K_j(z^\lambda) - K_j(z)| \leq C M_j^{-\frac{2}{n-2\sigma}} (|z| - \lambda), ~~~~~ z \in \Omega_j \backslash B_\lambda. 
\end{equation}}
\begin{proof}
The first estimate follows from the mean value theorem and  $\lim_{x \to 0} \nabla K (x) = \nabla K(0) = e$. The second estimate follows from the mean value theorem.  
\end{proof}

\begin{lemma}\label{Le-hm-074}
There exist constants $C_1, C_2>0$ independent of $\lambda$ such that for all large $j$, 
\begin{equation}\label{S4_314=03}
C_1 |z|^{2\sigma - n} \leq w_j^\lambda(z) \leq C_2 |z|^{2\sigma - n} ~~~~~~ \textmd{for}  ~ z \in \Omega_j \backslash (B_\lambda \cup D_\lambda)
\end{equation}
and
\begin{equation}\label{S4_314=04}
C_1 \left( \frac{\lambda}{|z|} \right)^{n-2\sigma} \left( \frac{1}{1 + (z_1 - \lambda)^2 + |z^{\prime}|^2}\right)^{\frac{n-2\sigma}{2}} \leq w_j^\lambda(z) \leq 2 ~~~~~~ \textmd{for}  ~ z \in D_\lambda. 
\end{equation}
\end{lemma} 
\begin{proof}
Since the functions $w_j(y)$ converge in $C_{loc}^2(\mathbb{R}^n)$ to $U_1(y)=(1+ |y - Re|^2)^{-\frac{n-2\sigma}{2}}$,  for $z\in D_\lambda$ and for large $j$ we have 
\begin{equation}\label{SAF4_315=001}
 \frac{1}{2} \left(\frac{\lambda}{|z|} \right)^{n-2\sigma} \left( \frac{1}{1 + \left| \frac{\lambda^2 z_1}{|z|^2} - R \right|^2 +  \left| \frac{\lambda^2 z^{\prime}}{|z|^2} \right|^2}\right)^{\frac{n-2\sigma}{2}} \leq \left(\frac{\lambda}{|z|} \right)^{n-2\sigma} w_j \left(  \frac{\lambda^2 z}{|z|^2} \right)  = w_j^\lambda(z) \leq 2. 
\end{equation}
Note that $|\lambda - R| \leq 2$ and for $ |z| > \lambda$, 
$$
\aligned
\left| \frac{\lambda^2 z_1}{|z|^2} - \lambda \right|^2 & = \frac{\lambda^2}{|z|^2}  \left(\frac{\lambda^2}{|z|^2} z_1^2 - 2 z_1 \lambda + \frac{|z|^2}{\lambda^2} \lambda^2 \right) \\
& \leq   (z_1 - \lambda)^2 + (|z|^2 - \lambda^2) \left( 1 - \frac{z_1^2}{|z|^2}\right)  \\
&  \leq  (z_1 - \lambda)^2 + |z^{\prime}|^2, 
\endaligned
$$
from which we obtain
$$
1 + \left| \frac{\lambda^2 z_1}{|z|^2} - R \right|^2 +  \left| \frac{\lambda^2 z^{\prime}}{|z|^2} \right|^2 \leq 10 (1 + (z_1 - \lambda)^2 + |z^{\prime}|^2 ). 
$$
This together with \eqref{SAF4_315=001} gives the proof of \eqref{S4_314=04}.  

Next we establish  \eqref{S4_314=03}. Firstly one can verify  the following elementary estimate:
\begin{equation}\label{SAF4_315=002}
|z^\lambda - Re|^2 \geq C \lambda^2,  ~~~~~~ z \in \Omega_j \backslash (B_\lambda \cup D_\lambda)
\end{equation}
for some constant $C>0$ independent of $\lambda$. Because $w_j(y)$ converges  in $C_{loc}^2(\mathbb{R}^n)$ to $U_1(y)=(1+ |y - Re|^2)^{-\frac{n-2\sigma}{2}}$, we have for large $j$ that 
\begin{equation}\label{SAF4_315=003}
 \frac{1}{2} \left(\frac{\lambda}{|z|} \right)^{n-2\sigma} \left( \frac{1}{1 + \left| z^\lambda - Re\right|^2}\right)^{\frac{n-2\sigma}{2}} \leq  w_j^\lambda(z) \leq 2  \left(\frac{\lambda}{|z|} \right)^{n-2\sigma} \left( \frac{1}{1 + \left| z^\lambda - Re\right|^2}\right)^{\frac{n-2\sigma}{2}}, 
\end{equation}
where $z \in \Omega_j \backslash B_\lambda$. By \eqref{SAF4_315=002}  we get 
$$
w_j^\lambda(z) \leq 2  \left(\frac{\lambda}{|z|} \right)^{n-2\sigma} \left( \frac{1}{1 + C \lambda^2}\right)^{\frac{n-2\sigma}{2}} \leq C_2 |z|^{2\sigma -n},~~~~~~ z \in \Omega_j \backslash (B_\lambda \cup D_\lambda)
$$
for another constant $C_2 >0$ independent of $\lambda$. On the other hand, using $|z^\lambda - Re|\leq 3\lambda$ for  $|z| \geq  \lambda$  we obtain 
$$
w_j^\lambda(z) \geq  \frac{1}{2} \left(\frac{\lambda}{|z|} \right)^{n-2\sigma} \left( \frac{1}{1 + 9\lambda^2}\right)^{\frac{n-2\sigma}{2}} \geq C_1 |z|^{2\sigma -n},~~~~~~ z \in \Omega_j \backslash (B_\lambda \cup D_\lambda)
$$
for some constant $C_1 >0$ independent of $\lambda$.  Lemma \ref{Le-hm-074} is proved.   
\end{proof} 

Lemma \ref{Le-hm-074} tells us that $w_j^\lambda$ is bigger on $D_\lambda$ and is smaller on $\Omega_j \backslash (B_\lambda \cup D_\lambda)$, while $K_j(z^\lambda) - K_j(z)$ is negative on $D_\lambda$ by Claim 2. These combined with Lemma \ref{InK-01} will lead to a crucial fact that $\Phi_\lambda$ is non-positive. Indeed, the positive parts of the integral in \eqref{SL4-312=02} can be controlled by the negative one.  More precisely, we have 

\begin{lemma}\label{LS4-312=01} 
There  exists a constant $C>0$ independent of  $\lambda$ such that for all large $j$, 
$$ 
\Phi_\lambda (y) \leq 
\begin{cases}
- C M_j^{-\frac{2}{n-2\sigma}} (|y| - \lambda) \lambda^{-n} \log \lambda  ~~~~~~ & \textmd{for} ~ y \in B_{4\lambda} \backslash \overline{B}_\lambda, \\
-C M_j^{-\frac{2}{n-2\sigma}} |y|^{2\sigma -n} \lambda^{1-2\sigma} \log \lambda ~~~~~~ & \textmd{for} ~ y \in \Pi_j \backslash \overline{B}_{4\lambda}
\end{cases}
$$
and
$$ 
|\Phi_\lambda (y)| \leq 
\begin{cases}
C M_j^{-\frac{2}{n-2\sigma}} (|y| - \lambda) \lambda^{2\sigma}  ~~~~~~ & \textmd{for} ~ y \in B_{4\lambda} \backslash \overline{B}_\lambda, \\
C M_j^{-\frac{2}{n-2\sigma}} |y|^{2\sigma -n} \lambda^{n+1} ~~~~~~ & \textmd{for} ~ y \in  \Pi_j  \backslash \overline{B}_{4\lambda}. 
\end{cases}
$$
\end{lemma}
\begin{proof}
Let 
\begin{equation}\label{SL4-312=05}
Q_\lambda(z) = \left( K_j(z^\lambda) - K_j(z) \right) w_j^\lambda(z)^{\frac{n+2\sigma}{n-2\sigma}} ~~~~~ \textmd{for} ~ z \in\Omega_j \backslash B_\lambda. 
\end{equation}
Then $\Phi_\lambda$ can be written as 
$$
\Phi_\lambda(y)= \int_{\Omega_j \backslash B_\lambda} G(0, \lambda; y, z) Q_\lambda(z) dz.  
$$
By Claim 2 and Lemma \ref{Le-hm-074}  we have 
\begin{equation}\label{SAF4_315=004}
Q_\lambda(z) \leq -C M_j^{-\frac{2}{n-2\sigma}} (|z| - \lambda)\left( \frac{1}{1 + (z_1 - \lambda)^2 + |z^{\prime}|^2}\right)^{\frac{n + 2\sigma}{2}} ~~~~~~ \textmd{for} ~ z\in D_\lambda
\end{equation}
and 
\begin{equation}\label{SAF4_315=005}
|Q_\lambda(z)| \leq 
\begin{cases}
C M_j^{-\frac{2}{n-2\sigma}} (|z| - \lambda) ~~~~~~ & \textmd{for} ~ z\in D_\lambda, \\
C M_j^{-\frac{2}{n-2\sigma}} (|z| - \lambda) |z|^{-2\sigma-n} ~~~~~~ & \textmd{for} ~ z \in \Omega_j \backslash (B_\lambda \cup D_\lambda). 
\end{cases} 
\end{equation}
These lead to 
\begin{equation}\label{SAF4_315=006}
\aligned
\Phi_\lambda(y) & \leq - C M_j^{-\frac{2}{n-2\sigma}} \int_{D_\lambda} G(0, \lambda; y, z) (|z| - \lambda)\left( \frac{1}{1 + (z_1 - \lambda)^2 + |z^{\prime}|^2}\right)^{\frac{n + 2\sigma}{2}}  dz\\
& ~~~~~ + C M_j^{-\frac{2}{n-2\sigma}} \int_{\Omega_j \backslash (B_\lambda \cup D_\lambda)} G(0, \lambda; y, z) (|z| - \lambda) |z|^{-2\sigma-n}dz. 
\endaligned
\end{equation}
We consider the following two cases separately. 

\vskip0.10in 

{\it Case 1: } $\lambda <  |y| \leq 4\lambda$.  When $ z\in D_\lambda$ and $|y - z| < \frac{1}{3}(|y| - \lambda)$, by Lemma \ref{InK-01} we have
$$
G(0, \lambda; y, z) \geq C |y -z|^{2\sigma - n} \geq C \frac{(|y| - \lambda)^2}{(|y| - \lambda)^{n-2\sigma+2}} \geq C \frac{(|y|- \lambda)(|z| -\lambda)}{\lambda^{n-2\sigma+2}}. 
$$
When $ z\in D_\lambda$ and $|y - z| \geq  \frac{1}{3}(|y| - \lambda)$, by Lemma \ref{InK-01} we also have
$$
G(0, \lambda; y, z) \geq C \frac{(|y| - \lambda)(|z|^2 - \lambda^2)}{\lambda |y-z|^{n-2\sigma+2}} \geq C \frac{(|y|- \lambda)(|z| -\lambda)}{\lambda^{n-2\sigma+2}}. 
$$
Thus, in either case we have
\begin{equation}\label{SAF4_315=007}
G(0, \lambda; y, z) \geq C \frac{(|y|- \lambda)(|z| -\lambda)}{\lambda^{n-2\sigma+2}}. 
\end{equation}
To estimate $G(0, \lambda; y, z)$ for $z \in \Omega_j \backslash B_\lambda$, we define three sets 
$$
\mathcal{A}_1 = \{z \in \Omega_j \backslash B_\lambda: |z -y|< (|y| - \lambda)/3 \}, 
$$
$$
\mathcal{A}_2 = \{ z \in \Omega_j \backslash B_\lambda: |z - y| \geq (|y| - \lambda)/3 ~ \textmd{and} ~ |z|\leq 8\lambda \}, 
$$
$$
\mathcal{A}_3 = \{z\in \Omega_j \backslash B_\lambda: |z| \geq  8\lambda \}. 
$$
Then,  from Lemma \ref{InK-01}  we have
\begin{equation}\label{SAF4_315=008}
G(0, \lambda; y, z) \leq
\begin{cases}
C |y -z|^{2\sigma -n}  ~~~~~~ & \textmd{if} ~z\in \mathcal{A}_1, \\
C \frac{(|y| - \lambda)(|z|^2 - \lambda^2)}{\lambda|y -z|^{n-2\sigma+2}} \leq C \frac{(|y| - \lambda)(|z| - \lambda)}{|y -z|^{n-2\sigma+2}}~~~~~~ & \textmd{if} ~z\in \mathcal{A}_2,\\
C \frac{(|y| - \lambda)(|z|^2 - \lambda^2)}{\lambda|y -z|^{n-2\sigma+2}} \leq C \frac{|y| - \lambda}{\lambda} |z|^{2\sigma-n}~~~~~~ & \textmd{if} ~z\in \mathcal{A}_3. 
\end{cases}
\end{equation}
Hence, combining \eqref{SAF4_315=006}, \eqref{SAF4_315=007} and \eqref{SAF4_315=008} we get
\begin{equation}\label{SAF4_315=009}
\aligned
\Phi_\lambda(y) & \leq - C M_j^{-\frac{2}{n-2\sigma}} \int_{D_\lambda} \frac{(|y|- \lambda)(|z| -\lambda)^2}{\lambda^{n-2\sigma+2}} \left( \frac{1}{1 + (z_1 - \lambda)^2 + |z^{\prime}|^2}\right)^{\frac{n + 2\sigma}{2}}  dz\\
& ~~~~~ + C M_j^{-\frac{2}{n-2\sigma}} \int_{\mathcal{A}_1 \backslash D_\lambda} |y -z|^{2\sigma -n}  (|z| - \lambda) |z|^{-2\sigma-n}dz\\
& ~~~~~ + C M_j^{-\frac{2}{n-2\sigma}} \int_{\mathcal{A}_2 \backslash D_\lambda}\frac{(|y| - \lambda)(|z| - \lambda)^2}{|y -z|^{n-2\sigma+2}} |z|^{-2\sigma-n}dz\\
& ~~~~~ + C M_j^{-\frac{2}{n-2\sigma}} \int_{\mathcal{A}_3} \frac{|y| - \lambda}{\lambda} |z|^{2\sigma-n} (|z| - \lambda)|z|^{-2\sigma-n} dz\\
& =: I_1 + I_2 + I_3 + I_4. 
\endaligned
\end{equation}
For $I_1$ we have 
\begin{equation}\label{SAF4_316+1}
\aligned
I_1 & \leq - C M_j^{-\frac{2}{n-2\sigma}} \frac{(|y|- \lambda)}{\lambda^{n-2\sigma+2}} \int_{D_\lambda} \frac{(|z| -\lambda)^2}{\left( 1 + (z_1 - \lambda)^2 + |z^{\prime}|^2 \right)^{\frac{n + 2\sigma}{2}}} dz \\
& \leq  - C M_j^{-\frac{2}{n-2\sigma}} \cdot \frac{(|y|- \lambda)}{\lambda^{n-2\sigma+2}} \cdot \frac{1}{\lambda^{2\sigma-2}} \int_{D_\lambda} \frac{(|z| -\lambda)^2}{ \left( 1 + (z_1 - \lambda)^2 + |z^{\prime}|^2 \right)^{\frac{n + 2}{2}}} dz\\
& \leq  - C M_j^{-\frac{2}{n-2\sigma}} (|y|- \lambda) \frac{\log \lambda}{\lambda^n}, 
\endaligned
\end{equation}
where we used the fact $\sigma \geq 1$ in the second inequality.  Since $|z| - \lambda \leq 4(|y| - \lambda)/3$ when $z\in \mathcal{A}_1$, for $\I_2$ we have  
\begin{equation}\label{SAF4_316+2}
\aligned
I_2 & \leq C M_j^{-\frac{2}{n-2\sigma}} \int_{\mathcal{A}_1 \backslash D_\lambda} |y -z|^{2\sigma -n}  (|z| - \lambda) |z|^{-2\sigma-n}dz \\
& \leq C M_j^{-\frac{2}{n-2\sigma}} (|y| - \lambda) \lambda^{-2\sigma-n} \int_{|z-y|\leq \lambda} |y -z|^{2\sigma -n} dz\\
& \leq C M_j^{-\frac{2}{n-2\sigma}} (|y| - \lambda) \frac{1}{\lambda^{n}}.
\endaligned
\end{equation}
Since $|z|- \lambda \leq 4 |z - y|$ when $z\in \mathcal{A}_2$, for $I_3$ we have  
\begin{equation}\label{SAF4_316+3}
\aligned
I_3 & \leq C M_j^{-\frac{2}{n-2\sigma}} (|y| - \lambda) \int_{\mathcal{A}_2 \backslash D_\lambda}\frac{1}{|y -z|^{n-2\sigma}} |z|^{-2\sigma-n}dz\\
& \leq C M_j^{-\frac{2}{n-2\sigma}} (|y| - \lambda) \frac{1}{\lambda^{n}}.
\endaligned
\end{equation}
For $I_4$ we have 
\begin{equation}\label{SAF4_316+4}
I_4 \leq C M_j^{-\frac{2}{n-2\sigma}}(|y| - \lambda) \frac{1}{\lambda}  \int_{\mathcal{A}_3 } |z|^{1-2n} dz \leq C M_j^{-\frac{2}{n-2\sigma}} (|y| - \lambda) \frac{1}{\lambda^{n}}. 
\end{equation} 
Substituting  these estimates into  \eqref{SAF4_315=009} and taking $R$ sufficiently large,  we obtain 
$$
\Phi_\lambda(y) \leq -C M_j^{-\frac{2}{n-2\sigma}} (|y|- \lambda) \frac{\log \lambda}{\lambda^n} ~~~~~~ \textmd{for} ~ \lambda <  |y| \leq 4\lambda.
$$
Next we estimate $|\Phi_\lambda(y)|$. It is clear that we only need to give the estimation of 
$$
\widetilde{I}_1:=\int_{D_\lambda} G(0, \lambda; y, z) |Q_\lambda(z)| dz. 
$$
By \eqref{SAF4_315=005} and \eqref{SAF4_315=008} we get 
$$
\aligned
\widetilde{I}_1 & \leq C M_j^{-\frac{2}{n-2\sigma}} \int_{D_\lambda} G(0, \lambda; y, z)  (|z| - \lambda)  dz \\
& \leq C M_j^{-\frac{2}{n-2\sigma}} \left(  \int_{\mathcal{A}_1} |y -z|^{2\sigma -n}  (|z| - \lambda)  dz + \int_{\mathcal{A}_2} \frac{(|y| - \lambda)(|z| - \lambda)^2}{|y -z|^{n-2\sigma+2}} dz\right)\\
& \leq C M_j^{-\frac{2}{n-2\sigma}}  (|y| -\lambda)  \int_{|z-y|\leq 12\lambda} |y -z|^{2\sigma-n} dz \\
& \leq C M_j^{-\frac{2}{n-2\sigma}}  (|y| -\lambda) \lambda^{2\sigma}, 
\endaligned
$$
where we used the facts that $|z| - \lambda \leq 4(|y| - \lambda)/3$ for  $z\in \mathcal{A}_1$ and $|z|- \lambda \leq 4 |z - y|$ for $z\in \mathcal{A}_2$ in the third inequality.  This,  together with \eqref{SAF4_316+2}, \eqref{SAF4_316+3} and \eqref{SAF4_316+4},  yields 
$$
|\Phi_\lambda(y)| \leq  C M_j^{-\frac{2}{n-2\sigma}}  (|y| -\lambda) \lambda^{2\sigma} ~~~~~~ \textmd{for} ~ \lambda <  |y| \leq 4\lambda. 
$$

\vskip0.10in 

{\it Case 2:} $|y| > 4\lambda$.  When $z \in D_\lambda$, we have $|y - z| > (|z| - \lambda)/3$. By the symmetry of $G(0, \lambda; y, z)$ and $(2)$ of Lemma \ref{InK-01},  we obtain 
\begin{equation}\label{SAF4_316*T1}
G(0, \lambda; y, z) = G(0, \lambda; z, y) \geq C \frac{(|z| -\lambda)(|y|^2 - \lambda^2)}{\lambda |y -z|^{n-2\sigma+2}} \geq C \frac{|z| -\lambda}{\lambda} |y|^{2\sigma-n} ~~~~~~ \textmd{for} ~ z\in D_\lambda. 
\end{equation}
Moreover, we also have
\begin{equation}\label{SAF4_316*T2}
G(0, \lambda; y, z) \leq C |y -z|^{2\sigma-n} ~~~~~~ \textmd{for} ~ z\in \Omega_j \backslash B_\lambda. 
\end{equation}
Therefore, by \eqref{SAF4_315=006}, 
\begin{equation}\label{SAF4_316*T3}
\aligned
\Phi_\lambda(y) & \leq - C M_j^{-\frac{2}{n-2\sigma}} \int_{D_\lambda} \frac{(|z| -\lambda)^2}{\lambda} |y|^{2\sigma-n} \left( \frac{1}{1 + (z_1 - \lambda)^2 + |z^{\prime}|^2}\right)^{\frac{n + 2\sigma}{2}}  dz\\
& ~~~~~ + C M_j^{-\frac{2}{n-2\sigma}} \int_{\Omega_j \backslash (B_\lambda \cup D_\lambda)} |y -z|^{2\sigma-n} (|z| - \lambda) |z|^{-2\sigma-n}dz\\
& =: \textit{II}_1 + \textit{II}_2. 
\endaligned
\end{equation}
For $\textit{II}_1$ we have 
$$
\aligned
\textit{II}_1 & \leq -C M_j^{-\frac{2}{n-2\sigma}} |y|^{2\sigma-n}  \lambda^{-1}  \int_{D_\lambda} \frac{(|z| -\lambda)^2}{ \left( 1 + (z_1 - \lambda)^2 + |z^{\prime}|^2 \right)^{\frac{n + 2\sigma}{2}} }  dz\\
& \leq -C M_j^{-\frac{2}{n-2\sigma}} |y|^{2\sigma-n}  \lambda^{1 - 2\sigma}  \int_{D_\lambda} \frac{(|z| -\lambda)^2}{ \left( 1 + (z_1 - \lambda)^2 + |z^{\prime}|^2 \right)^{\frac{n + 2}{2}} }  dz\\
& \leq  - C M_j^{-\frac{2}{n-2\sigma}} |y|^{2\sigma-n}  \frac{\log \lambda}{\lambda^{2\sigma-1}}, 
\endaligned
$$
where we used the fact $\sigma \geq 1$ in the second inequality. For $\textit{II}_2$  we have 
$$
\aligned
\textit{II}_2 & \leq C M_j^{-\frac{2}{n-2\sigma}} \int_{\Omega_j \backslash (B_\lambda \cup D_\lambda)} |y -z|^{2\sigma-n}  |z|^{1-2\sigma-n}dz \\
& =  C M_j^{-\frac{2}{n-2\sigma}} \sum_{k=1}^4 \int_{\mathcal{S}_k \backslash D_\lambda}|y -z|^{2\sigma-n}  |z|^{1-2\sigma-n}dz, 
\endaligned
$$ 
where 
$$
\mathcal{S}_1 = \{ z \in \Omega_j \backslash B_\lambda: |z| < |y|/2 \}, 
$$
$$
\mathcal{S}_2 = \{ z \in \Omega_j \backslash B_\lambda: |z| > 2|y| \}, 
$$
$$
\mathcal{S}_3 = \{ z \in \Omega_j \backslash B_\lambda: |z-y| \leq  |y|/2 \}, 
$$
$$
\mathcal{S}_4 = \{ z \in \Omega_j \backslash B_\lambda: |z-y| \geq  |y|/2  ~ \textmd{and} ~ |y|/2 \leq |z| \leq  2|y|\}. 
$$
Direct computations give 
$$
\int_{\mathcal{S}_1} |y -z|^{2\sigma-n}  |z|^{1-2\sigma-n}dz \leq C |y|^{2\sigma - n} \int_{\mathcal{S}_1} |z|^{1-2\sigma-n}dz \leq C \frac{|y|^{2\sigma - n} }{\lambda^{2\sigma-1}}~~~ \textmd{if} ~\sigma > \frac{1}{2}, 
$$
$$
\int_{\mathcal{S}_2} |y -z|^{2\sigma-n}  |z|^{1-2\sigma-n}dz  \leq C \int_{\mathcal{S}_2} |z|^{2\sigma-n}  |z|^{1-2\sigma-n}dz \leq C |y|^{1-n} \leq C  \frac{|y|^{2\sigma - n}}{\lambda^{2\sigma-1}}~~~ \textmd{if} ~\sigma \geq  \frac{1}{2},
$$
$$
\int_{\mathcal{S}_3} |y -z|^{2\sigma-n}  |z|^{1-2\sigma-n}dz  \leq C  |y|^{1-2\sigma-n} \int_{\mathcal{S}_3} |y- z|^{2\sigma-n} dz \leq C |y|^{1-n} \leq C  \frac{|y|^{2\sigma - n}}{\lambda^{2\sigma-1}}~~~ \textmd{if} ~\sigma \geq  \frac{1}{2},
$$
and
$$
\int_{\mathcal{S}_4} |y -z|^{2\sigma-n}  |z|^{1-2\sigma-n}dz  \leq C  |y|^{1-2\sigma-n} \int_{\mathcal{S}_4} |y- z|^{2\sigma-n} dz \leq C |y|^{1-n} \leq C  \frac{|y|^{2\sigma - n}}{\lambda^{2\sigma-1}}~~~ \textmd{if} ~\sigma \geq  \frac{1}{2}. 
$$
These estimates imply that 
\begin{equation}\label{SAF4_316*T4}
\textit{II}_2 \leq C M_j^{-\frac{2}{n-2\sigma}} |y|^{2\sigma - n} \frac{1}{\lambda^{2\sigma-1}}. 
\end{equation}
Hence, by taking $R$ sufficiently large  we have  
$$
\Phi_\lambda(y) \leq  - C M_j^{-\frac{2}{n-2\sigma}} |y|^{2\sigma-n}  \frac{\log \lambda}{\lambda^{2\sigma-1}} ~~~~~~ \textmd{for} ~  |y| >  4\lambda.  
$$
To bound $|\Phi_\lambda(y)|$, we only need to give the estimation of 
$$
\widetilde{\textit{II}}_1:=\int_{D_\lambda} G(0, \lambda; y, z) |Q_\lambda(z)| dz. 
$$
By \eqref{SAF4_315=005} and \eqref{SAF4_316*T2} we get 
$$
\aligned
\widetilde{\textit{II}}_1 & \leq	C M_j^{-\frac{2}{n-2\sigma}}  \int_{D_\lambda} |y -z|^{2\sigma-n} (|z| - \lambda) dz \\
&  \leq C  M_j^{-\frac{2}{n-2\sigma}} |y |^{2\sigma-n} \int_{D_\lambda}  (|z| - \lambda) dz \leq  C  M_j^{-\frac{2}{n-2\sigma}} |y |^{2\sigma-n} \lambda^{n+1}, 
\endaligned
$$ 
where we have used $|y-z|\geq |y|/2$ for $z \in D_\lambda$. Combining this with \eqref{SAF4_316*T4} we obtain
$$
|\Phi_\lambda(y)| \leq C  M_j^{-\frac{2}{n-2\sigma}} |y |^{2\sigma-n} \lambda^{n+1} ~~~~~~ \textmd{for} ~  |y| >  4\lambda. 
$$
The proof of Lemma \ref{LS4-312=01} is completed.  
\end{proof}

In the proof of Lemma \ref{LS4-312=01}, $R$ is chosen so large that $\Phi_\lambda$ is negative for any $\lambda \in [R-2, R+2]$.   Recall $\lambda_0= R-2$ and $\lambda_1 =R +2$.   By Lemmas \ref{LAn-01} and \ref{LS4-312=01} we know  that for large $j$, 
$$
\varphi_{\lambda_0}(y) +  \Phi_{\lambda_0}(y) \geq 0,  ~~~~~~~  y \in \Pi_j \backslash \overline{B}_{\lambda_0}. 
$$
We define
$$
\bar{\lambda}:= \sup\{ \mu \geq \lambda_0 ~ | ~ \varphi_{\lambda}(y) +  \Phi_{\lambda}(y) \geq 0, ~ \forall ~ y \in \Pi_j \backslash \overline{B}_{\lambda}, ~ \forall ~ \lambda_0 \leq  \lambda \leq  \mu \}. 
$$
Then,  $\bar{\lambda}$ is well defined and $\bar{\lambda}  \geq \lambda_0$ for all large $j$. Furthermore, from Lemmas \ref{LAn-01} and \ref{LS4-312=01} we have that  $\bar{\lambda} <  \lambda_1$ for all large $j$. Using the continuity we obtain 
$$
\varphi_{\bar{\lambda}}(y) +  \Phi_{\bar{\lambda}}(y) \geq 0,  ~~~~~~~  y \in \Pi_j \backslash \overline{B}_{\bar{\lambda}}. 
$$
This together with Lemma \ref{LS4-312=01} implies that 
\begin{equation}\label{BYhf=01}
\varphi_{\bar{\lambda}}(y) \geq  0,  ~~~~~~~  y \in \Pi_j \backslash \overline{B}_{\bar{\lambda}}. 
\end{equation}
It follows from \eqref{SL4-312=04} that for $\bar{\lambda} \leq \lambda \leq \lambda_1$ and $y\in \Pi_j \backslash \overline{B}_{\lambda}$, 
\begin{equation}\label{y-01}
\varphi_\lambda(y) +  \Phi_\lambda(y) \geq \int_{\Pi_j \backslash B_\lambda} G(0, \lambda; y, z) b_\lambda(z) \varphi_\lambda(z) dz + J_\lambda(y), 
\end{equation}
where $J_\lambda$ is given by   
$$
\aligned
J_\lambda(y) & = \int_{\Omega_j \backslash \Pi_j} G(0,\lambda; y, z) K_j(z)\left( w_j(z)^{\frac{n+2\sigma}{n-2\sigma}} -  w_j^\lambda(z)^{\frac{n+2\sigma}{n-2\sigma}} \right) dz\\
& ~~~~~ - \int_{\Omega_j^c} G(0,\lambda; y, z) K_j(z^\lambda) w_j^\lambda(z)^{\frac{n+2\sigma}{n-2\sigma}} dz.
\endaligned 
$$

\begin{lemma}\label{hhff=01} 
There  exists a constant $C>0$ depending on $R$ such that for any $\lambda_0 \leq \lambda \leq \lambda_1$ and for all large $j$, 
\begin{equation}\label{y-05}
J_\lambda(y)  \geq 
\begin{cases}
C (|y| - \lambda) u(x_j)^{-1}, ~~~~~   &\textmd{if}~   \lambda \leq |y| \leq \lambda_1  +1, \\
C   u(x_j)^{-1}, ~~~~~ & \textmd{if}~ |y| > \lambda_1 +1,  ~ y \in \Pi_j. 
\end{cases}
\end{equation}
\end{lemma} 
\begin{proof} 
For any $z\in \mathbb{R}^n \backslash (\Pi_j \cup \Sigma_j)$  and $\lambda_0 \leq \lambda \leq \lambda_1 $,  we have $|z| \geq \frac{1}{2} u(x_j)^{\frac{2}{n-2\sigma}}$ for large $j$ and hence  
$$
w_j^\lambda(z) \leq \left( \frac{\lambda}{|z|} \right)^{n-2\sigma} \max_{B_{\lambda_1}} w_j \leq C u(x_j)^{-2},  
$$
where $C>0$ depends on $R$.  On the other hand,  by the equation \eqref{Int} we have
\begin{equation}\label{Esi-301}
u(x) \geq 4^{2\sigma-n} \int_{B_2} K(y) u(y)^{\frac{n+2\sigma}{n-2\sigma}} dy=:A_0 > 0 ~~~~~ \textmd{for} ~ \textmd{all}  ~ x\in B_2 \backslash \Sigma,
\end{equation} 
and by the definition of $w_j$, we obtain
\begin{equation}\label{y-03}
w_j(y) \geq \frac{A_0}{u(x_j)} ~~~~~~ \textmd{for} ~ y \in \Omega_j \backslash \Pi_j. 
\end{equation}
Therefore, for large $j$ we have 
$$
w_j(z)^{\frac{n+2\sigma}{n-2\sigma}}  -  w_j^\lambda(z)^{\frac{n+2\sigma}{n-2\sigma}} \geq \frac{1}{2} w_j(z)^{\frac{n+2\sigma}{n-2\sigma}} ~~~~~~ \textmd{in} ~ \Omega_j \backslash \Pi_j. 
$$
Since $G(0,\lambda;y,z) =0$ for  $|y|=\lambda$ and 
$$
y\cdot \nabla_y G(0,\lambda;y,z) \Big|_{|y|=\lambda} = (n-2\sigma) |y-z|^{2\sigma-n-2}(|z|^2 - |y|^2) >0
$$
for $|z| \geq {\lambda}_1 +2$,  and by  the positivity and smoothness of $G$ we obtain
\begin{equation}\label{K011}
\frac{\delta_1}{|y -z|^{n-2\sigma}} (|y| - \lambda) \leq G(0,\lambda; y, z) \leq  \frac{\delta_2}{|y -z|^{n-2\sigma}} (|y| - \lambda)
\end{equation}
for $\lambda_0 \leq \lambda  \leq |y| \leq \lambda_1 +1$ and  $\lambda_1+2 \leq |z| \leq  M <\infty$, where $0< \delta_1\leq \delta_2 < \infty$.  Moreover,   if $M$ is large, then 
$$
0< c \leq y \cdot \nabla_y( |y-z|^{n-2\sigma} G(0,\lambda;y,z)) \leq C < \infty
$$
for all  $|z|\geq M $ and $\lambda_0 \leq \lambda \leq |y| \leq \lambda_1 +1$. 
Hence,  \eqref{K011} also holds for $\lambda_0 \leq \lambda \leq |y| \leq \lambda_1 +1$ and  $|z|\geq M$.  Moreover, by the definition of $G(0,\lambda;y,z)$, we can verify that for $|y|\geq \lambda_1 +1$ and $|z| \geq  \lambda_1 +2$,
\begin{equation}\label{K022}
\frac{\delta_3}{|y -z|^{n-2\sigma}}  \leq G(0,\lambda; y, z) \leq  \frac{1}{|y -z|^{n-2\sigma}} 
\end{equation}
for some $\delta_3\in (0, 1)$.

Denote $\rho_j:=u(x_j)^{\frac{2}{n-2\sigma}} $ and  $A_1:=\min_{B_2} K >0$.   Then,   for $\lambda \leq |y| \leq \lambda_1 +1$ and for  large $j$,   we have
$$
\aligned
J_\lambda(y) & \geq \frac{A_1}{2}\left( \frac{A_0}{u(x_j)} \right)^{\frac{n+2\sigma}{n-2\sigma}} \int_{\Omega_j \backslash \Pi_j}  \frac{\delta_1}{|y -z|^{n-2\sigma}} (|y| - \lambda)  dz \\
&~~~~~ - C \int_{\Omega_j^c} \frac{\delta_2}{|y -z|^{n-2\sigma}} (|y| - \lambda)  \left(  \frac{\lambda}{|z|}\right)^{n+2\sigma} dz \\
& \geq C (|y| - \lambda)  u(x_j)^{-\frac{n+2\sigma}{n-2\sigma}}  \int_{\{\frac{5}{4} \rho_j \leq  |z| \leq \frac{7}{4}\rho_j \} \backslash \Sigma_j} \frac{1}{|y -z|^{n-2\sigma}}  dz \\
& ~~~~~ -C (|y| - \lambda) \int_{\{ |z| \geq \frac{7}{4}\rho_j \} \cup \Sigma_j}  \frac{1}{|y -z|^{n-2\sigma}} \left( \frac{1}{|z|}\right)^{n+2\sigma} dz\\
& \geq C (|y| - \lambda) u(x_j)^{-1} - C (|y| - \lambda) u(x_j)^{-\frac{2n}{n-2\sigma}} \\
& \geq C (|y| - \lambda) u(x_j)^{-1}, 
\endaligned
$$
where we have used $\mathcal{L}^n(\Sigma_j) =0$  and $u(x_j)  \to \infty$ as $j\to \infty$.   Similarly, for $|y| \geq \lambda_1 +1$ and $y \in \Pi_j$, we have
$$
J_\lambda(y)\geq C  u(x_j)^{-1} - C  u(x_j)^{-\frac{2n}{n-2\sigma}}  \geq C  u(x_j)^{-1}. 
$$
 Lemma \ref{hhff=01} is established. 
\end{proof}

By \eqref{BYhf=01}, \eqref{y-01} and \eqref{y-05},  there exists a small $\varepsilon_1 \in (0, \lambda_1-\bar{\lambda})$  (which depends on $j$)  such that 
$$
\varphi_{\bar{\lambda}}(y) +  \Phi_{\bar{\lambda}}(y) \geq J_{\bar{\lambda}}(y) \geq C u(x_j)^{-1} \geq \frac{\varepsilon_1}{|y|^{n-2\sigma}} ~~~~~~ \forall ~ |y| \geq  \lambda_1 + 1,  y \in \Pi_j. 
$$
By the above estimate and the explicit formulas for  $w_j^{\lambda}(y)$ and $\Phi_\lambda(y)$, there exists $\varepsilon_2 \in (0, \varepsilon_1)$ such that for any $\bar{\lambda} \leq \lambda \leq \bar{\lambda} +\varepsilon_2$, 
\begin{equation}\label{1wj}
\aligned
\varphi_\lambda(y) +  \Phi_\lambda(y)  & = [\varphi_{\bar{\lambda}}(y)  +  \Phi_{\bar{\lambda}}(y)] + [\Phi_{\lambda}(y) - \Phi_{\bar{\lambda}}(y)] + [w_j^{\bar{\lambda}}(y) - w_j^\lambda(y) ] \\
& \geq  \frac{\varepsilon_1}{2|y|^{n-2\sigma}}~~~~~~ \forall ~ |y| \geq  \lambda_1 +1,  ~ y\in \Pi_j. 
\endaligned
\end{equation}
This,  together with Lemma \ref{LS4-312=01},  also implies that for any $\bar{\lambda} \leq \lambda \leq \bar{\lambda} +\varepsilon_2$, 
\begin{equation}\label{MHYBF-01}
w_j(y) - w_j^\lambda(y) \geq  \frac{\varepsilon_1}{2|y|^{n-2\sigma}}~~~~~~ \forall ~ |y| \geq  \lambda_1 +1,  ~ y\in \Pi_j. 
\end{equation}
For $\varepsilon \in \left( 0, \varepsilon_2 \right)$ which we choose below, by \eqref{y-01},  \eqref{y-05} and \eqref{MHYBF-01} we have, for $\bar{\lambda} \leq \lambda \leq \bar{\lambda} +\varepsilon $ and for $\lambda \leq |y| \leq \lambda_1 +1$, 
$$
\aligned
\varphi_\lambda(y) +  \Phi_\lambda(y)   & \geq \int_{\lambda \leq |z| \leq \lambda_1 +1}  G(0,\lambda; y, z) K_j(z) \left( w_j(z)^{\frac{n+2\sigma}{n-2\sigma}} - w_j^\lambda(z)^{\frac{n+2\sigma}{n-2\sigma}} \right) dz\\
& ~~~ +     \int_{\lambda_1 +2  \leq |z| \leq \lambda_1 +3}  G(0,\lambda; y, z) K_j(z) \left( w_j(z)^{\frac{n+2\sigma}{n-2\sigma}} - w_j^\lambda(z)^{\frac{n+2\sigma}{n-2\sigma}} \right) dz \\
& \geq -C \int_{\lambda \leq |z| \leq \lambda +\varepsilon}   G(0,\lambda; y, z) (|z| -\lambda) dz\\
& ~~~ + \int_{\lambda+\varepsilon \leq |z| \leq \lambda_1 +1}  G(0,\lambda; y, z) K_j(z) \left( w_j^{\bar{\lambda}}(z)^{\frac{n+2\sigma}{n-2\sigma}} - w_j^\lambda(z)^{\frac{n+2\sigma}{n-2\sigma}} \right) dz \\
& ~~~ + A_1\int_{\lambda_1 +2  \leq |z| \leq \lambda_1 +3 }  G(0,\lambda; y, z) \left( w_j(z)^{\frac{n+2\sigma}{n-2\sigma}} -  w_j^\lambda(z)^{\frac{n+2\sigma}{n-2\sigma}} \right) dz,
\endaligned
$$
where $A_1=\min_{B_2} K >0$, and we also used 
$$
| w_j(z)^{\frac{n+2\sigma}{n-2\sigma}} - (w_j)_\lambda(z)^{\frac{n+2\sigma}{n-2\sigma}}  |\leq C(|z| - \lambda)
$$
in the second inequality. By \eqref{MHYBF-01}  there exists $\hat{\delta} >0$ (which depends on $j$) such that
$$
w_j(z)^{\frac{n+2\sigma}{n-2\sigma}} - w_j^\lambda(z)^{\frac{n+2\sigma}{n-2\sigma}} \geq \hat{\delta} ~~~~~~ \textmd{for}  ~ \lambda_1 +2  \leq |z| \leq \lambda_1 +3.  
$$
Since $\|w_j\|_{C^1(B_{\lambda_1 + 2})}\leq C$ and $\|K_j\|_{L^\infty(B_{\lambda_1 + 2})} \leq C $ for some constant $C$ independent of $j$,  there exists some constant $C>0$ independent of both $\varepsilon$ and $j$  such that for any $\bar{\lambda} \leq \lambda \leq \bar{\lambda} +\varepsilon$, 
$$
K_j(z)| w_j^{\bar{\lambda}}(z)^{\frac{n+2\sigma}{n-2\sigma}} - w_j^\lambda(z)^{\frac{n+2\sigma}{n-2\sigma}}  |\leq C(\lambda - \bar{\lambda})\leq C \varepsilon~~~~~ \forall ~ \lambda \leq |z| \leq  \lambda_1 +1.
$$
For any  $\lambda \leq |y| \leq \lambda_1+1$, one can estimate the integrals of the kernel $G$ (or, see \cite{JX19}) as follows:   
$$
\aligned
\int_{\lambda+\varepsilon  \leq |z| \leq \lambda_1 +1}  G(0,\lambda; y, z) dz &  \leq \left|  \int_{\lambda+\varepsilon \leq |z| \leq \lambda_1 +1}  \left( \frac{1}{|y -z|^{n-2\sigma}} -  \frac{1}{|y^{\lambda} -z|^{n-2\sigma}}\right) dz  \right| \\
& ~~~~ + \int_{\lambda+\varepsilon \leq |z| \leq \lambda_1 +1}  \left|  \left(\frac{\lambda}{|y|}\right)^{n-2\sigma} -1 \right|   \frac{1}{|y^{\lambda} -z|^{n-2\sigma}} dz\\
& \leq C (\varepsilon^{2\sigma-1}  +  |\ln \varepsilon|  +1) (|y| - \lambda) 
\endaligned
$$
and 
$$
\aligned
\int_{ \lambda \leq |z| \leq \lambda +\varepsilon }   G(0,\lambda; y, z) (|z| -\lambda) dz & \leq \left|  \int_{ \lambda \leq |z| \leq \lambda +\varepsilon }   \left( \frac{|z| -\lambda}{|y -z|^{n-2\sigma}} -  \frac{|z| -\lambda}{|y^{\lambda} -z|^{n-2\sigma}}\right) dz  \right| \\
& ~~~ + \varepsilon  \int_{ \lambda \leq |z| \leq \lambda +\varepsilon}  \left|  \left(\frac{\lambda}{|y|}\right)^{n-2\sigma} -1 \right|   \frac{1}{|y^{\lambda} -z|^{n-2\sigma}} dz\\
& \leq C (|y| - \lambda) \varepsilon^{2\sigma/n} + C\varepsilon(|y| - \lambda) \\
& \leq C (|y| - \lambda) \varepsilon^{2\sigma/n}. 
\endaligned
$$
Therefore,   by  \eqref{K011}  we have  for  $\lambda < |y| \leq \lambda_1 + 1$ that 
$$
\aligned
 \varphi_\lambda(y) +  \Phi_\lambda(y) & \geq -C \varepsilon^{2\sigma/n} (|y| - \lambda) + \delta_1\hat{\delta} (|y| - \lambda) \int_{\lambda_1 +2  \leq |z| \leq \lambda_1 +3} \frac{1}{|y-z|} dz\\
&  \geq \left(\delta_1\hat{\delta} c -C \varepsilon^{2\sigma/n} \right) (|y| - \lambda) \geq 0
\endaligned
$$
if $\varepsilon$ is sufficiently small.  This  and  \eqref{1wj}  contradict the definition of $\bar{\lambda}$ .  Hence the proof of Step 2 is finished. 

\vskip0.10in

{\it Step 3. } In this step, we will complete the proof of Theorem \ref{IEthm01} under the assumption of (K1).  By Step 2 we have $\lim_{j\to \infty} |\nabla K(x_j)| = |\nabla K(0)| = 0$. From the assumption (K1) we know that there exists a neighborhood $N$ of $0$ such that \eqref{K1-0} and \eqref{K1-1} hold.  Here we can suppose $N=B_2$, since otherwise we just consider the integral equation on $N$ and then make a rescaling argument. That is,  we  suppose  $K$ satisfies 
\begin{equation}\label{K1hy-0}
c_1 |y |^{\alpha-1} \leq |\nabla K(y)| \leq c_2 |y |^{\alpha-1} ~~~~~~  \textmd{for} ~ y \in B_2, 
\end{equation}
and for any $\varepsilon>0$, there exists $\delta=\delta(\varepsilon)$ such that for $|z - y|< \delta|y|$ we have
\begin{equation}\label{K1hy-1}
|\nabla K(z) - \nabla K(y)| < \varepsilon |y|^{\alpha -1}, 
\end{equation}
where $\alpha >1$, $y, z \in B_2$,  $c_1$ and $c_2$ are two positive constants. 

Without loss of generality, we may assume
\begin{equation}\label{YFmgh=01}
\lim_{j \to \infty} \frac{\nabla K(x_j)}{|\nabla K(x_j)|} = e = (1, 0, \dots, 0). 
\end{equation} 
As in Step 2, let $w_j$ and $h_j$ be defined as in \eqref{21SS3-01} and let  $\Omega_j$ be defined as in \eqref{hhmm-9e7=5}.  Then $w_j$ still satisfies \eqref{21SS3-044} on  $\Omega_j $ where $K_j$ is defined as in \eqref{21SS3-033}. Moreover, on every compact subset of $\mathbb{R}^n$,   $w_j$ converges in $C^2$ norm to $U_1=U_0(\cdot - Re)$ by  Step 1.  We also extend $w_j$ to be identically $0$ outside $\Omega_j$ and $K$ to be identically $0$ outside $B_2$. 

Define $w_j^\lambda$ and $h_j^\lambda$ as in \eqref{21SS3-022}, then $w_j^\lambda$ satisfies \eqref{21SS3-055}.  It is clear that Lemma \ref{LAn-01} still holds.  Let $\varphi_\lambda(y) =w_j(y)  -w_j^\lambda(y)$ with $\lambda \in [R-2, R+2]$,  then $\varphi_\lambda$ satisfies the following integral inequality for  large $j$,  
\begin{equation}\label{Ste3-BFh-01}
\varphi_\lambda(y) +  \Phi_\lambda(y) \geq \int_{\Pi_j \backslash B_\lambda} G(0, \lambda; y, z) b_\lambda(z) \varphi_\lambda(z) dz + J_\lambda(y),  ~~~~~ y \in \Pi_j \backslash \overline{B}_\lambda, 
\end{equation}
where $\Pi_j$, $b_\lambda$, $\Phi_\lambda$  and $J_\lambda$ are respectively given as in \eqref{21PAI01}, \eqref{SL4-312=01}, \eqref{SL4-312=02} and \eqref{SL4-312=03}.

In order to apply the moving sphere method,  as in Step 2,  we need to establish some estimates for $\Phi_\lambda$ on $\Pi_j \backslash \overline{B}_\lambda$.  Define a special domain  $D_\lambda$ in $\Omega_j \backslash B_\lambda$ as follows,   
$$
D_\lambda = \{z \in \mathbb{R}^n : \lambda < |z| < 2\lambda, z_1 > 2 |z^{\prime}| ~\textmd{where} ~ z^{\prime}=(z_2, \dots, z_n) \}.
$$ 
Define $l_j := M_j^{\frac{2}{n-2\sigma}}|x_j|$ where $M_j=u(x_j)$, then $l_j \to \infty$ and so $l_j \gg R$ when $j$ is large.    Under our current assumption \eqref{YFmgh=01},  we have    

\vskip0.10in 

{\it Claim 3: There exists a constant $C>0$ independent of $\lambda$ such that for all large $j$,   
\begin{equation}\label{S4^m317=01}
K_j(z^\lambda) - K_j(z) \leq  -C M_j^{-\frac{2}{n-2\sigma}} |x_j|^{\alpha -1} (|z| - \lambda), ~~~~~ z \in D_\lambda. 
\end{equation}
\begin{equation}\label{S4^m317=02}
|K_j(z^\lambda) - K_j(z)| \leq C M_j^{-\frac{2}{n-2\sigma}} |x_j|^{\alpha -1} (|z| - \lambda), ~~~~~ z \in B_{2l_j} \backslash \overline{B}_\lambda. 
\end{equation}
\begin{equation}\label{S4^m317=03}
|K_j(z^\lambda) - K_j(z)| \leq C M_j^{-\frac{2\alpha}{n-2\sigma}} |z|^\alpha, ~~~~~ z \in \Omega_j \backslash B_{2l_j}. 
\end{equation}
}
\begin{proof}
(1)  For $z \in D_\lambda$,  by the mean value theorem and the assumptions \eqref{K1hy-0}-\eqref{K1hy-1} on $K$  we obtain that for $j$ large, 
$$
\aligned
K_j(z^\lambda) - K_j(z) & = M_j^{-\frac{2}{n-2\sigma}} \nabla K\Big( x_j + M_j^{-\frac{2}{n-2\sigma}}\big( (1-\theta)z^\lambda + \theta z -Re \big)  \Big)  \cdot (z^\lambda  - z) \\
& \leq M_j^{-\frac{2}{n-2\sigma}} \nabla K(x_j)\cdot (z^\lambda -z) + \varepsilon M_j^{-\frac{2}{n-2\sigma}}  |x_j|^{\alpha-1}|z^\lambda -z| \\
& \leq -C M_j^{-\frac{2}{n-2\sigma}} |x_j|^{\alpha-1} (|z| - \lambda)
\endaligned
$$
by taking $\varepsilon$ sufficiently small, where $\theta \in (0, 1)$ and $C>0$ is a constant.

(2) For $z \in B_{2l_j} \backslash \overline{B}_\lambda$,  by the mean value theorem and the assumption \eqref{K1hy-0} on $K$ we obtain that 
$$
\aligned
|K_j(z^\lambda) - K_j(z)| & \leq  M_j^{-\frac{2}{n-2\sigma}} \Big| \nabla K\Big( x_j + M_j^{-\frac{2}{n-2\sigma}}\big( (1-\theta)z^\lambda + \theta z -Re \big)  \Big)  \Big| |z^\lambda  - z| \\
& \leq C M_j^{-\frac{2}{n-2\sigma}} |x_j|^{\alpha-1} (|z| - \lambda)
\endaligned
$$
for some constant $C>0$,  where $\theta \in (0, 1)$.

(3) For $z \in \Omega_j \backslash B_{2l_j}$,  by the mean value theorem and the assumption \eqref{K1hy-0} on $K$ we obtain that 
$$
\aligned
|K_j(z^\lambda) - K_j(z)| & \leq  M_j^{-\frac{2}{n-2\sigma}} \Big| \nabla K\Big( x_j + M_j^{-\frac{2}{n-2\sigma}}\big( (1-\theta)z^\lambda + \theta z -Re \big)  \Big)  \Big| |z^\lambda  - z| \\
&\leq C M_j^{-\frac{2}{n-2\sigma}} \Big| x_j + M_j^{-\frac{2}{n-2\sigma}}\big( (1-\theta)z^\lambda + \theta z -Re \big) \Big|^{\alpha-1}  |z^\lambda  - z| \\
& \leq C M_j^{-\frac{2\alpha}{n-2\sigma}} |z|^{\alpha}
\endaligned
$$
for  some  constant  $C>0$,  where $\theta \in (0, 1)$.   Claim 3 is proved.  
\end{proof} 

It is also clear that Lemma \ref{Le-hm-074} still holds.  If we let 
$$
Q_\lambda(z) = \left( K_j(z^\lambda) - K_j(z) \right) w_j^\lambda(z)^{\frac{n+2\sigma}{n-2\sigma}} ~~~~~ \textmd{for} ~ z \in\Omega_j \backslash B_\lambda, 
$$
then by Claim 3 and Lemma \ref{Le-hm-074} we have
\begin{equation}\label{ST=-kfh01}
Q_\lambda(z) \leq -C M_j^{-\frac{2}{n-2\sigma}} |x_j|^{\alpha-1} (|z| - \lambda)\left( \frac{1}{1 + (z_1 - \lambda)^2 + |z^{\prime}|^2}\right)^{\frac{n + 2\sigma}{2}} ~~~~~~ \textmd{for} ~ z\in D_\lambda
\end{equation}
and 
\begin{equation}\label{ST=-kfh02}
|Q_\lambda(z)| \leq 
\begin{cases}
C M_j^{-\frac{2}{n-2\sigma}} |x_j|^{\alpha-1} (|z| - \lambda) ~~~~~~ & \textmd{for} ~ z\in D_\lambda, \\
C M_j^{-\frac{2}{n-2\sigma}} |x_j|^{\alpha-1} (|z| - \lambda) |z|^{-2\sigma-n} ~~~~~~ & \textmd{for} ~ z \in B_{2l_j} \backslash (B_\lambda \cup D_\lambda), \\
C M_j^{-\frac{2\alpha}{n-2\sigma}} |z|^{\alpha-2\sigma-n}~~~~~~ & \textmd{for} ~ z\in \Omega_j \backslash B_{2l_j}. 
\end{cases} 
\end{equation}
We split $\Phi_\lambda(y)$ into two parts:   
\begin{equation}\label{S-L-01h=}
\aligned
\Phi_\lambda(y) & = \int_{B_{2l_j} \backslash B_\lambda} G(0, \lambda; y, z) Q_\lambda(z) dz + \int_{\Omega_j \backslash B_{2l_j}} G(0, \lambda; y, z) Q_\lambda(z) dz \\
& =: \Phi_{1,\lambda}(y) + \Phi_{2, \lambda}(y). 
\endaligned
\end{equation} 
By the estimates of $Q_\lambda(z)$ on $B_{2l_j} \backslash B_\lambda$,  one can see easily that the estimates of $\Phi_{1,\lambda}$ are very similar to  those of $\Phi_\lambda$ in Lemma \ref{LS4-312=01} of Step 2.  Hence we have the following:
\begin{lemma}\label{L-swjhdg-010} 
There  exists  a  constant $C>0$ independent of  $\lambda$ such that for all large $j$, 
$$ 
\Phi_{1,\lambda} (y) \leq 
\begin{cases}
- C M_j^{-\frac{2}{n-2\sigma}} |x_j|^{\alpha-1}  (|y| - \lambda) \lambda^{-n} \log \lambda  ~~~~~~ & \textmd{for} ~ y \in B_{4\lambda} \backslash \overline{B}_\lambda, \\
-C M_j^{-\frac{2}{n-2\sigma}} |x_j|^{\alpha-1}  |y|^{2\sigma -n} \lambda^{1-2\sigma} \log \lambda ~~~~~~ & \textmd{for} ~ y \in \Pi_j \backslash \overline{B}_{4\lambda}
\end{cases}
$$
and
$$ 
|\Phi_{1,\lambda} (y)| \leq 
\begin{cases}
C M_j^{-\frac{2}{n-2\sigma}} |x_j|^{\alpha-1}  (|y| - \lambda) \lambda^{2\sigma}  ~~~~~~ & \textmd{for} ~ y \in B_{4\lambda} \backslash \overline{B}_\lambda, \\
C M_j^{-\frac{2}{n-2\sigma}} |x_j|^{\alpha-1}  |y|^{2\sigma -n} \lambda^{n+1} ~~~~~~ & \textmd{for} ~ y \in  \Pi_j  \backslash \overline{B}_{4\lambda}. 
\end{cases}
$$
\end{lemma}
\begin{proof}
We only need to replace $M_j^{-\frac{2}{n-2\sigma}}$ and $\Omega_j \backslash B_\lambda$ in Lemma \ref{LS4-312=01}  by  $M_j^{-\frac{2}{n-2\sigma}} |x_j|^{\alpha-1}$ and $B_{2l_j} \backslash B_\lambda$, respectively. 
\end{proof}

The estimates of $\Phi_{2,\lambda}$ depend on $\alpha$. For the sake of application, let's define  
$$
g_\lambda(y)  = M_j^{-\frac{2\alpha}{n-2\sigma}}  \int_{\Omega_j \backslash B_{2l_j}} G(0, \lambda; y, z) |z|^{\alpha-2\sigma-n} dz ~~~~~~ \textmd{for} ~ \lambda < |y| \leq  3 L_j, 
$$
where $L_j :=M_j^{\frac{2}{n-2\sigma}}$.  Then  $\Omega_j \subset B_{3L_j}$ for large $j$ and $|\Phi_{2,\lambda}(y)| \leq C g_\lambda(y)$ for some constant $C>0$ depending  only on $K, n, \sigma$.   

\begin{lemma}\label{Lg_lammfb=01} 
For $\lambda < |y| \leq 4\lambda$, we have
\begin{equation}\label{317=yfmh-001}
g_\lambda(y) \leq
\begin{cases}
C M_j^{-\frac{2n}{n-2\sigma}} (|y| - \lambda)\lambda^{-1} ~~~~~~ & \textmd{if} ~ \alpha>n, \\
C M_j^{-\frac{2n}{n-2\sigma}} \log M_j (|y| - \lambda)\lambda^{-1} ~~~~~~ & \textmd{if} ~ \alpha =n,\\
C M_j^{-\frac{2\alpha}{n-2\sigma}} l_j^{\alpha-n} (|y| - \lambda)\lambda^{-1} ~~~~~~ & \textmd{if} ~ 1 < \alpha <n. 
\end{cases}
\end{equation}
For $4\lambda \leq |y| \leq l_j$, we have
\begin{equation}\label{317=yfmh-002}
g_\lambda(y) \leq
\begin{cases}
C M_j^{-\frac{2n}{n-2\sigma}}  ~~~~~~ & \textmd{if} ~ \alpha>n, \\
C M_j^{-\frac{2n}{n-2\sigma}} \log M_j  ~~~~~~ & \textmd{if} ~ \alpha =n,\\
C M_j^{-\frac{2\alpha}{n-2\sigma}} l_j^{\alpha-n}  ~~~~~~ & \textmd{if} ~ 1 < \alpha <n. 
\end{cases}
\end{equation}
For $l_j \leq |y| \leq 3L_j$,   we have
\begin{equation}\label{317=yfmh-003}
g_\lambda(y) \leq
\begin{cases}
C M_j^{-\frac{2n}{n-2\sigma}}  ~~~~~~ & \textmd{if} ~ \alpha>n, \\
C M_j^{-\frac{2n}{n-2\sigma}} \log M_j  ~~~~~~ & \textmd{if} ~ \alpha =n,\\
C M_j^{-\frac{2\alpha}{n-2\sigma}} |y|^{\alpha -n}   ~~~~~~ & \textmd{if} ~ 2 \sigma< \alpha <n, \\
C M_j^{-\frac{2\alpha}{n-2\sigma}} |y|^{2\sigma-n} \log |y|  ~~~~~~ & \textmd{if} ~ \alpha =2 \sigma, \\
C M_j^{-\frac{2\alpha}{n-2\sigma}} l_j^{\alpha-2\sigma} |y|^{2\sigma -n} ~~~~~~ & \textmd{if} ~ 1< \alpha < 2\sigma. 
\end{cases}
\end{equation}
\end{lemma}
\begin{proof}
(1) $\lambda < |y|\leq 4\lambda$. By Lemma \ref{InK-01}  we have
$$
G(0, \lambda; y, z) \leq C \frac{(|y|-\lambda) (|z|^2 - \lambda^2)}{\lambda|y - z|^{n-2\sigma+2}} \leq C \frac{|y| - \lambda}{\lambda} |z|^{2\sigma-n}.  
$$
Hence
$$
\aligned
g_\lambda(y) & \leq C M_j^{-\frac{2\alpha}{n-2\sigma}} (|y| - \lambda)  \lambda^{-1}   \int_{\Omega_j \backslash B_{2l_j}}  |z|^{\alpha-2n} dz\\
& \leq  C M_j^{-\frac{2\alpha}{n-2\sigma}} (|y| - \lambda)  \lambda^{-1}   \int_{2l_j \leq |z| \leq 3L_j}  |z|^{\alpha-2n} dz \\
& \leq 
\begin{cases}
C M_j^{-\frac{2n}{n-2\sigma}} (|y| - \lambda)\lambda^{-1} ~~~~~~ & \textmd{if} ~ \alpha>n, \\
C M_j^{-\frac{2n}{n-2\sigma}} \log M_j (|y| - \lambda)\lambda^{-1} ~~~~~~ & \textmd{if} ~ \alpha =n,\\
C M_j^{-\frac{2\alpha}{n-2\sigma}} l_j^{\alpha-n} (|y| - \lambda)\lambda^{-1} ~~~~~~ & \textmd{if} ~ 1 < \alpha <n. 
\end{cases}
\endaligned
$$
(2) $4\lambda \leq  |y| \leq l_j$.  Since $|z| \geq 2l_j$ and thus $|y| \leq |z|/2$,  we  have  $G(0, \lambda; y, z) \leq C |y -z|^{2\sigma-n} \leq C |z|^{2\sigma-n}$.  Hence 
$$
\aligned
g_\lambda(y) & \leq C M_j^{-\frac{2\alpha}{n-2\sigma}}    \int_{\Omega_j \backslash B_{2l_j}}  |z|^{\alpha-2n} dz\\
& \leq 
\begin{cases}
C M_j^{-\frac{2n}{n-2\sigma}}  ~~~~~~ & \textmd{if} ~ \alpha>n, \\
C M_j^{-\frac{2n}{n-2\sigma}} \log M_j  ~~~~~~ & \textmd{if} ~ \alpha =n,\\
C M_j^{-\frac{2\alpha}{n-2\sigma}} l_j^{\alpha-n}  ~~~~~~ & \textmd{if} ~ 1 < \alpha <n. 
\end{cases}
\endaligned
$$
(3) $l_j \leq |y| \leq 3L_j$.  We separate the region $\Omega_j \backslash B_{2l_j}$ into four parts: 
$$
\mathcal{E}_1 = \{ z \in \Omega_j \backslash B_{2l_j}: |z| < |y|/2 \}, 
$$
$$
\mathcal{E}_2 = \{ z \in \Omega_j \backslash B_{2l_j}: |z| > 2|y| \}, 
$$
$$
\mathcal{E}_3 = \{ z \in \Omega_j \backslash B_{2l_j}: |z-y| \leq  |y|/2 \}, 
$$
$$
\mathcal{E}_4 = \{ z \in \Omega_j \backslash B_{2l_j}: |z-y| \geq  |y|/2  ~ \textmd{and} ~ |y|/2 \leq |z| \leq  2|y|\}. 
$$
Note that some of these sets may be empty. Then
$$
\aligned
g_\lambda(y) & \leq C M_j^{-\frac{2\alpha}{n-2\sigma}}    \int_{\Omega_j \backslash B_{2l_j}} |y -z|^{2\sigma -n} |z|^{\alpha-2\sigma-n} dz \\
& \leq  C M_j^{-\frac{2\alpha}{n-2\sigma}}   \bigg(  |y|^{2\sigma -n} \int_{\mathcal{E}_1}  |z|^{\alpha-2\sigma-n} dz  + \int_{\mathcal{E}_2}  |z|^{2\sigma-n} |z|^{\alpha-2\sigma-n} dz  \\
& ~~~~~~  + |y|^{\alpha-2\sigma-n} \int_{\mathcal{E}_3}  |y - z|^{2\sigma-n} dz + |y|^{2\sigma-n} \int_{\mathcal{E}_4}  |z|^{\alpha-2\sigma-n} dz \bigg) \\ 
& \leq C M_j^{-\frac{2\alpha}{n-2\sigma}}  \bigg( |y|^{2\sigma -n} \int_{2l_j}^{|y|/2} r^{\alpha -2\sigma-1} dr + \int_{2|y|}^{3L_j} r^{\alpha -n-1} dr  \\
& ~~~~~~  + |y|^{\alpha - n} + |y|^{2\sigma-n} \int_{|y|/2}^{2|y|} r^{\alpha-2\sigma-1} dr \bigg) \\
& \leq 
\begin{cases}
C M_j^{-\frac{2n}{n-2\sigma}}  ~~~~~~ & \textmd{if} ~ \alpha>n, \\
C M_j^{-\frac{2n}{n-2\sigma}} \log M_j  ~~~~~~ & \textmd{if} ~ \alpha =n,\\
C M_j^{-\frac{2\alpha}{n-2\sigma}} |y|^{\alpha -n}   ~~~~~~ & \textmd{if} ~ 2 \sigma< \alpha <n, \\
C M_j^{-\frac{2\alpha}{n-2\sigma}} |y|^{2\sigma-n} \log |y|  ~~~~~~ & \textmd{if} ~ \alpha =2 \sigma, \\
C M_j^{-\frac{2\alpha}{n-2\sigma}} l_j^{\alpha-2\sigma} |y|^{2\sigma -n} ~~~~~~ & \textmd{if} ~ 1< \alpha < 2\sigma. 
\end{cases}
\endaligned
$$
The proof of Lemma \ref{Lg_lammfb=01} is finished. 
\end{proof} 

Next we consider two cases for $\alpha>1$ separately.

\vskip0.10in 

{\it Case 1:   $\alpha<2\sigma$ or $\alpha \geq (n-2\sigma)/2$}. We construct a function $H_\lambda$ which is non-positive on $\Pi_j \backslash \overline{B}_\lambda$.  Let  
\begin{equation}\label{318=yhzmfn-01} 
H_\lambda(y)  = - \varepsilon M_j^{-1} (\lambda^{2\sigma-n} - |y|^{2\sigma -n} ) + \Phi_\lambda(y), ~~~~~ y \in \Pi_j \backslash \overline{B}_\lambda
\end{equation} 
for some small $\varepsilon \in (0, \varepsilon_0/4)$ where $\varepsilon_0$ is defined in Lemma \ref{LAn-01}.  It is clear that  Lemma \ref{hhff=01} still holds, and by using Lemma  \ref{hhff=01} we can choose a small $\varepsilon>0$ such that  for any $ \lambda \in [\lambda_0, \lambda_1]$, 
\begin{equation}\label{318-wut4g=001}
J_\lambda(y) - \varepsilon M_j^{-1} (\lambda^{2\sigma-n} - |y|^{2\sigma -n} ) \geq \frac{1}{2} J_\lambda(y),  ~~~~~ y \in \Pi_j \backslash \overline{B}_\lambda. 
\end{equation}
Here we recall that $\lambda_0=R-2$ and $\lambda_1=R+2$.   Then by Lemmas \ref{L-swjhdg-010} and \ref{Lg_lammfb=01} we have for $\lambda\in [\lambda_0, \lambda_1]$ and  for large $j$, 
\begin{equation}\label{318+sqfto84n-01} 
H_\lambda(y)  < 0, ~~~~~ y \in \Pi_j \backslash \overline{B}_\lambda.
\end{equation} 
Moreover, it follows from Lemmas \ref{LAn-01}, \ref{L-swjhdg-010} and \ref{Lg_lammfb=01} that for large $j$, 
$$
\varphi_{\lambda_0}(y) +  H_{\lambda_0}(y) \geq 0,  ~~~~~~~  y \in \Pi_j \backslash \overline{B}_{\lambda_0}. 
$$
By \eqref{Ste3-BFh-01} and \eqref{318-wut4g=001},  $\varphi_{\lambda} +  H_{\lambda}$ satisfies 
\begin{equation}\label{Ste3-aj65gbk69-01}
\varphi_\lambda(y) +  H_\lambda(y)  \geq \int_{\Pi_j \backslash B_\lambda} G(0, \lambda; y, z) b_\lambda(z) \varphi_\lambda(z) dz  + \frac{1}{2}J_\lambda(y) ~~~~~ \textmd{for} ~ y \in \Pi_j \backslash \overline{B}_\lambda, 
\end{equation}
where $\lambda \in [\lambda_0, \lambda_1]$.  We define
$$
\bar{\lambda}:= \sup\{ \mu \geq \lambda_0 ~ | ~ \varphi_{\lambda}(y) +  H_{\lambda}(y) \geq 0, ~ \forall ~ y \in \Pi_j \backslash \overline{B}_{\lambda}, ~ \forall ~ \lambda_0 \leq  \lambda \leq  \mu \}. 
$$
Then,  $\bar{\lambda}$ is well defined and $\bar{\lambda}  \geq \lambda_0$ for all large $j$. Furthermore, by \eqref{318+sqfto84n-01} and Lemma \ref{LAn-01} we have $\bar{\lambda} <  \lambda_1$ for all large $j$. Then, as in Step 2,  applying the moving sphere method one can derive a contradiction to the definition of $\bar{\lambda}$.  The proof of Theorem \ref{IEthm01}  in Case 1 is finished.     

\vskip0.10in 

{\it Case 2:   $2\sigma  \leq \alpha <  (n-2\sigma)/2$}. By \eqref{ST=-kfh02} we can choose a constant $\hat{C}>0$  such that 
\begin{equation}\label{319-hYm==01}
 \hat{C} M_j^{-\frac{2\alpha}{n-2\sigma}} |z|^{\alpha-2\sigma-n} - Q_\lambda(z)  \geq M_j^{-\frac{2\alpha}{n-2\sigma}} |z|^{\alpha-2\sigma-n} ~~~~~~ \textmd{for} ~ z \in \Omega_j \backslash B_{2l_j}. 
\end{equation} 
Let
\begin{equation}\label{Tiyyt=20210624}
T_\lambda(y)  = \hat{C} M_j^{-\frac{2\alpha}{n-2\sigma}}  \int_{\Omega_j \backslash B_{2l_j}} G(0, \lambda; y, z) |z|^{\alpha-2\sigma-n} dz ~~~~~~ \textmd{for} ~ \lambda < |y| \leq 3L_j.
\end{equation} 
As a consequence of Lemma \ref{Lg_lammfb=01}, we have the following estimates for $T_\lambda$.
\begin{lemma}\label{319=ehciu23t}
Suppose $2\sigma  \leq \alpha <  (n-2\sigma)/2$.  For $\lambda < |y| \leq 4\lambda$, we have 
\begin{equation}\label{319=anjfpk2083+01}
T_\lambda(y) \leq C M_j^{-\frac{2}{n-2\sigma}} |x_j|^{\alpha-1} (|y| - \lambda) \lambda^{-n}. 
\end{equation} 
For $4\lambda \leq |y| \leq 2l_j$, we have 
\begin{equation}\label{319=anjfpk2083+02}
T_\lambda(y) \leq C M_j^{-\frac{2}{n-2\sigma}} |x_j|^{\alpha-1} |y|^{2\sigma -n} \lambda^{1-2\sigma}. 
\end{equation} 
For $2l_j \leq |y| \leq 3L_j$,   we have 
\begin{equation}\label{319=anjfpk2083+03}
T_\lambda(y) \leq 
\begin{cases}
C M_j^{-\frac{2\alpha}{n-2\sigma}} |y|^{\alpha -n}   ~~~~~~ & \textmd{if}  ~ 2 \sigma< \alpha < (n-2\sigma)/2, \\
C M_j^{-\frac{4\sigma}{n-2\sigma}} |y|^{2\sigma-n} \log |y|  ~~~~~~ & \textmd{if}  ~ \alpha =2 \sigma. 
\end{cases} 
\end{equation} 
\end{lemma} 
\begin{proof}
When $\lambda < |y| \leq 4\lambda$, by \eqref{317=yfmh-001} we obtain  
$$
T_\lambda(y) \leq C M_j^{-\frac{2\alpha}{n-2\sigma}} l_j^{\alpha-1} \cdot l_j^{1-n} (|y| - \lambda)\lambda^{-1} \leq C M_j^{-\frac{2}{n-2\sigma}}  |x_j|^{\alpha-1}  (|y| - \lambda)\lambda^{-n}, 
$$
where we used $l_j=M_j^{\frac{2}{n-2\sigma}} |x_j|$ and $l_j \to \infty$. 

When $4\lambda\leq |y| \leq 3L_j$, by \eqref{317=yfmh-002}  and   \eqref{317=yfmh-003} we have 
$$
T_\lambda(y) \leq 
\begin{cases}
C M_j^{-\frac{2\alpha}{n-2\sigma}} |y|^{\alpha -n}   ~~~~~~ & \textmd{if}  ~ 2 \sigma< \alpha < (n-2\sigma)/2, \\
C M_j^{-\frac{4\sigma}{n-2\sigma}} |y|^{2\sigma-n} \log |y|  ~~~~~~ & \textmd{if}  ~ \alpha =2 \sigma. 
\end{cases} 
$$
In particular, for $4\lambda \leq |y| \leq 2l_j$ we have 
$$
\aligned
T_\lambda(y) & \leq  C M_j^{-\frac{2\alpha}{n-2\sigma}} |y|^{\alpha -n} \log|y| \\
& \leq C M_j^{-\frac{2\alpha}{n-2\sigma}} |y|^{\alpha -2\sigma}  |y|^{2\sigma-n} \log|y| \\
& \leq C M_j^{-\frac{2}{n-2\sigma}} |x_j|^{\alpha -1}  |y|^{2\sigma-n}  l_j^{1-2\sigma} \log l_j\\
& \leq C M_j^{-\frac{2}{n-2\sigma}} |x_j|^{\alpha-1} |y|^{2\sigma -n} \lambda^{1-2\sigma}, 
\endaligned
$$ 
where we used $\sigma > \frac{1}{2}$ and $l_j \to \infty$.  Lemma \ref{319=ehciu23t} is proved.   
\end{proof}

Define 
$$
H_\lambda(y) := \Phi_{1, \lambda}(y) + T_\lambda(y) ~~~~~~ \textmd{for} ~ y \in \Pi_j \backslash \overline{B}_{\lambda}. 
$$
Then from Lemmas  \ref{L-swjhdg-010}  and \ref{319=ehciu23t} we obtain 
\begin{equation}\label{319-43hym=01}
H_\lambda(y) \leq
\begin{cases}
- C M_j^{-\frac{2}{n-2\sigma}} |x_j|^{\alpha-1}  (|y| - \lambda) \lambda^{-n} \log \lambda   & \textmd{for} ~ y \in B_{4\lambda} \backslash \overline{B}_\lambda, \\
-C M_j^{-\frac{2}{n-2\sigma}} |x_j|^{\alpha-1}  |y|^{2\sigma -n} \lambda^{1-2\sigma} \log \lambda   & \textmd{for} ~ y \in B_{2l_j} \backslash B_{4\lambda},\\
C M_j^{-\frac{2\alpha}{n-2\sigma}} |y|^{\alpha -n}    & \textmd{for} ~ y \in \Pi_j \backslash B_{2l_j},   ~ 2 \sigma< \alpha < (n-2\sigma)/2, \\
C M_j^{-\frac{4\sigma}{n-2\sigma}} |y|^{2\sigma-n} \log |y|   & \textmd{for} ~ y \in \Pi_j \backslash B_{2l_j},~  \alpha =2 \sigma 
\end{cases}
\end{equation} 
and 
\begin{equation}\label{319-43hym=03}
|H_{\lambda} (y)| \leq 
\begin{cases}
C M_j^{-\frac{2}{n-2\sigma}} |x_j|^{\alpha-1}  (|y| - \lambda) \lambda^{2\sigma}  ~~~~~~ & \textmd{for} ~ y \in B_{4\lambda} \backslash \overline{B}_\lambda, \\
C M_j^{-\frac{2}{n-2\sigma}} |x_j|^{\alpha-1}  |y|^{2\sigma -n} \lambda^{n+1} ~~~~~~ & \textmd{for} ~ y \in  B_{2l_j}  \backslash  B_{4\lambda}, \\
\theta_j |y|^{2\sigma-n} & \textmd{for} ~ y \in  \Pi_j  \backslash B_{2l_j}, 
\end{cases}
\end{equation}
where $\{ \theta_j \}$ is a positive sequence satisfying $\theta_j \to 0$ as $j \to \infty$.  Remark that in order to obtain  \eqref{319-43hym=01},  we could take $R$ larger and then fix it.   Furthermore, by \eqref{Ste3-BFh-01} and \eqref{319-hYm==01} we see that $\varphi_\lambda$ satisfies 
\begin{equation}\label{318=09sdbfjke}
\aligned
\varphi_\lambda(y) +  H_\lambda(y) & \geq \int_{\Pi_j \backslash B_\lambda} G(0, \lambda; y, z) b_\lambda(z) \varphi_\lambda(z) dz  \\
& ~~~~ + M_j^{-\frac{2\alpha}{n-2\sigma}}\int_{\Omega_j \backslash B_{2l_j}} G(0, \lambda; y, z)  |z|^{\alpha-2\sigma-n} dz + J_\lambda(y),  ~~~~~ y \in \Pi_j \backslash \overline{B}_\lambda, 
 \endaligned
\end{equation}
where $\lambda \in [\lambda_0, \lambda_1]$ with $\lambda_0=R-2$ and $\lambda_1=R+2$.  Moreover, it follows from Lemma \ref{LAn-01} and \eqref{319-43hym=03} that for large $j$,  
$$
\varphi_{\lambda_0}(y) +  H_{\lambda_0}(y) \geq 0,  ~~~~~~~  y \in \Pi_j \backslash \overline{B}_{\lambda_0}. 
$$
We define
$$
\bar{\lambda}:= \sup\{ \mu \geq \lambda_0 ~ | ~ \varphi_{\lambda}(y) +  H_{\lambda}(y) \geq 0, ~ \forall ~ y \in \Pi_j \backslash \overline{B}_{\lambda}, ~ \forall ~ \lambda_0 \leq  \lambda \leq  \mu \}. 
$$
Then,  $\bar{\lambda}$ is well defined and $\bar{\lambda}  \geq \lambda_0$ for all large $j$. Furthermore, by \eqref{319-43hym=01} and Lemma \ref{LAn-01} we have $\bar{\lambda} <  \lambda_1$ for all large $j$. 
Using the continuity we obtain 
$$
\varphi_{\bar{\lambda}}(y) +  H_{\bar{\lambda}}(y) \geq 0,  ~~~~~~~  y \in \Pi_j \backslash \overline{B}_{\bar{\lambda}}. 
$$
This together with Lemma \ref{L-swjhdg-010} and \eqref{319-43hym=01} yields 
\begin{equation}\label{319BYhf=01}
\varphi_{\bar{\lambda}}(y) \geq -H_{\bar{\lambda}}(y) \geq 
\begin{cases} 
0, ~~~~~~~  & y \in B_{2l_j} \backslash (\overline{B}_{\bar{\lambda}} \cup \Sigma_j), \\
 -T_{\bar{\lambda}}(y),  ~~~~~~~  &y \in \Pi_j \backslash B_{2l_j}, 
\end{cases} 
\end{equation}
where $\Sigma_j$ is the singular set of $w_j$. 

Let  
$$
\mathcal{O}_\lambda= \{y \in \Pi_j \backslash B_{2l_j}:  w_j(y) <   w_j^\lambda(y) \}. 
$$
By \eqref{318=09sdbfjke} and \eqref{319BYhf=01}, 
$$
\aligned
\varphi_{\bar{\lambda}}(y) +  H_{\bar{\lambda}}(y) & \geq \int_{\Pi_j \backslash B_{2l_j}} G(0, \bar{\lambda}; y, z) b_{\bar{\lambda}}(z) \varphi_{\bar{\lambda}}(z) dz  \\
& ~~~~ + M_j^{-\frac{2\alpha}{n-2\sigma}}\int_{\Omega_j \backslash B_{2l_j}} G(0, \bar{\lambda}; y, z)  |z|^{\alpha-2\sigma-n} dz + J_{\bar{\lambda}}(y)  \\
& \geq - \int_{ \mathcal{O}_{\bar{\lambda}} } G(0, \bar{\lambda}; y, z) b_{\bar{\lambda}}(z) T_{\bar{\lambda}} (z) dz  + \int_{(\Pi_j\backslash B_{2l_j}) \backslash \mathcal{O}_{\bar{\lambda}}} G(0, \bar{\lambda}; y, z) b_{\bar{\lambda}}(z) \varphi_{\bar{\lambda}}(z) dz \\  
& ~~~~ + M_j^{-\frac{2\alpha}{n-2\sigma}}\int_{\Omega_j \backslash B_{2l_j}} G(0, \bar{\lambda}; y, z)  |z|^{\alpha-2\sigma-n} dz + J_{\bar{\lambda}}(y),  ~~~~~ y \in \Pi_j \backslash \overline{B}_{\bar{\lambda}}. \\
 \endaligned
$$
Note that on $\mathcal{O}_{\bar{\lambda}}$, by Lemma \ref{Le-hm-074} we have 
$$
0 \leq b_{\bar{\lambda}}(z) =K_j(z) \frac{ w_j(z)^{\frac{n+2\sigma}{n-2\sigma}} -  w_j^{\bar{\lambda}}(z)^{\frac{n+2\sigma}{n-2\sigma}} }{w_j(z) - w_j^{\bar{\lambda}}(z)} \leq C w_j^{\bar{\lambda}} (z)^{\frac{4\sigma}{n-2\sigma}} \leq C |z|^{-4 \sigma}. 
$$
Thus,  Lemma \ref{319=ehciu23t} gives 
$$
\aligned
& - \int_{ \mathcal{O}_{\bar{\lambda}} } G(0, \bar{\lambda}; y, z) b_{\bar{\lambda}}(z) T_{\bar{\lambda}} (z) dz + M_j^{-\frac{2\alpha}{n-2\sigma}}\int_{\Omega_j \backslash B_{2l_j}} G(0, \bar{\lambda}; y, z)  |z|^{\alpha-2\sigma-n} dz \\
& ~~~~~ \geq M_j^{-\frac{2\alpha}{n-2\sigma}}  \int_{ \mathcal{O}_{\bar{\lambda}} } G(0, \bar{\lambda}; y, z) \left( |z|^{\alpha-2\sigma-n}  - C|z|^{\alpha -4 \sigma -n} \log |z| \right) dz \geq 0 
\endaligned 
$$
when $j$ is large.  Therefore 
$$
\varphi_{\bar{\lambda}}(y) +  H_{\bar{\lambda}}(y) \geq J_{\bar{\lambda}}(y),  ~~~~~ y \in \Pi_j \backslash \overline{B}_{\bar{\lambda}}. 
$$
Lemma \ref{hhff=01} is clearly still true.  It follows from  Lemma \ref{hhff=01}  that there exists a small $\varepsilon_1 \in (0, \lambda_1-\bar{\lambda})$  (which depends on $j$)  such that 
$$
\varphi_{\bar{\lambda}}(y) +  H_{\bar{\lambda}}(y) \geq J_{\bar{\lambda}}(y) \geq C u(x_j)^{-1} \geq \frac{\varepsilon_1}{|y|^{n-2\sigma}} ~~~~~~ \forall ~ |y| \geq  \lambda_1 + 1,  y \in \Pi_j. 
$$
By the above estimate and the explicit formulas for  $w_j^{\lambda}(y)$ and $H_\lambda(y)$, there exists $\varepsilon_2 \in (0, \varepsilon_1)$ such that for any $\bar{\lambda} \leq \lambda \leq \bar{\lambda} +\varepsilon_2$, 
\begin{equation}\label{16wj-=0934}
\aligned
\varphi_\lambda(y) + H_\lambda(y)  & = [\varphi_{\bar{\lambda}}(y)  +  H_{\bar{\lambda}}(y)] + [H_{\lambda}(y) - H_{\bar{\lambda}}(y)] + [w_j^{\bar{\lambda}}(y) - w_j^\lambda(y) ] \\
& \geq  \frac{\varepsilon_1}{2|y|^{n-2\sigma}}~~~~~~ \forall ~ |y| \geq  \lambda_1 +1,  ~ y\in \Pi_j. 
\endaligned
\end{equation}
This,  together with Lemma \ref{L-swjhdg-010} and \eqref{319-43hym=01},  implies that for any $\bar{\lambda} \leq \lambda \leq \bar{\lambda} +\varepsilon_2$, 
\begin{equation}\label{37dvYBF-01}
w_j(y) - w_j^\lambda(y) \geq -H_\lambda(y) +  \frac{\varepsilon_1}{2|y|^{n-2\sigma}} \geq 
\begin{cases}
\frac{\varepsilon_1}{2|y|^{n-2\sigma}}, ~~~ & y \in B_{2l_j} \backslash \Sigma_j, ~  |y| \geq  \lambda_1 +1,  \\
 -T_\lambda(y), ~~~ & y \in  \Pi_j \backslash B_{2l_j}. 
 \end{cases} 
\end{equation}
For $\varepsilon \in \left( 0, \varepsilon_2 \right)$ which we choose below, by \eqref{318=09sdbfjke}, \eqref{37dvYBF-01} and \eqref{y-05},  we have, for $\bar{\lambda} \leq \lambda \leq \bar{\lambda} +\varepsilon $ and for $\lambda \leq |y| \leq \lambda_1 +1$,  
$$
\aligned
\varphi_\lambda(y) +  H_\lambda(y)   & \geq \int_{\lambda \leq |z| \leq \lambda_1 +1}  G(0,\lambda; y, z) K_j(z) \left( w_j(z)^{\frac{n+2\sigma}{n-2\sigma}} - w_j^\lambda(z)^{\frac{n+2\sigma}{n-2\sigma}} \right) dz\\
& ~~~ +     \int_{\lambda_1 +2  \leq |z| \leq \lambda_1 +3}  G(0,\lambda; y, z) K_j(z) \left( w_j(z)^{\frac{n+2\sigma}{n-2\sigma}} - w_j^\lambda(z)^{\frac{n+2\sigma}{n-2\sigma}} \right) dz \\ 
& ~~~ - \int_{\mathcal{O}_\lambda} G(0, \lambda; y, z) b_\lambda(z) T_\lambda(z) dz  + \int_{(\Pi_j\backslash B_{2l_j}) \backslash \mathcal{O}_{\lambda}} G(0, \lambda; y, z) b_{\lambda}(z) \varphi_{\lambda}(z) dz \\
& ~~~ + M_j^{-\frac{2\alpha}{n-2\sigma}}\int_{\Omega_j \backslash B_{2l_j}} G(0, \lambda; y, z)  |z|^{\alpha-2\sigma-n} dz. 
\endaligned
$$
On  the set $\mathcal{O}_{\lambda}$,  Lemma \ref{Le-hm-074} leads to 
$$
0 \leq b_{\lambda}(z)  \leq C w_j^{\lambda} (z)^{\frac{4\sigma}{n-2\sigma}} \leq C |z|^{-4 \sigma}, 
$$
thus,  using Lemma \ref{319=ehciu23t} yields that for $j$  large, 
$$
\aligned
& - \int_{ \mathcal{O}_{\lambda} } G(0, \lambda; y, z) b_{\lambda}(z) T_{\lambda} (z) dz + M_j^{-\frac{2\alpha}{n-2\sigma}}\int_{\Omega_j \backslash B_{2l_j}} G(0, \lambda; y, z)  |z|^{\alpha-2\sigma-n} dz \\
& ~~~~~ \geq M_j^{-\frac{2\alpha}{n-2\sigma}}  \int_{ \mathcal{O}_{\lambda} } G(0, \lambda; y, z) \left( |z|^{\alpha-2\sigma-n}  - C|z|^{\alpha -4 \sigma -n} \log |z| \right) dz \geq 0. 
\endaligned 
$$
Hence,  for $\bar{\lambda} \leq \lambda \leq \bar{\lambda} +\varepsilon $ and for $\lambda \leq |y| \leq \lambda_1 +1$, we have  
$$
\aligned
\varphi_\lambda(y) +  H_\lambda(y)   &  \geq \int_{\lambda \leq |z| \leq \lambda_1 +1}  G(0,\lambda; y, z) K_j(z) \left( w_j(z)^{\frac{n+2\sigma}{n-2\sigma}} - w_j^\lambda(z)^{\frac{n+2\sigma}{n-2\sigma}} \right) dz\\
& ~~~ +     \int_{\lambda_1 +2  \leq |z| \leq \lambda_1 +3}  G(0,\lambda; y, z) K_j(z) \left( w_j(z)^{\frac{n+2\sigma}{n-2\sigma}} - w_j^\lambda(z)^{\frac{n+2\sigma}{n-2\sigma}} \right) dz \\ 
& \geq -C \int_{\lambda \leq |z| \leq \lambda +\varepsilon}   G(0,\lambda; y, z) (|z| -\lambda) dz\\
& ~~~ + \int_{\lambda+\varepsilon \leq |z| \leq \lambda_1 +1}  G(0,\lambda; y, z) K_j(z) \left( w_j^{\bar{\lambda}}(z)^{\frac{n+2\sigma}{n-2\sigma}} - w_j^\lambda(z)^{\frac{n+2\sigma}{n-2\sigma}} \right) dz \\
& ~~~ + A_1\int_{\lambda_1 +2  \leq |z| \leq \lambda_1 +3 }  G(0,\lambda; y, z) \left( w_j(z)^{\frac{n+2\sigma}{n-2\sigma}} - w_j^\lambda(z)^{\frac{n+2\sigma}{n-2\sigma}} \right) dz,
\endaligned
$$
where $A_1=\min_{B_2} K >0$, and we have  used \eqref{319BYhf=01} and 
$$
| w_j(z)^{\frac{n+2\sigma}{n-2\sigma}} - (w_j)_\lambda(z)^{\frac{n+2\sigma}{n-2\sigma}}  |\leq C(|z| - \lambda)
$$
in the second inequality. By \eqref{37dvYBF-01}  there exists $\hat{\delta} >0$ (which depends on $j$) such that
$$
w_j(z)^{\frac{n+2\sigma}{n-2\sigma}} - w_j^\lambda(z)^{\frac{n+2\sigma}{n-2\sigma}} \geq \hat{\delta} ~~~~~~ \textmd{for}  ~ \lambda_1 +2  \leq |z| \leq \lambda_1 +3.  
$$
Since $\|w_j\|_{C^1(B_{\lambda_1 + 2})}\leq C$ and $\|K_j\|_{L^\infty(B_{\lambda_1 + 2})} \leq C $ for some constant $C$ independent of $j$,  there exists some constant $C>0$ independent of both $\varepsilon$ and $j$  such that for any $\bar{\lambda} \leq \lambda \leq \bar{\lambda} +\varepsilon$, 
$$
K_j(z)| w_j^{\bar{\lambda}}(z)^{\frac{n+2\sigma}{n-2\sigma}} - w_j^\lambda(z)^{\frac{n+2\sigma}{n-2\sigma}}  |\leq C(\lambda - \bar{\lambda})\leq C \varepsilon~~~~~ \forall ~ \lambda \leq |z| \leq  \lambda_1 +1.
$$
For any  $\lambda \leq |y| \leq \lambda_1+1$, we can estimate the kernel $G$ as in Step 2 to obtain   
$$
\int_{\lambda+\varepsilon  \leq |z| \leq \lambda_1 +1}  G(0,\lambda; y, z) dz    \leq C (\varepsilon^{2\sigma-1}  +  |\ln \varepsilon|  +1) (|y| - \lambda) 
$$
and 
$$
\int_{ \lambda \leq |z| \leq \lambda +\varepsilon }   G(0,\lambda; y, z) (|z| -\lambda) dz  \leq  C (|y| - \lambda) \varepsilon^{2\sigma/n}. 
$$
Thus, using  \eqref{K011} we have  for  $\lambda < |y| \leq \lambda_1 + 1$ that 
$$
\aligned
 \varphi_\lambda(y) + H_\lambda(y) & \geq -C \varepsilon^{2\sigma/n} (|y| - \lambda) + \delta_1\hat{\delta} (|y| - \lambda) \int_{\lambda_1 +2  \leq |z| \leq \lambda_1 +3} \frac{1}{|y-z|} dz\\
&  \geq \left(\delta_1\hat{\delta} c -C \varepsilon^{2\sigma/n} \right) (|y| - \lambda) \geq 0
\endaligned
$$
if $\varepsilon$ is sufficiently small.  This  and  \eqref{16wj-=0934}  contradict the definition of $\bar{\lambda}$, and so the proof  in Case 2 is finished.  The proof of Theorem \ref{IEthm01} under the assumption (K1) is completed. 
\hfill$\square$

\vskip0.10in

\section{Local estimates under the assumption (K2)}\label{S5000} 
In this section, by using the method of moving spheres introduced by Li-Zhu \cite{Li-Zhu,Li04},  we shall prove Theorem \ref{IEthm01} under the assumption (K2) in the spirit of the works of Chen-Lin \cite{Chen-Lin97}  and Taliaferro-Zhang \cite{TZ}.   Similar to Section \ref{S3000}, we have to set up a framework to fit the integral equation rather than dealing directly with  differential equations as in \cite{Chen-Lin97,TZ}.  In addition to developing integral  techniques to overcome the lack of maximum principle, the non-locality of integral equation will bring a new difficulty under the present assumption (K2).  

\vskip0.10in

\noindent{\it Proof of Theorem \ref{IEthm01} under the assumption (K2). }  Suppose by contradiction that there exists a sequence $\{x_j\}_{j=_1}^\infty  \subset B_1 \backslash \Sigma$ such that   
$$
d_j:=\textmd{dist} (x_j, \Sigma) \to 0~~~~~ \textmd{as}  ~ j\to \infty, 
$$
but
$$
d_j^{\frac{n-2\sigma}{2}} u(x_j) \to \infty~~~~~ \textmd{as}  ~ j\to \infty. 
$$
Without  loss of generality we may assume that $0\in \Sigma$ and $x_j\to 0$ as $j\to \infty$. 

\vskip0.10in 

{\it Step 1.  We show that $x_j$ can be chosen as the local maximum points of $u$. Moreover, the functions $u(x_j)^{-1} u( x_j +  u(x_j)^{-\frac{2}{n-2\sigma}} y)$ converge in $C_{loc}^2(\mathbb{R}^n)$, after passing a subsequence,  to a positive function $U_0 \in C^2(\mathbb{R}^n)$ where $U_0$ satisfies   
\begin{equation}\label{S3-W1-0376+08g}
\begin{cases}
U_0(y) = \int_{\mathbb{R}^n} \frac{K(0) U_0(z)^{\frac{n+2\sigma}{n-2\sigma}}}{|y -z |^{n-2\sigma}} dz~~~~~ \textmd{for} ~ y \in \mathbb{R}^n, \\
\max_{\mathbb{R}^n} U_0=U_0(0)=1.
\end{cases}
\end{equation}
}

{\it Step 2. For $n > 2\sigma$, we have  $\nabla K(0)=0$ under the assumption of $K \in C^1(B_2)$. }

\vskip0.10in 

The proofs of these two steps are the same as those of Step 1 and Step 2 in Section \ref{S3000}, we omit their proofs.  Since the following proof is long, we first explain the idea and the sketch of the proof.   Let $w_j$ be the scaled function in \eqref{5SqSS3-01} (here we don't need to shift it) and let $w_j^\lambda$ be the Kelvin transformation of $w_j$ in \eqref{5Sq21SS3-022}.  In order to take full advantage of  the property in the assumption (K2) that $c(\cdot)$ is sufficiently small near $0$,  we will restrict the integral equation \eqref{Int} to a small ball.  Thus,  for some small $\tau \in (0, \frac{1}{4})$ we denote 
$$
\Pi_j:=  \left\{y\in \mathbb{R}^n :   x_j + u(x_j)^{-\frac{2}{n-2\sigma}} y  \in  B_\tau  \backslash \Sigma \right\}. 
$$
Note that the set $\Pi_j$ depends on $\tau$. Let $\varphi_\lambda=w_j - w_j^\lambda$ and let $\Phi_\lambda$ be as in \eqref{SL4-31=huy2=02} where $\lambda \in [1/2, 2]$.  Then $\varphi_\lambda + \Phi_\lambda$ satisfies the integral inequality \eqref{SL4-312=04=huy}.   Moreover, by Lemma \ref{S5^326=01} the remainder term $J_\lambda$ has a lower bound which is independent of $\tau$.  After performing the above restriction to \eqref{Int}, we need extra efforts to obtain estimates for $\Phi_\lambda$.  Based on the assumption (K2) about the function $K$, we will complete the proof in three situations. Denote   $M_j=u(x_j)$. 

\begin{itemize}

\item [$(1)$] $2\sigma < n < 2\sigma +2$.  By the assumption (K2) we know that $K\in \mathcal{C}^\alpha(B_2)$ with $\alpha=\frac{n-2\sigma}{2}<1$. 
Furthermore,  $\Phi_\lambda$ satisfies the estimates in Lemma \ref{LemmaS5-327=01}.  By choosing sufficiently small $\varepsilon_1>0$  and $\tau>0$ we can construct  
$$
H_\lambda(y)  = - \varepsilon_1 M_j^{-1} (\lambda^{2\sigma-n} - |y|^{2\sigma -n} ) + \Phi_\lambda(y), ~~~~~~ y \in \Pi_j \backslash \overline{B}_\lambda  
$$
such that $H_\lambda <0$  for all $\lambda \in [1/2, 2]$.  With the help of  Lemma \ref{S5^326=01}, the method of moving spheres can be applied to $\varphi_\lambda + H_\lambda$ to reach a contradiction.  
\item [$(2)$] $n=2\sigma+2$. In this case,  we have $K\in C^1(B_2)$.  By Step 2  we know $\nabla K(0)=0$. Similar to situation $(1)$, $\Phi_\lambda$ also satisfies the estimates in Lemma \ref{LemmaS5-327=01}. The rest of the proof is the same as that of situation $(1)$. 

\item [$(3)$] $n > 2\sigma+2$. By the assumption (K2) we have $K\in \mathcal{C}^\alpha(B_2)$ with $\alpha=\frac{n-2\sigma}{2}>1$.  It follows from Step 2  that $|\nabla K(x_j)| \to 0$.   Furthermore, using the moving sphere method  as in Step 2 of Section \ref{S3000}  and the assumption (K2) we obtain $|\nabla K(x_j)|^{\frac{1}{\alpha-1}} M_j^{\frac{2}{n-2\sigma}} \to 0$ (see Lemma \ref{StepL5-401=01}). This together with the assumption (K2) yields that  the estimates for  $\Phi_\lambda$ in Lemma \ref{LemmaS5-327=01} still hold.  The rest of the proof is the same as that of situation $(1)$. 

\end{itemize}

Now we return to the detailed proof of Theorem \ref{IEthm01} under the assumption (K2).   We denote  $M_j=u(x_j)$ and define  
\begin{equation}\label{5SqSS3-01}
w_j(y)=M_j^{-1} u \left( x_j + M_j^{-\frac{2}{n-2\sigma}} y \right), ~h_j(y)=M_j^{-1}  h \left( x_j + M_j^{-\frac{2}{n-2\sigma}} y \right) ~  \textmd{in}~ \Omega_j, 
\end{equation} 
where
\begin{equation}\label{5Sqhhmm-9e7=5}
\Omega_j =\left\{y\in \mathbb{R}^n :   x_j + M_j^{-\frac{2}{n-2\sigma}} y   \in  B_{2} \backslash \Sigma \right\}. 
\end{equation} 
It follows from Step 1 and the classification results in \cite{CLO06} or \cite{Li04} that, modulo a positive constant, $w_j(y)$ converges in $C^2$ norm to the standard bubble $U_0(y) = (1+ |y|^2)^{-\frac{n-2\sigma}{2}}$  on every compact subset of $\mathbb{R}^n$.  We also suppose that  $w_j$ is to be identically $0$ outside $\Omega_j$ and $K$ is  to be identically $0$ outside $B_{2}$.  For $\lambda >0$, let  
\begin{equation}\label{5Sq21SS3-022}
w_j^\lambda(y) =\left( \frac{\lambda}{|y|} \right)^{n-2\sigma} w_j \left( \frac{\lambda^2 y}{|y|^2} \right),~~h_j^\lambda(y) =  \left( \frac{\lambda}{|y|} \right)^{n-2\sigma}  h_j \left( \frac{\lambda^2 y}{|y|^2} \right)
\end{equation}
and let 
\begin{equation}\label{5Sq21SS3-033}
K_j(y) =K \left( x_j + M_j^{-\frac{2}{n-2\sigma}} y  \right). 
\end{equation}
Then 
\begin{equation}\label{5Sq21SS3-044}
w_j(y) = \int_{\mathbb{R}^n} \frac{K_j(z) w_j(z)^{\frac{n+2\sigma}{n-2\sigma}}}{|y -z |^{n-2\sigma}} dz +  h_j(y)~~~~~\textmd{for} ~ y\in \Omega_j. 
\end{equation}
By \eqref{ABC-01},  $w_j^\lambda$ satisfies  
\begin{equation}\label{5Sq21SS3-055}
w_j^\lambda(y) = \int_{\mathbb{R}^n} \frac{K_j(z^\lambda) w_j^\lambda(z)^{\frac{n+2\sigma}{n-2\sigma}}}{|y -z |^{n-2\sigma}} dz +   h_j^\lambda(y)~~~~~\textmd{for} ~ y \in \Omega_j \backslash   \overline{B}_{\lambda},  
\end{equation}
where $z^\lambda = \frac{\lambda^2 z}{|z|^2}$ is the inversion of $z$ with respect to $\partial B_\lambda$.

As Lemma \ref{LAn-01} we also have 
\begin{lemma}\label{gmhuL5SAn-01}
Let $\lambda_0= \frac{1}{2}$ and $\lambda_1 =2$.  Then there exist $\varepsilon_0>0$ and $j_0>1$ such that for all $j \geq j_0$, 
\begin{equation}\label{21hmuSS3&8-001}
w_j(y) - w_j^{\lambda_0}(y) \geq \varepsilon_0 (|y| - \lambda_0) |y|^{2\sigma-1-n} + \varepsilon_0 M_j^{-1} (\lambda_0^{2\sigma-n} - |y|^{2\sigma-n}), ~~~ y \in \Omega_j \backslash  \overline{B}_{\lambda_0}.
\end{equation}
Moreover, there exists $y^* \in B_{2\lambda_1} \backslash  \overline{B}_{\lambda_1}$ such that for all $j \geq j_0$, 
\begin{equation}\label{21hmuSS3&8-003}
w_j(y^*) -  w_j^{\lambda_1}(y^*) \leq -\varepsilon_0. 
\end{equation}
\end{lemma}

For some small $\tau \in (0, \frac{1}{4})$ to be determined later,  we denote 
\begin{equation}\label{21=huyPAI01}
\Pi_j:=  \left\{y\in \mathbb{R}^n :   x_j + M_j^{-\frac{2}{n-2\sigma}} y  \in  B_\tau  \backslash \Sigma \right\} \subset \Omega_j 
\end{equation} 
and
\begin{equation}\label{21=huyPAI02}
\Sigma_j:=\left\{y\in \mathbb{R}^n :  x_j + M_j^{-\frac{2}{n-2\sigma}} y   \in  \Sigma \right\}. 
\end{equation}
Note that we have $\mathcal{L}^n (\Sigma_j)=0$ due to $\mathcal{L}^n (\Sigma)=0$.   It follows from the same arguments as in Lemma 3.1 of \cite{JX19}  that  for all large $j$, there holds
\begin{equation}\label{3-Local3=huy}
h_j^{\lambda}(y) \leq h_j(y)~~~~~~ \forall ~  y\in \Pi_j \backslash \overline{B}_\lambda,  ~ \lambda\in [1/2,  2]. 
\end{equation}
Let $\varphi_\lambda(y) = w_j(y) - w_j^\lambda(y)$, where we omit $j$ in the notation for brevity.   By \eqref{ABC-0e571}, \eqref{3-Local3=huy} and $K_j \equiv 0$ on $\Omega_j^c \backslash \Sigma_j$,    we have for $\lambda\in [1/2,  2]$ and $y\in \Pi_j \backslash \overline{ B}_\lambda$ that 
\begin{equation}\label{yh38-01=huy}
\aligned
\varphi_\lambda(y) & \geq \int_{B_\lambda^c} G(0,\lambda; y, z) \left( K_j(z) w_j(z)^{\frac{n+2\sigma}{n-2\sigma}} - K_j(z^\lambda) w_j^\lambda(z)^{\frac{n+2\sigma}{n-2\sigma}} \right) dz  \\
&= \int_{\Pi_j \backslash B_\lambda} G(0, \lambda; y, z) b_\lambda(z) \varphi_\lambda(z) dz + J_\lambda(y) -\Phi_\lambda(y),
\endaligned
\end{equation}
where 
\begin{equation}\label{SL4-31=huy2=01}
b_\lambda(y)= K_j(y) \frac{ w_j(y)^{\frac{n+2\sigma}{n-2\sigma}} -  w_j^\lambda(y)^{\frac{n+2\sigma}{n-2\sigma}} }{w_j(y) - w_j^\lambda(y)},
\end{equation}
\begin{equation}\label{SL4-31=huy2=02}
\Phi_\lambda(y)= \int_{\Omega_j \backslash B_\lambda} G(0, \lambda; y, z) \left( K_j(z^\lambda) - K_j(z) \right) w_j^\lambda(z)^{\frac{n+2\sigma}{n-2\sigma}} dz 
\end{equation}
and
\begin{equation}\label{SL4-312=03=huy}
\aligned
J_\lambda(y) & = \int_{\Omega_j \backslash \Pi_j} G(0,\lambda; y, z) K_j(z)\left( w_j(z)^{\frac{n+2\sigma}{n-2\sigma}} -  w_j^\lambda(z)^{\frac{n+2\sigma}{n-2\sigma}} \right) dz\\
& ~~~~~ - \int_{\Omega_j^c} G(0,\lambda; y, z) K_j(z^\lambda) w_j^\lambda(z)^{\frac{n+2\sigma}{n-2\sigma}} dz.
\endaligned
\end{equation}
Notice that $b_\lambda(y)$  is always non-negative.  Thus, $\varphi_\lambda + \Phi_\lambda$ satisfies the integral inequality   
\begin{equation}\label{SL4-312=04=huy}
\varphi_\lambda(y) +  \Phi_\lambda(y) \geq \int_{\Pi_j \backslash B_\lambda} G(0, \lambda; y, z) b_\lambda(z) \varphi_\lambda(z) dz + J_\lambda(y),  ~~~~~ y \in \Pi_j \backslash \overline{B}_\lambda, 
\end{equation}
where $\lambda \in [1/2,  2]$. 

Now, we give an estimate for $J_\lambda$ defined in \eqref{SL4-312=03=huy}.  Suppose $\lambda_0=1/2$ and $\lambda_1=2$. 
\begin{lemma}\label{S5^326=01} 
There exists a constant $\beta_0>0$ independent of  $\tau$ such that for any $\lambda_0 \leq \lambda \leq \lambda_1$ and for all large $j$,    
\begin{equation}\label{S5^326y-05} 
J_\lambda(y)  \geq 
\begin{cases}
\beta_0 (|y| - \lambda) u(x_j)^{-1}, ~~~~~   &\textmd{if}~   \lambda \leq |y| \leq \lambda_1  +1, \\
\beta_0  u(x_j)^{-1}, ~~~~~ & \textmd{if}~ |y| > \lambda_1 +1,  ~ y \in \Pi_j. 
\end{cases}
\end{equation}
\end{lemma} 
\begin{proof} 
The proof is similar to that of Lemma \ref{hhff=01}. The main difference here  is that the set $\Pi_j$ defined in \eqref{21=huyPAI01} depends on $\tau$, but we require that the constant $\beta_0$ in \eqref{S5^326y-05} cannot depend on $\tau$. For the sake of completeness, we also include the  proof.

For any $z\in \mathbb{R}^n \backslash (\Pi_j \cup \Sigma_j)$  and $\lambda_0 \leq \lambda \leq \lambda_1 $,  we have $|z| \geq \frac{\tau}{2} u(x_j)^{\frac{2}{n-2\sigma}}$ for all large $j$ and thus   
$$
w_j^\lambda(z) \leq \left( \frac{\lambda}{|z|} \right)^{n-2\sigma} \max_{B_{\lambda_1}} w_j \leq C \tau^{-(n-2\sigma)} u(x_j)^{-2},  
$$
where $C>0$ depends  only on $n$ and $\sigma$.   On the other hand,  by the equation \eqref{Int} we have
\begin{equation}\label{S5^326Esi-301}
u(x) \geq 4^{2\sigma-n} \int_{B_2} K(y) u(y)^{\frac{n+2\sigma}{n-2\sigma}} dy=:A_0 > 0 ~~~~~ \textmd{for} ~ \textmd{all}  ~ x\in B_2 \backslash \Sigma,
\end{equation} 
and by the definition of $w_j$, we obtain
\begin{equation}\label{S5^326y-03}
w_j(y) \geq \frac{A_0}{u(x_j)} ~~~~~~ \textmd{for} ~ y \in \Omega_j \backslash \Pi_j. 
\end{equation}
Therefore,  for any $\tau \in (0, 1/4)$ there exists $j_0$  such that for all $ j\geq j_0$ we have   
$$
w_j(z)^{\frac{n+2\sigma}{n-2\sigma}}  -  w_j^\lambda(z)^{\frac{n+2\sigma}{n-2\sigma}} \geq \frac{1}{2} w_j(z)^{\frac{n+2\sigma}{n-2\sigma}} ~~~~~~ \textmd{in} ~ \Omega_j \backslash \Pi_j. 
$$
As in Lemma \ref{hhff=01}, we still have for $\lambda_0 \leq \lambda  \leq |y| \leq \lambda_1 +1$ and $ |z| \geq \lambda_1+2  $ that 
\begin{equation}\label{S5^326K011}
\frac{\delta_1}{|y -z|^{n-2\sigma}} (|y| - \lambda) \leq G(0,\lambda; y, z) \leq  \frac{\delta_2}{|y -z|^{n-2\sigma}} (|y| - \lambda)
\end{equation}
 where $\delta_1, \delta_2>0$  depend only on $n$ and $\sigma$.     Moreover, we can verify that for $|y|\geq \lambda_1 +1$ and $|z| \geq  \lambda_1 +2$,
\begin{equation}\label{K022=06dh9}
\frac{\delta_3}{|y -z|^{n-2\sigma}}  \leq G(0,\lambda; y, z) \leq  \frac{1}{|y -z|^{n-2\sigma}} 
\end{equation}
for some $\delta_3\in (0, 1)$ depending only on $n$ and $\sigma$.

Denote $\rho_j:=u(x_j)^{\frac{2}{n-2\sigma}} $ and  $A_1:=\min_{B_2} K >0$.   Then,   for $\lambda \leq |y| \leq \lambda_1 +1$ and for  large $j$,   we have
$$
\aligned
J_\lambda(y) & \geq \frac{A_1}{2}\left( \frac{A_0}{u(x_j)} \right)^{\frac{n+2\sigma}{n-2\sigma}} \int_{\Omega_j \backslash \Pi_j}  \frac{\delta_1}{|y -z|^{n-2\sigma}} (|y| - \lambda)  dz \\
&~~~~~ - C \int_{\Omega_j^c} \frac{\delta_2}{|y -z|^{n-2\sigma}} (|y| - \lambda)  \left(  \frac{\lambda}{|z|}\right)^{n+2\sigma} dz \\
& \geq \frac{A_1A_0^{\frac{n+2\sigma}{n-2\sigma}}}{2}   (|y| - \lambda)  u(x_j)^{-\frac{n+2\sigma}{n-2\sigma}}  \int_{\{\frac{3\tau}{2} \rho_j \leq  |z| \leq \frac{7}{4}\rho_j \} \backslash \Sigma_j} \frac{1}{|y -z|^{n-2\sigma}}  dz \\
& ~~~~~ -C (|y| - \lambda) \int_{\{ |z| \geq \frac{7}{4}\rho_j \} \cup \Sigma_j}  \frac{1}{|y -z|^{n-2\sigma}} \left( \frac{1}{|z|}\right)^{n+2\sigma} dz\\
& \geq C (|y| - \lambda) u(x_j)^{-1} - C (|y| - \lambda) u(x_j)^{-\frac{2n}{n-2\sigma}} \\
& \geq \beta_0 (|y| - \lambda) u(x_j)^{-1}
\endaligned
$$
for some constant $\beta_0>0$  independent of  $\tau$, where we have used $\mathcal{L}^n(\Sigma_j) =0$  and $u(x_j)  \to \infty$.  Similarly, for $|y| \geq \lambda_1 +1$ and $y \in \Pi_j$, we have
$$
J_\lambda(y)\geq C  u(x_j)^{-1} - C  u(x_j)^{-\frac{2n}{n-2\sigma}}  \geq \beta_0  u(x_j)^{-1}
$$
for another constant $\beta_0>0$ independent of  $\tau$. Lemma \ref{S5^326=01} is established. 
\end{proof}  

To estimate $\Phi_\lambda$  with  $\lambda \in [1/2, 2]$,  we denote    
\begin{equation}\label{SL5-326=001} 
Q_\lambda(z) = \left( K_j(z^\lambda) - K_j(z) \right) w_j^\lambda(z)^{\frac{n+2\sigma}{n-2\sigma}} ~~~~~ \textmd{for} ~ z \in\Omega_j \backslash B_\lambda. 
\end{equation}
Since $w_j(y)$ converges  in $C_{loc}^2(\mathbb{R}^n)$ to $U_0(y)=(1+ |y |^2)^{-\frac{n-2\sigma}{2}}$, we have for large $j$ that 
\begin{equation}\label{SL5-326=003}
0 \leq w_j^\lambda(z) \leq 2  \left(\frac{\lambda}{|z|} \right)^{n-2\sigma} \left( \frac{1}{1 + \left| z^\lambda\right|^2}\right)^{\frac{n-2\sigma}{2}} \leq C |z|^{2\sigma-n}  ~~~~~~ \textmd{for} ~ z \in \Omega_j \backslash B_\lambda, 
\end{equation}
where $C>0$ depends only on $n$ and $\sigma$.    In order to describe  the behavior of $Q_\lambda$  on $\Omega_j \backslash B_\lambda$ we define the following set 
\begin{equation}\label{SL5-326=004}
\mathcal{P}_j:=  \left\{y\in \mathbb{R}^n :   x_j + M_j^{-\frac{2}{n-2\sigma}} y  \in  B_{2\tau}  \backslash \Sigma \right\} \subset \Omega_j. 
\end{equation}  

We next complete the proof of  Theorem \ref{IEthm01} under the assumption (K2) in three steps. 

\vskip0.10in  

{\it Step 3. We prove Theorem \ref{IEthm01} under the assumption (K2) when  $2\sigma < n < 2\sigma +2$. } By  the  assumption  on  $K$ and \eqref{SL5-326=003} we have
\begin{equation}\label{SL5-327=001}
|Q_\lambda(z)| \leq
\begin{cases}
 c(\tau) M_j^{-1}  (|z| - \lambda)^\alpha  |z|^{-2\sigma-n}~~~~~~~~ & \textmd{for}  ~ z \in \mathcal{P}_j \backslash B_\lambda, \\
 C M_j^{-1}  |z|^{\alpha-2\sigma-n} ~~~~~~~~ &  \textmd{for}  ~ z \in\Omega_j \backslash \mathcal{P}_j, 
\end{cases} 
\end{equation} 
where $ c(\cdot)$ is a nonnegative continuous function with $c(0)=0$.   Recall also that $\alpha=\frac{n-2\sigma}{2}$.  By the definition of $Q_\lambda$, $\Phi_\lambda$ given in \eqref{SL4-31=huy2=02} can be written as 
\begin{equation}\label{SL5-326=002}
\Phi_\lambda(y)= \int_{\Omega_j \backslash B_\lambda} G(0, \lambda; y, z) Q_\lambda(z) dz. 
\end{equation} 
We  have  the  following  estimates  for $\Phi_\lambda$.  
\begin{lemma}\label{LemmaS5-327=01} 
For any $\lambda \in [1/2, 2]$,  we have 
\begin{equation}\label{SL5-327=003}
|\Phi_\lambda (y)| \leq 
\begin{cases}
\Big( c(\tau)  M_j^{-1} +  \tau^{-\frac{n+2\sigma}{2}} M_j^{-1 -\frac{n + 2\sigma}{n-2\sigma}} \Big)  (|y| - \lambda) ~ & \textmd{for} ~ y \in B_{4\lambda} \backslash B_\lambda, \\
c(\tau)  M_j^{-1}  (|y|^{\alpha -n} + |y|^{2\sigma-n} \log |y|) + C  \tau^{-\frac{n+2\sigma}{2}} M_j^{-1-\frac{n + 2\sigma}{n-2\sigma}}  ~ & \textmd{for} ~ y \in  \Pi_j  \backslash B_{4\lambda},  
\end{cases}
\end{equation} 
where $ c(\cdot)$ is a nonnegative continuous function satisfying  $c(0)=0$ and $C>0$ is a constant  depending only on $n$, $\sigma$ and $K$.  
\end{lemma} 
\begin{proof}
It follows from \eqref{SL5-327=001} that 
\begin{equation}\label{SL5-327=004}
\aligned
|\Phi_\lambda (y)| & \leq  c(\tau) M_j^{-1} \int_{\mathcal{P}_j \backslash B_\lambda} G(0, \lambda; y, z) (|z| - \lambda)^\alpha |z|^{-2\sigma-n} dz \\
& ~~~~  + C M_j^{-1} \int_{\Omega_j \backslash \mathcal{P}_j} G(0, \lambda; y, z)  |z|^{\alpha-2\sigma-n}   dz. 
\endaligned
\end{equation}
We consider the following two cases separately. 

\vskip0.10in  

{\it Case 1: } $\lambda <  |y| < 4\lambda$.   As in the proof of Case 1 of Lemma \ref{LS4-312=01}, we consider three sets 
$$
\mathcal{A}_1 = \{z \in \mathcal{P}_j \backslash B_\lambda: |z -y|< (|y| - \lambda)/3 \}, 
$$
$$
\mathcal{A}_2 = \{ z \in \mathcal{P}_j \backslash B_\lambda: |z - y| \geq (|y| - \lambda)/3 ~ \textmd{and} ~ |z|\leq 8\lambda \}, 
$$
$$
\mathcal{A}_3 = \{z\in \mathcal{P}_j \backslash B_\lambda: |z| \geq  8\lambda \}. 
$$
Then,  from Lemma \ref{InK-01}  we have
\begin{equation}\label{SL5-327=005}
G(0, \lambda; y, z) \leq
\begin{cases}
C |y -z|^{2\sigma -n}  ~~~~~~ & \textmd{if} ~z\in \mathcal{A}_1, \\
C \frac{(|y| - \lambda)(|z|^2 - \lambda^2)}{\lambda|y -z|^{n-2\sigma+2}} \leq C \frac{(|y| - \lambda)(|z| - \lambda)}{|y -z|^{n-2\sigma+2}}~~~~~~ & \textmd{if} ~z\in \mathcal{A}_2,\\
C \frac{(|y| - \lambda)(|z|^2 - \lambda^2)}{\lambda|y -z|^{n-2\sigma+2}} \leq C (|y| - \lambda) |z|^{2\sigma-n}~~~~~~ & \textmd{if} ~z\in \mathcal{A}_3. 
\end{cases}
\end{equation}
Thus,
$$
\aligned
\int_{\mathcal{A}_1} G(0, \lambda; y, z) (|z| - \lambda)^\alpha |z|^{-2\sigma-n} dz & \leq C  \int_{\mathcal{A}_1} |y -z|^{2\sigma- n} (|z| - \lambda)^\alpha |z|^{-2\sigma-n}  dz \\
& \leq C (|y| - \lambda), 
\endaligned
$$
where we used $|z| - \lambda \leq 4(|y| - \lambda)/3$ for $z \in \mathcal{A}_1$ and $\alpha+2\sigma =\frac{n+2\sigma}{2} >1$.  For the integral over $\mathcal{A}_2$,  we have
$$
\aligned
\int_{\mathcal{A}_2} G(0, \lambda; y, z) (|z| - \lambda)^\alpha |z|^{-2\sigma-n} dz & \leq C  \int_{\mathcal{A}_2} \frac{(|y| - \lambda)(|z| - \lambda)^{1+\alpha}}{|y -z|^{n-2\sigma+2}}  |z|^{-2\sigma-n}  dz \\
& \leq C (|y| - \lambda), 
\endaligned
$$
where we used $|z| - \lambda \leq 4|z -y|$ for $z \in \mathcal{A}_2$ and $\alpha+2\sigma>1$.  For the integral over $\mathcal{A}_3$,  we have 
$$
\aligned
\int_{\mathcal{A}_3} G(0, \lambda; y, z) (|z| - \lambda)^\alpha |z|^{-2\sigma-n} dz & \leq C  (|y| - \lambda) \int_{\mathcal{A}_3} |z|^{\alpha-2n}  dz \\
& \leq C (|y| - \lambda). 
\endaligned
$$
These estimates give that 
\begin{equation}\label{SL5-327=006} 
\int_{\mathcal{P}_j \backslash B_\lambda} G(0, \lambda; y, z) (|z| - \lambda)^\alpha |z|^{-2\sigma-n} dz \leq C (|y| - \lambda)
\end{equation} 
for some constant $C>0$ depending only on $n$ and $\sigma$.  On the other hand,  by Lemma \ref{InK-01} we obtain for large $j$ that  
\begin{equation}\label{SL5-327=007} 
\aligned
\int_{\Omega_j \backslash \mathcal{P}_j} G(0, \lambda; y, z)  |z|^{\alpha-2\sigma-n} dz  & \leq C (|y| - \lambda)  \int_{|z| \geq \tau M_j^{\frac{2}{n-2\sigma}}}  |z|^{\alpha-2n} dz \\
&  \leq C  \tau^{-\frac{n+2\sigma}{2}} M_j^{-\frac{n + 2\sigma}{n-2\sigma}} (|y| - \lambda). 
\endaligned
\end{equation} 
This together with \eqref{SL5-327=004} and \eqref{SL5-327=006} implies that Lemma \ref{LemmaS5-327=01} holds for $y \in B_{4\lambda} \backslash B_\lambda$. 

\vskip0.10in  

{\it Case 2:  $y \in  \Pi_j  \backslash B_{4\lambda}$.}  As in the proof of Case 2 of Lemma \ref{LS4-312=01}, we define four sets: 
$$
\mathcal{S}_1 = \{ z \in \mathcal{P}_j  \backslash B_\lambda: |z| < |y|/2 \}, 
$$
$$
\mathcal{S}_2 = \{ z \in \mathcal{P}_j  \backslash B_\lambda: |z| > 2|y| \}, 
$$
$$
\mathcal{S}_3 = \{ z \in \mathcal{P}_j  \backslash B_\lambda: |z-y| \leq  |y|/2 \}, 
$$
$$
\mathcal{S}_4 = \{ z \in \mathcal{P}_j  \backslash B_\lambda: |z-y| \geq  |y|/2  ~ \textmd{and} ~ |y|/2 \leq |z| \leq  2|y|\}. 
$$
Notice that $G(0, \lambda; y, z)  \leq C|y -z|^{2\sigma -n}$.  By direct calculations we can get the following estimates: 
$$ 
\int_{\mathcal{S}_1} |y -z|^{2\sigma-n}  |z|^{\alpha -2\sigma-n} dz  
\leq 
\begin{cases}
C |y|^{\alpha -n} ~~~ & \textmd{if} ~ \alpha > 2\sigma,\\
C |y|^{2\sigma -n}\log |y| ~~~ & \textmd{if} ~ \alpha = 2\sigma,\\
C |y|^{2\sigma -n} ~~~  & \textmd{if} ~ \alpha < 2\sigma, 
\end{cases}  
$$ 
$$
\int_{\mathcal{S}_2} |y -z|^{2\sigma-n}  |z|^{\alpha -2\sigma-n} dz \leq C |y|^{\alpha -n}, 
$$
$$
\int_{\mathcal{S}_3} |y -z|^{2\sigma-n}  |z|^{\alpha -2\sigma-n} dz \leq C |y|^{\alpha -2\sigma-n} \int_{\mathcal{S}_3} |y -z|^{2\sigma-n} dz \leq C |y|^{\alpha -n} 
$$
and 
$$
\int_{\mathcal{S}_4} |y -z|^{2\sigma-n}  |z|^{\alpha -2\sigma-n} dz\leq C |y|^{\alpha -n}. 
$$
Therefore, we have 
\begin{equation}\label{SL5-327=008} 
\int_{\mathcal{P}_j \backslash B_\lambda} G(0, \lambda; y, z) (|z| - \lambda)^\alpha |z|^{-2\sigma-n} dz \leq C (|y|^{\alpha -n} + |y|^{2\sigma-n} \log |y|). 
\end{equation}
On the other hand, for  $y\in \Pi_j \backslash B_\lambda$ we have $|y| \leq \frac{5\tau}{4} M_j^{\frac{2}{n-2\sigma}}$ and for  $z\in \Omega_j \backslash \mathcal{P}_j$  we have $|z| \geq \frac{7\tau}{4} M_j^{\frac{2}{n-2\sigma}} $ when $j$ is large. Thus,   
\begin{equation}\label{SL5-327=009} 
\int_{\Omega_j \backslash \mathcal{P}_j} G(0, \lambda; y, z)  |z|^{\alpha-2\sigma-n} dz   \leq C \int_{|z| \geq \tau M_j^{\frac{2}{n-2\sigma}}}  |z|^{\alpha-2n} dz \leq C  \tau^{-\frac{n+2\sigma}{2}} M_j^{-\frac{n + 2\sigma}{n-2\sigma}}. 
\end{equation} 
This together with \eqref{SL5-327=004} and \eqref{SL5-327=008} implies that Lemma \ref{LemmaS5-327=01} holds for $y \in \Pi_j \backslash B_{4\lambda}$.   Thus the proof of Lemma \ref{LemmaS5-327=01} is finished.  
\end{proof}

For $\lambda \in [1/2, 2]$, we construct  $H_\lambda$  as 
\begin{equation}\label{SL5-327=010} 
H_\lambda(y)  = - \varepsilon_1 M_j^{-1} (\lambda^{2\sigma-n} - |y|^{2\sigma -n} ) + \Phi_\lambda(y) ~~~~~ \textmd{for} ~ y \in \Pi_j \backslash \overline{B}_\lambda. 
\end{equation} 
Here  $\varepsilon_1$ is chosen so that $0< \varepsilon_1 \ll \varepsilon_0$ and 
\begin{equation}\label{SL5-328=01} 
- \varepsilon_1 M_j^{-1} (\lambda^{2\sigma-n} - |y|^{2\sigma -n} ) + J_\lambda(y) \geq \frac{1}{2}J_\lambda(y)  ~~~~~  \textmd{for} ~ y\in \Pi_j \backslash \overline{B}_\lambda, 
\end{equation}
where $\varepsilon_0$ is defined  in Lemma \ref{gmhuL5SAn-01}.  It follows from Lemma  \ref{S5^326=01} that \eqref{SL5-328=01} is true by taking sufficiently small $\varepsilon_1$.  Notice that when $\tau$ is small, $c(\tau)$ is correspondingly small.  Hence, by Lemma \ref{LemmaS5-327=01} we can choose $\tau$ to be small enough such that for any $\lambda \in [1/2, 2]$
\begin{equation}\label{SL5-328=02} 
H_\lambda <0  ~~~~~~~ \textmd{in} ~   \Pi_j \backslash \overline{B}_\lambda
\end{equation} 
when $j$ is sufficiently large.  Moreover, by \eqref{SL4-312=04=huy} and \eqref{SL5-328=01} we know that $\varphi_\lambda + H_\lambda$   
satisfies the following integral inequality 
\begin{equation}\label{SL5-328=03}
\varphi_\lambda(y) +  H_\lambda(y) \geq \int_{\Pi_j \backslash B_\lambda} G(0, \lambda; y, z) b_\lambda(z) \varphi_\lambda(z) dz + \frac{1}{2}J_\lambda(y),  ~~~~~ y \in \Pi_j \backslash \overline{B}_\lambda 
\end{equation}
with  $\lambda \in [1/2,  2]$.  

Now we apply the method of moving spheres to $\varphi_\lambda + H_\lambda$ for $\lambda \in [\lambda_0, \lambda_1]$ where $\lambda_0=1/2$ and $\lambda_1=2$.  First, when $\lambda=\lambda_0$, by Lemmas \ref{gmhuL5SAn-01}  and \ref{LemmaS5-327=01} we have $\varphi_{\lambda_0} + H_{\lambda_0} \geq 0$ in $\Pi_j \backslash \overline{B}_{\lambda_0}$ for all large $j$.   Define
$$
\bar{\lambda}:= \sup\{ \mu \geq \lambda_0 ~ | ~ \varphi_{\lambda}(y) +  H_{\lambda}(y) \geq 0, ~ \forall ~ y \in \Pi_j \backslash \overline{B}_{\lambda}, ~ \forall ~ \lambda_0 \leq  \lambda \leq  \mu \}. 
$$
Then,  $\bar{\lambda}$ is well defined and $\bar{\lambda}  \geq \lambda_0$ for all large $j$. Furthermore, by \eqref{SL5-328=02} and Lemma \ref{gmhuL5SAn-01} we see that  $\bar{\lambda} <  \lambda_1$ for all large $j$. Then, as in Step 2 of Section \ref{S3000},  applying the moving sphere method one can derive a contradiction to the definition of $\bar{\lambda}$. Step 3 is established. 

\vskip0.10in  

{\it Step 4. We prove Theorem \ref{IEthm01} under the assumption (K2) when $n  =  2\sigma +2$. }  In  this  case,  $K \in C^1(B_2)$ and thus  by Step 2 we have $\nabla K(0)=0$.   For simplicity,  we may assume $K\in C^1(\overline{B}_2)$.   Based on  these properties of $K$ we have for $\lambda \in [1/2, 2]$ that
\begin{equation}\label{SL5-328=04}
|K(x_j + M_j^{-\frac{2}{n-2\sigma}}z^\lambda) - K(x_j + M_j^{-\frac{2}{n-2\sigma}}z) | \leq
\begin{cases}
 c(\tau) M_j^{-1}  (|z| - \lambda)  & \textmd{for}  ~ z \in \mathcal{P}_j \backslash B_\lambda, \\
 C M_j^{-1}  |z|  &  \textmd{for}  ~ z \in\Omega_j \backslash \mathcal{P}_j. 
\end{cases} 
\end{equation} 
Thus, by  \eqref{SL5-326=003} we get 
\begin{equation}\label{SL5-328=05}
|Q_\lambda(z)| \leq
\begin{cases}
 c(\tau) M_j^{-1}  (|z| - \lambda)  |z|^{-2\sigma-n}~~~~~~~~ & \textmd{for}  ~ z \in \mathcal{P}_j \backslash B_\lambda, \\
 C M_j^{-1}  |z|^{1-2\sigma-n} ~~~~~~~~ &  \textmd{for}  ~ z \in\Omega_j \backslash \mathcal{P}_j, 
\end{cases} 
\end{equation} 
where $ c(\cdot)$ is a nonnegative continuous function with $c(0)=0$.  This indicates that Lemma \ref{LemmaS5-327=01} is still true. 
The rest of the proof of this step is the same as that of Step 3.   Step 4 is established. 

\vskip0.10in 

{\it Step 5. We prove Theorem \ref{IEthm01} under the assumption (K2) when $n >  2\sigma +2$.}    Under the current case, our assumption  is $K\in \mathcal{C}^{\alpha}(B_2)$ with $\alpha=\frac{n-2\sigma}{2}>1$.   In particular,  this implies that $K \in C^1(B_2)$ and,  by Step 2 we obtain $\nabla K(0)=0$.  Thus $0$ is a critical point of $K$,  for $n \geq 2\sigma +4 $, by the assumption on $K$ there exists a neighborhood $N$ of $0$ such that \eqref{Khy2-1} holds.  Without loss of generality we may assume $N=B_2$, since otherwise we can consider the integral equation \eqref{Int} on $N$ and then make a rescaling argument.  Namely,  for $n \geq 2\sigma +4 $ we assume that $K$ satisfies 
\begin{equation}\label{Khy2-10-086}
|\nabla^i K(y)| \leq c(|y|) |\nabla K(y)|^{\frac{\alpha-i}{\alpha-1}}, ~~~ 2\leq i \leq [\alpha], ~~ y \in B_2 
\end{equation}
for some  nonnegative continuous function $c(\cdot)$ with $c(0)=0$. 

Let $Z_j=|\nabla K(x_j)|$,  then  $\lim_{j \to \infty} Z_j=|\nabla K(0)|=0$.  Moreover,  we have the following result. 
\begin{lemma}\label{StepL5-401=01}
\begin{equation}\label{SL5-329=01}
Z_j^{\frac{1}{\alpha-1}}  M_j^{\frac{2}{n-2\sigma}} \to 0.  
\end{equation} 
\end{lemma} 
\begin{proof}
We first prove that there exists a constant $C>0$ such that
\begin{equation}\label{Gt7ejw301h=01}
Z_j^{\frac{1}{\alpha-1}}  M_j^{\frac{2}{n-2\sigma}}\leq C ~~~~~~ \textmd{for} ~  \textmd{all}  ~ j.  
\end{equation} 
If not, then up to a subsequence 
\begin{equation}\label{SL5-331=01}
 Z_j^{\frac{1}{\alpha-1}}  M_j^{\frac{2}{n-2\sigma}} \to +\infty.  
\end{equation} 
The following proof is similar to that of Step 2  in Section \ref{S3000}, we  mainly show the modifications that need to be made.   Without loss of generality, we may assume
\begin{equation}\label{SLg75-401=01}
\lim_{j \to \infty} \frac{\nabla K(x_j)}{|\nabla K(x_j)|} = e = (1, 0, \dots, 0). 
\end{equation} 
As in Step 2 of Section \ref{S3000},   we define  $w_j$ and $h_j$ as in \eqref{21SS3-01} and define  $\Omega_j$  as in \eqref{hhmm-9e7=5}.  Then $w_j$ still satisfies \eqref{21SS3-044} on  $\Omega_j $ where $K_j$ is defined as in \eqref{21SS3-033}. Moreover, on every compact subset of $\mathbb{R}^n$,   $w_j$ converges in $C^2$ norm to $U_1=U_0(\cdot - Re)$ by  Step 1.  We also extend $w_j$ to be identically $0$ outside $\Omega_j$ and $K$ to be identically $0$ outside $B_2$. 

Define $w_j^\lambda$ and $h_j^\lambda$ as in \eqref{21SS3-022}, then $w_j^\lambda$ satisfies \eqref{21SS3-055}.  It is clear that Lemma \ref{LAn-01} still holds.  Let $\varphi_\lambda(y) =w_j(y)  -w_j^\lambda(y)$ with $\lambda \in [R-2, R+2]$.  Recall that  $R>10$ is a large positive constant to be determined later.   Then $\varphi_\lambda$ satisfies the following integral inequality for  large $j$,  
\begin{equation}\label{SLg75-401=02}
\varphi_\lambda(y) +  \Phi_\lambda(y) \geq \int_{\Pi_j \backslash B_\lambda} G(0, \lambda; y, z) b_\lambda(z) \varphi_\lambda(z) dz + J_\lambda(y),  ~~~~~ y \in \Pi_j \backslash \overline{B}_\lambda, 
\end{equation}
where $\Pi_j$, $b_\lambda$, $\Phi_\lambda$  and $J_\lambda$ are respectively given as in \eqref{21PAI01}, \eqref{SL4-312=01}, \eqref{SL4-312=02} and \eqref{SL4-312=03}. 

Now  we need to establish some estimates for $\Phi_\lambda$ on $\Pi_j \backslash \overline{B}_\lambda$.  Define a special domain  $D_\lambda$ in $\Omega_j \backslash B_\lambda$ as follows,   
$$
D_\lambda = \{z \in \mathbb{R}^n : \lambda < |z| < 2\lambda, z_1 > 2 |z^{\prime}| ~\textmd{where} ~ z^{\prime}=(z_2, \dots, z_n) \}.
$$ 
Let  $\tilde{l}_j := Z_j^{\frac{1}{\alpha-1}}  M_j^{\frac{2}{n-2\sigma}} $,  then by \eqref{SL5-331=01} we know $\tilde{l}_j \to +\infty$ and so $\tilde{l}_j \gg R$ when $j$ is large.     Under our current assumptions on $K$,  we have

\vskip0.10in

{\it Claim 1: There exists a constant $C_0>0$ independent of $\lambda$ such that for all large $j$,   
\begin{equation}\label{SLg75-401=03}
K_j(z^\lambda) - K_j(z) \leq  -C_0 M_j^{-\frac{2}{n-2\sigma}} |\nabla K(x_j)| (|z| - \lambda), ~~~~~ z \in D_\lambda,
\end{equation}
\begin{equation}\label{SLg75-401=04}
|K_j(z^\lambda) - K_j(z)| \leq C_0 M_j^{-\frac{2}{n-2\sigma}} |\nabla K(x_j)| (|z| - \lambda), ~~~~~ z \in B_{2\tilde{l}_j} \backslash \overline{B}_\lambda,
\end{equation}
\begin{equation}\label{SLg75-401=05}
|K_j(z^\lambda) - K_j(z)| \leq C_0 M_j^{-1} |z|^\alpha, ~~~~~ z \in \Omega_j \backslash B_{2\tilde{l}_j}.
\end{equation}
}
\begin{proof}
(1) For $z \in D_\lambda$ and $2\sigma +2 < n < 2\sigma +4$,   by the mean value theorem and the assumptions on $K$ we obtain that 
$$
\aligned
K_j(z^\lambda) - K_j(z) & = M_j^{-\frac{2}{n-2\sigma}} \nabla K\Big( x_j + M_j^{-\frac{2}{n-2\sigma}}\big( (1-\theta)z^\lambda + \theta z -Re \big)  \Big)  \cdot (z^\lambda  - z) \\
& \leq M_j^{-\frac{2}{n-2\sigma}} \nabla K(x_j)\cdot (z^\lambda -z) + \sup_{0\leq t \leq Z_j^{\frac{1}{\alpha-1}}  } c(t)  M_j^{-\frac{2}{n-2\sigma}} |\nabla K(x_j)|  |z^\lambda -z| \\
& \leq -C_0 M_j^{-\frac{2}{n-2\sigma}} |\nabla K(x_j)|   (|z| - \lambda)
\endaligned
$$
for all large $j$,  where $\theta \in (0, 1)$ and $C_0>0$ is a constant.  

For $z \in D_\lambda$ and $n \geq  2\sigma +4$,   by the mean value theorem and the assumptions on $K$ we obtain that 
$$
\aligned
K_j(z^\lambda) - K_j(z) & = M_j^{-\frac{2}{n-2\sigma}} \nabla K\Big( x_j + M_j^{-\frac{2}{n-2\sigma}}\big( (1-\theta)z^\lambda + \theta z -Re \big)  \Big)  \cdot (z^\lambda  - z) \\
& \leq M_j^{-\frac{2}{n-2\sigma}} \nabla K(x_j)\cdot (z^\lambda -z) + C_1 \sum_{i=2}^{[\alpha]} M_j^{-\frac{2i}{n-2\sigma}}  |\nabla^i  K(x_j)|  |z|^{i-1}|z^\lambda - z| \\
& ~~~~~  +  \sup_{0\leq t \leq |x_j| + Z_j^{\frac{1}{\alpha-1}}  } c(t)  M_j^{-\frac{2\alpha}{n-2\sigma}} |z|^{\alpha-1}   |z^\lambda -z| \\
& \leq  -C_1 M_j^{-\frac{2}{n-2\sigma}} |\nabla K(x_j)|   (|z| - \lambda) \\
& ~~~~~  + c(|x_j|) \sum_{i=2}^{[\alpha]} M_j^{-\frac{2i}{n-2\sigma}}  |\nabla K(x_j)|^{\frac{\alpha-i}{\alpha-1}}   |z|^{i-1}|z^\lambda - z|  \\
& \leq -C_0 M_j^{-\frac{2}{n-2\sigma}} |\nabla K(x_j)|   (|z| - \lambda)
\endaligned
$$
for all large $j$,  where $\theta \in (0, 1)$,  $C_0, C_1, C_2>0$ are different  constants. 

(2) When $z \in B_{2\tilde{l}_j} \backslash \overline{B}_\lambda$, the proof is similar to that of (1),  and we only need to notice that $M_j^{-\frac{2}{n-2\sigma}} |\nabla K(x_j)\cdot (z^\lambda -z)|\leq C M_j^{-\frac{2}{n-2\sigma}} |\nabla K(x_j)|(|z| - \lambda)$. 

(3) For $z \in \Omega_j \backslash B_{2\tilde{l}_j}$ and $2\sigma +2 < n < 2\sigma +4$,   by the mean value theorem and the assumptions on $K$ we obtain that 
$$
\aligned
|K_j(z^\lambda) - K_j(z)| &  \leq M_j^{-\frac{2}{n-2\sigma}} |\nabla K(x_j)|(|z| -\lambda) + C   M_j^{-\frac{2\alpha}{n-2\sigma}} |z|^{\alpha-1} (|z| - \lambda)  \\
& \leq C_0 M_j^{-1} |z|^{\alpha} 
\endaligned
$$
for some constant  $C_0>0$. 

For $z \in \Omega_j \backslash B_{2\tilde{l}_j}$ and $n \geq  2\sigma +4$,   by the mean value theorem and the assumptions on $K$ we obtain that 
$$
\aligned
|K_j(z^\lambda) - K_j(z)| & \leq M_j^{-\frac{2}{n-2\sigma}} |\nabla K(x_j)| (|z| - \lambda) + C \sum_{i=2}^{[\alpha]} M_j^{-\frac{2i}{n-2\sigma}}  |\nabla^i  K(x_j)|  |z|^{i-1}|z^\lambda - z| \\
& ~~~~~  +  C  M_j^{-\frac{2\alpha}{n-2\sigma}} |z|^{\alpha-1}   |z^\lambda -z| \\
& \leq C_0 M_j^{-1} |z|^\alpha 
\endaligned
$$
for some constant $C_0$. Thus Claim 1 is proved. 
\end{proof}

It is also clear that Lemma \ref{Le-hm-074} still holds.  If we let 
$$
Q_\lambda(z) = \left( K_j(z^\lambda) - K_j(z) \right) w_j^\lambda(z)^{\frac{n+2\sigma}{n-2\sigma}} ~~~~~ \textmd{for} ~ z \in\Omega_j \backslash B_\lambda, 
$$
then by Claim 1 above  and Lemma \ref{Le-hm-074} we have
\begin{equation}\label{SLg75-401=06}
Q_\lambda(z) \leq -C M_j^{-\frac{2}{n-2\sigma}} |\nabla K(x_j)| (|z| - \lambda) \left( \frac{1}{1 + (z_1 - \lambda)^2 + |z^{\prime}|^2}\right)^{\frac{n + 2\sigma}{2}} ~~~~~~ \textmd{for} ~ z\in D_\lambda
\end{equation}
and 
\begin{equation}\label{SLg75-401=07}
|Q_\lambda(z)| \leq 
\begin{cases}
C  M_j^{-\frac{2}{n-2\sigma}} |\nabla K(x_j)| (|z| - \lambda)  ~~~~~~ & \textmd{for} ~ z\in D_\lambda, \\
C M_j^{-\frac{2}{n-2\sigma}} |\nabla K(x_j)|  (|z| - \lambda) |z|^{-2\sigma-n} ~~~~~~ & \textmd{for} ~ z \in B_{2\tilde{l}_j} \backslash (B_\lambda \cup D_\lambda), \\
C M_j^{-1}  |z|^{\alpha-2\sigma-n}~~~~~~ & \textmd{for} ~ z\in \Omega_j \backslash B_{2\tilde{l}_j}. 
\end{cases} 
\end{equation}
We split $\Phi_\lambda(y)$ into two parts:   
\begin{equation}\label{SLg75-401=08}
\aligned
\Phi_\lambda(y) & = \int_{B_{2\tilde{l}_j} \backslash B_\lambda} G(0, \lambda; y, z) Q_\lambda(z) dz + \int_{\Omega_j \backslash B_{2\tilde{l}_j}} G(0, \lambda; y, z) Q_\lambda(z) dz \\
& =: \Phi_{1,\lambda}(y) + \Phi_{2, \lambda}(y). 
\endaligned
\end{equation} 

By the estimates of $Q_\lambda(z)$ on $B_{2\tilde{l}_j} \backslash B_\lambda$,  one can see easily that the estimates of $\Phi_{1,\lambda}$ are very similar to  those  of $\Phi_\lambda$ in Lemma \ref{LS4-312=01} of Step 2 of Section \ref{S3000}.  Hence we have the following:

\vskip0.10in 

{\it Claim 2: 
There  exists  a  constant $C>0$ independent of  $\lambda$ such that for all large $j$, 
$$ 
\Phi_{1,\lambda} (y) \leq 
\begin{cases}
- C M_j^{-\frac{2}{n-2\sigma}} |\nabla K(x_j)| (|y| - \lambda) \lambda^{-n} \log \lambda  ~~~~~~ & \textmd{for} ~ y \in B_{4\lambda} \backslash \overline{B}_\lambda, \\
-C M_j^{-\frac{2}{n-2\sigma}} |\nabla K(x_j)| |y|^{2\sigma -n} \lambda^{1-2\sigma} \log \lambda ~~~~~~ & \textmd{for} ~ y \in \Pi_j \backslash \overline{B}_{4\lambda}
\end{cases}
$$
and
$$ 
|\Phi_{1,\lambda} (y)| \leq 
\begin{cases}
C M_j^{-\frac{2}{n-2\sigma}} |\nabla K(x_j)| (|y| - \lambda) \lambda^{2\sigma}  ~~~~~~ & \textmd{for} ~ y \in B_{4\lambda} \backslash \overline{B}_\lambda, \\
C M_j^{-\frac{2}{n-2\sigma}} |\nabla K(x_j)|  |y|^{2\sigma -n} \lambda^{n+1} ~~~~~~ & \textmd{for} ~ y \in  \Pi_j  \backslash \overline{B}_{4\lambda}. 
\end{cases}
$$
}
\begin{proof}
We only need to replace $M_j^{-\frac{2}{n-2\sigma}}$ and $\Omega_j \backslash B_\lambda$ in Lemma \ref{LS4-312=01}   by  $M_j^{-\frac{2}{n-2\sigma}} |\nabla K(x_j)|$ and $B_{2\tilde{l}_j} \backslash B_\lambda$, respectively.  
\end{proof}

Recall that under the current assumptions we have $1< \alpha=(n-2\sigma)/2 <n$.  For the estimates of  $\Phi_{2,\lambda}$,  by an argument similar to Lemma \ref{Lg_lammfb=01} of Section \ref{S3000} we have

\vskip0.10in

{\it Claim 3:  There  exists  a  constant $C>0$ independent of  $\lambda$ such that for all large $j$,
$$ 
|\Phi_{2,\lambda} (y)| \leq 
\begin{cases}
 C M_j^{-1} {\tilde{l}_j}^{\alpha-n} (|y| - \lambda)\lambda^{-1} ~~~~~~ & \textmd{for} ~ y \in B_{4\lambda} \backslash \overline{B}_\lambda, \\
C M_j^{-1} {\tilde{l}_j}^{\alpha-n}~~~~~~ & \textmd{for} ~ y \in \Pi_j  \backslash B_{4\lambda}, ~~ \alpha \ne 2 \sigma \\
C M_j^{-1} {\tilde{l}_j}^{2\sigma-n} \log \tilde{l}_j  ~~~~~~ & \textmd{for} ~ y \in \Pi_j \backslash B_{4\lambda},  ~~ \alpha =2 \sigma. \\
\end{cases}  
$$
}

\vskip0.10in

It is clear that  Lemma \ref{LAn-01} and Lemma \ref{hhff=01} still hold.  We construct a function $H_\lambda$ which is non-positive on $\Pi_j \backslash \overline{B}_\lambda$.  Let  
\begin{equation}\label{SLg75-401=09} 
H_\lambda(y)  = - \varepsilon_1 M_j^{-1} (\lambda^{2\sigma-n} - |y|^{2\sigma -n} ) + \Phi_\lambda(y), ~~~~~ y \in \Pi_j \backslash \overline{B}_\lambda
\end{equation} 
for some small $\varepsilon_1 \in (0, \varepsilon_0/4)$ where $\varepsilon_0$ is defined in Lemma \ref{LAn-01}.  By  Lemma  \ref{hhff=01} we can choose $\varepsilon_1>0$ sufficiently small so  that  for any $ \lambda \in [\lambda_0, \lambda_1]$, 
\begin{equation}\label{SLg75-402=01}
J_\lambda(y) - \varepsilon_1  M_j^{-1} (\lambda^{2\sigma-n} - |y|^{2\sigma -n} ) \geq \frac{1}{2} J_\lambda(y),  ~~~~~ y \in \Pi_j \backslash \overline{B}_\lambda. 
\end{equation}
Here we recall that $\lambda_0=R-2$ and $\lambda_1=R+2$.   Then by Claim 2 and Claim 3 above we have for $\lambda\in [\lambda_0, \lambda_1]$ and  for large $j$, 
\begin{equation}\label{SLg75-402=02} 
H_\lambda(y)  < 0, ~~~~~ y \in \Pi_j \backslash \overline{B}_\lambda.
\end{equation} 
Moreover, it follows from Lemma \ref{LAn-01},  Claim 2 and Claim 3 above that for large $j$, 
$$
\varphi_{\lambda_0}(y) +  H_{\lambda_0}(y) \geq 0,  ~~~~~~~  y \in \Pi_j \backslash \overline{B}_{\lambda_0}. 
$$
By \eqref{SLg75-401=02} and \eqref{SLg75-402=01},  $\varphi_{\lambda} +  H_{\lambda}$ satisfies 
\begin{equation}\label{SLg75-402=03}
\varphi_\lambda(y) +  H_\lambda(y)  \geq \int_{\Pi_j \backslash B_\lambda} G(0, \lambda; y, z) b_\lambda(z) \varphi_\lambda(z) dz  + \frac{1}{2}J_\lambda(y) ~~~~~ \textmd{for} ~ y \in \Pi_j \backslash \overline{B}_\lambda, 
\end{equation}
where $\lambda \in [\lambda_0, \lambda_1]$.  We define
$$
\bar{\lambda}:= \sup\{ \mu \geq \lambda_0 ~ | ~ \varphi_{\lambda}(y) +  H_{\lambda}(y) \geq 0, ~ \forall ~ y \in \Pi_j \backslash \overline{B}_{\lambda}, ~ \forall ~ \lambda_0 \leq  \lambda \leq  \mu \}. 
$$
Then,  $\bar{\lambda}$ is well defined and $\bar{\lambda}  \geq \lambda_0$ for all large $j$. Furthermore, by \eqref{SLg75-402=02} and Lemma \ref{LAn-01} we have $\bar{\lambda} <  \lambda_1$ for all large $j$. Then, as in Step 2 of Section \ref{S3000},  applying the moving sphere method one can derive a contradiction to the definition of $\bar{\lambda}$.   This proves that the sequence $Z_j^{\frac{1}{\alpha-1}}  M_j^{\frac{2}{n-2\sigma}}$  is bounded.

Next we show that $Z_j^{\frac{1}{\alpha-1}}  M_j^{\frac{2}{n-2\sigma}} \to 0$.   If this does not hold, then there exists a constant $\varepsilon_2>0$  such that, after passing to a subsequence if necessary, 
\begin{equation}\label{SLg75-402=04}
Z_j^{\frac{1}{\alpha-1}}  M_j^{\frac{2}{n-2\sigma}} \geq  \varepsilon_2 >0.  
\end{equation} 
The following proof is similar to that of \eqref{Gt7ejw301h=01}, we will use the same notation and give only the corresponding modifications.  
Let $l_0 \gg R $ be a large positive constant which is to be determined after fixing  $R$.   Under our current assumptions on $K$,   similar to  Claim 1  we have    

\vskip0.10in 

{\it Claim 4: There exists a constant $C>0$ independent of both $l_0$  and $\lambda$ such that for all large $j$,   
\begin{equation}\label{SmmLg75-402=01}
K_j(z^\lambda) - K_j(z) \leq  -C M_j^{-1} (|z| - \lambda), ~~~~~ z \in D_\lambda. 
\end{equation}
\begin{equation}\label{SmmLg75-402=02}
|K_j(z^\lambda) - K_j(z)| \leq C M_j^{-1}  (|z| - \lambda), ~~~~~ z \in B_{2l_0} \backslash B_{\lambda}. 
\end{equation}
\begin{equation}\label{SmmLg75-402=03} 
|K_j(z^\lambda) - K_j(z)| \leq C_0 M_j^{-1} |z|^\alpha, ~~~~~ z \in \Omega_j \backslash B_{2l_0}. 
\end{equation}
}

We split $\Phi_\lambda(y)$ into two parts:   
\begin{equation}\label{SmmLg75-402=04}
\aligned
\Phi_\lambda(y) & = \int_{B_{2l_0} \backslash B_\lambda} G(0, \lambda; y, z) Q_\lambda(z) dz + \int_{\Omega_j \backslash B_{2l_0}} G(0, \lambda; y, z) Q_\lambda(z) dz \\
& =: \Phi_{1,\lambda}(y) + \Phi_{2, \lambda}(y). 
\endaligned
\end{equation} 

For the estimates  of $\Phi_{1,\lambda}$,  similar to Claim  2 we have 

\vskip0.10in

{\it Claim 5: 
There  exists  a  constant $C>0$ independent of both $l_0$ and  $\lambda$ such that for all large $j$, 
$$ 
\Phi_{1,\lambda} (y) \leq 
\begin{cases}
- C M_j^{-1}  (|y| - \lambda) \lambda^{-n} \log \lambda  ~~~~~~ & \textmd{for} ~ y \in B_{4\lambda} \backslash \overline{B}_\lambda, \\
-C M_j^{-1} |y|^{2\sigma -n} \lambda^{1-2\sigma} \log \lambda ~~~~~~ & \textmd{for} ~ y \in \Pi_j \backslash \overline{B}_{4\lambda}
\end{cases} 
$$
and
$$ 
|\Phi_{1,\lambda} (y)| \leq 
\begin{cases}
C M_j^{-1}  (|y| - \lambda) \lambda^{2\sigma}  ~~~~~~ & \textmd{for} ~ y \in B_{4\lambda} \backslash \overline{B}_\lambda, \\
C M_j^{-1}  |y|^{2\sigma -n} \lambda^{n+1} ~~~~~~ & \textmd{for} ~ y \in  \Pi_j  \backslash \overline{B}_{4\lambda}. 
\end{cases}
$$
}

For the estimates  of $\Phi_{2,\lambda}$,  similar to Claim  3  we also have 

\vskip0.10in

{\it Claim 6:  There  exists  a  constant $C>0$ independent of  both $l_0$ and $\lambda$ such that for all large $j$,
$$ 
|\Phi_{2,\lambda} (y)| \leq 
\begin{cases}
 C M_j^{-1} {l_0}^{\alpha-n} (|y| - \lambda)\lambda^{-1} ~~~~~~ & \textmd{for} ~ y \in B_{4\lambda} \backslash \overline{B}_\lambda, \\
C M_j^{-1} {l_0}^{\alpha-n}~~~~~~ & \textmd{for} ~ y \in \Pi_j  \backslash B_{4\lambda}, ~~ \alpha \ne 2 \sigma \\
C M_j^{-1} {l_0}^{2\sigma-n} \log l_0 ~~~~~~ & \textmd{for} ~ y \in \Pi_j \backslash B_{4\lambda},  ~~ \alpha =2 \sigma. \\
\end{cases}  
$$
}

\vskip0.10in

It is clear that  Lemma \ref{LAn-01} and Lemma \ref{hhff=01} still hold.  We construct a function $H_\lambda$ which is non-positive on $\Pi_j \backslash \overline{B}_\lambda$.  Let  
\begin{equation}\label{SmmLg75-402=05} 
H_\lambda(y)  = - \varepsilon_1 M_j^{-1} (\lambda^{2\sigma-n} - |y|^{2\sigma -n} ) + \Phi_\lambda(y), ~~~~~ y \in \Pi_j \backslash \overline{B}_\lambda
\end{equation} 
for some small $\varepsilon_1 \in (0, \varepsilon_0/4)$ where $\varepsilon_0$ is defined in Lemma \ref{LAn-01}.  By  Lemma  \ref{hhff=01} we can choose $\varepsilon_1>0$ sufficiently small so  that  for any $ \lambda \in [\lambda_0, \lambda_1]$, 
\begin{equation}\label{SmmLg75-402=06}
J_\lambda(y) - \varepsilon_1  M_j^{-1} (\lambda^{2\sigma-n} - |y|^{2\sigma -n} ) \geq \frac{1}{2} J_\lambda(y),  ~~~~~ y \in \Pi_j \backslash \overline{B}_\lambda. 
\end{equation}
Recall that $\lambda_0=R-2$ and $\lambda_1=R+2$.   Then by Claim 5 and Claim 6 above we can determine $l_0$ to be large enough so that for $\lambda\in [\lambda_0, \lambda_1]$ and  for large $j$, 
\begin{equation}\label{SmmLg75-402=07} 
H_\lambda(y)  < 0, ~~~~~ y \in \Pi_j \backslash \overline{B}_\lambda.
\end{equation} 
Moreover, it follows from Lemma \ref{LAn-01},  Claim 5 and Claim 6 above that for large $j$, 
$$
\varphi_{\lambda_0}(y) +  H_{\lambda_0}(y) \geq 0,  ~~~~~~~  y \in \Pi_j \backslash \overline{B}_{\lambda_0}. 
$$
Now, as in Step 2 of Section \ref{S3000},  the method of moving spheres can be applied  to $\varphi_\lambda + H_\lambda$ for $\lambda \in[\lambda_0, \lambda_1]$ to get a contradiction.  Hence the proof of Lemma \ref{StepL5-401=01} is completed.   
\end{proof}

Once Lemma \ref{StepL5-401=01} holds,  then by the  assumptions  on $K$ we have the following estimates for $|K(x_j + M_j^{-\frac{2}{n-2\sigma}}z^\lambda) - K(x_j + M_j^{-\frac{2}{n-2\sigma}}z) |$ in $\Omega_j \backslash B_\lambda$.  Here  $\lambda \in [1/2, 2]$.   
 
\vskip0.10in 
 
{\it Case 1: $2\sigma +2 < n < 2\sigma +4$.}  Using the mean value theorem and the  assumptions on $K$ yields that for large $j$, 
$$
\aligned
& \big| K(x_j + M_j^{-\frac{2}{n-2\sigma}}z^\lambda) - K(x_j + M_j^{-\frac{2}{n-2\sigma}}z)  \big|  \\
& \leq M_j^{-\frac{2}{n-2\sigma}} |z- z^\lambda|   \Big| \nabla K \Big(x_j + M_j^{-\frac{2}{n-2\sigma}} \big(  (1-\theta) z^\lambda + \theta z  \big) \Big)  \Big|  \\
& \leq \sup_{0\leq t \leq \tilde{M}_j|z|}c(t)   M_j^{-\frac{2\alpha}{n-2\sigma}}  |z|^{\alpha-1}  |z- z^\lambda|      + C M_j^{-\frac{2}{n-2\sigma}} |\nabla K (x_j)|  |z- z^\lambda|   \\    
& \leq
\begin{cases} 
 c(\tau) M_j^{-1}  |z|^{\alpha-1}(|z| - \lambda)  ~~  & \textmd{for}  ~ z \in \mathcal{P}_j \backslash B_\lambda, \\
 C M_j^{-1}  |z|^\alpha   ~~  &  \textmd{for}  ~ z \in\Omega_j \backslash \mathcal{P}_j, 
\end{cases} 
\endaligned
$$
where $\theta \in (0, 1)$, $\tilde{M}_j:= M_j^{-\frac{2}{n-2\sigma}}$  and  Lemma \ref{StepL5-401=01} was used in the last inequality.   

\vskip0.10in 

{\it Case 2:  $n \geq  2\sigma +4$.}  Similar to Case 1, by the mean value theorem,  the  assumptions on $K$ and  Lemma \ref{StepL5-401=01}  we obtain that for large $j$,  
$$
\aligned
& \big| K(x_j + M_j^{-\frac{2}{n-2\sigma}}z^\lambda) - K(x_j + M_j^{-\frac{2}{n-2\sigma}}z)  \big|  \\
& \leq  C \sum_{i=1}^{[\alpha]}  M_j^{-\frac{2i}{n-2\sigma}}  |\nabla^i K(x_j)|  |z|^{i-1}|z-z^\lambda| + \sup_{0\leq t \leq |x_j| +\tilde{M}_j|z|}c(t) M_j^{-\frac{2\alpha}{n-2\sigma}} |z|^{\alpha-1} |z-z^\lambda|    \\
& \leq C \bigg( M_j^{-\frac{2}{n-2\sigma}}  |\nabla K(x_j)|  + c(|x_j|) \sum_{i=2}^{[\alpha]}  M_j^{-\frac{2i}{n-2\sigma}}  |\nabla K(x_j)|^{\frac{\alpha-i}{\alpha-1}}  |z|^{i-1} \bigg)  |z- z^\lambda| \\
&  ~~~~~   + \sup_{0\leq t \leq |x_j| +\tilde{M}_j|z|}c(t) M_j^{-\frac{2\alpha}{n-2\sigma}} |z|^{\alpha-1} |z - z^\lambda|   \\    
& \leq
\begin{cases} 
 c(\tau) M_j^{-1}  |z|^{\alpha-1}(|z| - \lambda)  & \textmd{for}  ~ z \in \mathcal{P}_j \backslash B_\lambda, \\
 C M_j^{-1}  |z|^\alpha   &  \textmd{for}  ~ z \in\Omega_j \backslash \mathcal{P}_j, 
\end{cases} 
\endaligned
$$
where $\tilde{M}_j:= M_j^{-\frac{2}{n-2\sigma}}$.   

Thus,   by combining with \eqref{SL5-326=003} we have for any $n>2\sigma+2$ that 
\begin{equation}\label{SL5-329=04} 
|Q_\lambda(z)| \leq
\begin{cases}
 c(\tau) M_j^{-1}  (|z| - \lambda)  |z|^{\alpha-2\sigma-n-1}~~~~~~~~ & \textmd{for}  ~ z \in \mathcal{P}_j \backslash B_\lambda, \\
 C M_j^{-1}  |z|^{\alpha-2\sigma-n} ~~~~~~~~ &  \textmd{for}  ~ z \in\Omega_j \backslash \mathcal{P}_j, 
\end{cases} 
\end{equation} 
where $ c(\cdot)$ is a nonnegative continuous function with $c(0)=0$.  This indicates that Lemma \ref{LemmaS5-327=01} is still true.   
The rest of the proof of this step is the same as that of Step 3.    Step 5 is established. 
Therefore, the proof of Theorem \ref{IEthm01} under the assumption (K2)  is completely finished.   
\hfill$\square$

%

\noindent{T. Jin and H. Yang}\\
Department of Mathematics, The Hong Kong University of Science and Technology\\
Clear Water Bay, Kowloon, Hong Kong, China \\
E-mail addresses:  tianlingjin@ust.hk (T. Jin) ~~~~~~ mahuiyang@ust.hk (H. Yang)

\end{document}